\pdfoutput=1
\documentclass[11pt, xcolor=table]{article}

\usepackage{graphicx}
\usepackage{tabularx}
\usepackage{threeparttable}
\usepackage{multirow}
\usepackage{url}
\usepackage{amsmath}
\usepackage{comment}
\usepackage{nameref}

\usepackage{placeins}
\usepackage{natbib}[sort]
\usepackage{authblk}
\usepackage{hyperref}
\usepackage{float}
\usepackage{caption}
\usepackage{subcaption}
\usepackage{amssymb}
\usepackage{color}
\usepackage{amsthm}
\usepackage[english]{babel}
\usepackage{bm}
\usepackage{tikz}

\newcommand{\R}{\mathbb{R}}
\newcommand{\set}[1]{\left\{#1\right\}}

\newcommand*\circled[1]{\tikz[baseline=(char.base)]{
            \node[shape=circle,draw,inner sep=1pt] (char) {#1};}}

\newtheorem{theorem}{Theorem}

\theoremstyle{theorem}

\newtheorem{proposition}{Proposition}
\theoremstyle{definition}
\newtheorem{definition}{Definition}
\theoremstyle{lemma}

\title{The effect of protest management strategies on protesting activity and social tension: a mathematical perspective}

\author[1,*]{Wang, W.}
\author[1]{Wessler, T.}
\author[2]{Brantingham, J.}
\author[1]{Rodriguez, N.}

\affil[1]{University of Colorado, Boulder, USA}
\affil[2]{University of California, Los Angeles, USA}
\affil[*]{corresponding author, wuyan.wang@colorado.edu}

\begin{document}
\maketitle
%\tableofcontents
\begin{abstract}

Protest activity, a constitutionally protected right in the United States under the First Amendment, serves as a key tool for individuals with limited individual influence to unite collectively and amplify their impact. Despite its legal recognition, the historical interaction between the government and protesters has been complex. In this study, we extend a model proposed in \cite{Wang_Berestycki_2015} to incorporate the influence of protest management strategies. Our research establishes the existence and stability of traveling wave solutions, supported by both theoretical analyses and numerical simulations. We delve into the impact of two distinct protest management approaches on the qualitative and semi-quantitative characteristics of these traveling waves.  These metrics can aid in categorizing different types of protests.  
\end{abstract}

\section{Introduction}
\label{sec:Introduction}
A protest is a form of collective behavior used by groups to affect political change \cite{Wang_Ratliff_2014,Wang_Meyer_2004}.  In many cases, protests are spontaneous or loosely organized \cite{Wang_Snow_2014} and exhibit interesting spatial and temporal patterns \cite{Wang_Koopmans_2004}.
A self-reinforcing reaction-diffusion model to describe the interplay between the intensity of activity and the level of social tension was introduced in \cite{Wang_Berestycki_2015} by Berestycki and collaborators. This model exhibited intricate behavior and coherent structures, such as traveling wave solutions. We extend this model to include the global effects of policing strategies. We are motivated by the fact that many societies promote protest activity as a right but have fraught relationships with protest events and protesters when it comes to practical matters of policing \cite{Wang_DellaPorta_1998,Wang_Metcalfe_2022}.  We therefore aim to understand the effect of different protest management strategies on the dynamics of protesting activity and social tension.  
The model we will analyze is the system for the interplay of $(u,v, P)$ functions of physical space, $x$, and time $t:$
\begin{align}
\label{Wang_Equation1}
    \begin{cases}
        u_t = d_1 \Delta u + f(u,v,P), \quad x\in \R^2,\;t>0,\\
        v_t = d_2 \Delta v + g(u,v,P), \quad x\in \R^2,\;t>0,\\
        P_t = s(u,v,P), \quad x\in \R,\;t>0,\\
        (u(x,0),v(x,0),p(x,0))=(u_0(x),v_0(x),p(x)).
    \end{cases}
\end{align}
In \eqref{Wang_Equation3} the scalar fields $u,v$, and $P$ measure the level of social activity, social tension, and police density, respectively, with prescribed initial conditions $u_0(x),\;v_0(x)$ and $p(x)$. The reaction functions will be explained in detail in the modeling section, see section \ref{sec:model}, and model the kinetics between the three unknowns.    

\subsection{Model motivation and background}\label{sec:model}
While the factors that lead individuals to make any given decision widely vary, in certain situations many individuals converge in space and time to generate collective behavior.  The macroscopic patterns produced by the collective exhibit regularities that can be studied using mathematical models.  A notable example is that of protests.  
Over the years, social scientists have explored the factors contributing to the convergence of people with a shared objective in both time and space. This research has revealed a crucial element for this collective action to occur: heightened social tension within a group, which often arises from a prolonged situation that deviates significantly from the group's expectations \cite{Wang_Parson_1951}. Additionally, a triggering event typically initiates the collective action \cite{Wang_Gurr_1951, Wang_Snow_2014}.  

Spatial-temporal patterns in protest activities manifest in numerous instances. For instance, data analysis from the 2005 French riots revealed the emergence of traveling wave-like solutions originating from the point of the triggering event \cite{Wang_Bonnasse-Gahot_2018}. Similarly, during the Revolution of Dignity in Ukraine from late 2013 to 2014, a temporal pattern of self-excitation followed by relaxation was observed, with diffusion influenced by the political implications for Oblasts \cite{Wang_Bahid_2023}.

An essential aspect of the evolution of protest activities is the use of police protest management strategies. A key question arises regarding how these varied management strategies impact the overall dynamics \cite{Wang_DellaPorta_1998, Wang_Maguire_2015,Wang_Metcalfe_2022}. Research indicates that, in certain situations, participating in a protest can alleviate stress for participants \cite{Wang_Cosner_1956}. Conversely, in other scenarios, engagement in protesting may result in mental and physical trauma \cite{Wang_Haar_2017b, Wang_Haar_2017a, Wang_Ni_2020}. Mathematical models offer a means to investigate the influence of different protest management strategies, aiming to identify approaches that lead to ``productive action" —  where protest participants experience reduced stress levels and diminished social tension within the group.  

\subsubsection{The model without protest management}

In \cite{Wang_Berestycki_2015}, Berestycki and collaborators introduced a model for the interplay between the level of social activity and social tension 
\begin{eqnarray}\label{Wang_Equation2}
\left\{\begin{array}{cll}
u_{t}&=&d_1 \Delta u+r(v)G(u)-\omega u, \quad x\in \mathbb{R}^n,\;t>0,\\
v_{t}&=&d_2\Delta v+1-h(u)v,\quad x\in \mathbb{R}^n,\;t>0,\\
u(x,0)&=&u_0(x)\quad\text{and}\quad v(x,0)=v_0(x),
\end{array}\right. 
\end{eqnarray} 
where $u$ denotes the {\it level of activity} and $v$ the {\it social tension}.  The functions 
\begin{align*}
r(v) =  \frac{1}{1 + e^{-\beta(v-1)}},\quad  G(u)=\Gamma u(1-u)\quad\text{and}\quad h(u)=\theta (1+u)^{-k_0}
\end{align*}
are based on various sociological theories and interviews with protest participants. Please refer to Table \ref{table:parameters-label} for the physical interpretation of these parameters.  The model is rooted in socio-psychological theories and insights gathered from interviews with participants in the 2005 French riots (Figure \ref{Wang_Figure3}). It is widely accepted that social tension, coupled with a triggering event, can give rise to collective action \cite{Wang_Gurr_1951,Wang_Snow_2014}. Social tension is characterized as a negative emotional state within a group of individuals due to a prolonged, insoluble situation involving a misalignment between social expectations, interests, and the overall needs of the population \cite{Wang_Prokazina_2018, Wang_Parson_1951}. If social tension reaches a critical level, a triggering event—a swift catalyst capable of reshaping the political landscape and attracting or intensifying participation—can spark a protest. A useful analogy is that social tension creates combustible material, and a triggering event acts as the match that ignites the fire.

Data, sociological theories, and interviews with protest participants indicate that protesting activity can be self-exciting. In other words, action begets more action, up to a certain carrying capacity. This phenomenon is captured by the function $G,$ which resembles a Fisher-KPP-like term. However, this self-excitation effect only occurs when there is sufficiently high tension in the system; the function $r$ represents a switching mechanism that modulates the self-excitation effect depending on the social tension. The influence of the level of action on social tension introduces intriguing variations in the dynamics of the model. The modulation of $v$ by $u$ is characterized by the function $h$. The parameter $k_0$ determines whether the relationship between the two variables is enhancing or inhibiting. When $k_0 > 0$, action leads to a slowdown in tension; conversely, when $k_0 < 0$, action decreases the tension. As mentioned earlier, both scenarios have been observed in real life.  On the one hand, many protests induce mental and physical trauma \cite{Wang_Haar_2017b, Wang_Haar_2017a, Wang_Ni_2020}; on the other, some individuals have reported a release of stress after participating in peaceful protests. A situation with the presence of counter-protesters or when a peaceful protest turns violent may lead to $k_0>0$.   

\begin{figure} [H]
     \centering
     \begin{subfigure}[b]{0.5\textwidth}
         \centering
         \includegraphics[width=\textwidth]{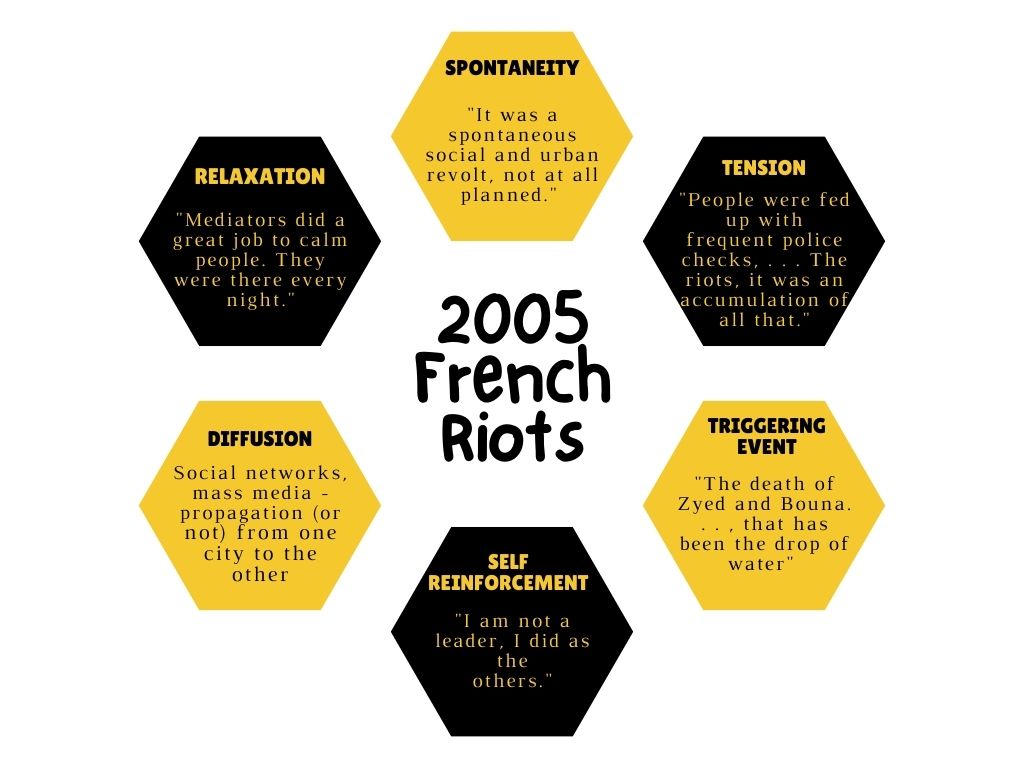}
         \caption{Interviews}
         \label{Wang_Figure1}
     \end{subfigure}
     \hfill
     \begin{subfigure}[b]{0.48\textwidth}
         \centering
         \includegraphics[width=\textwidth]{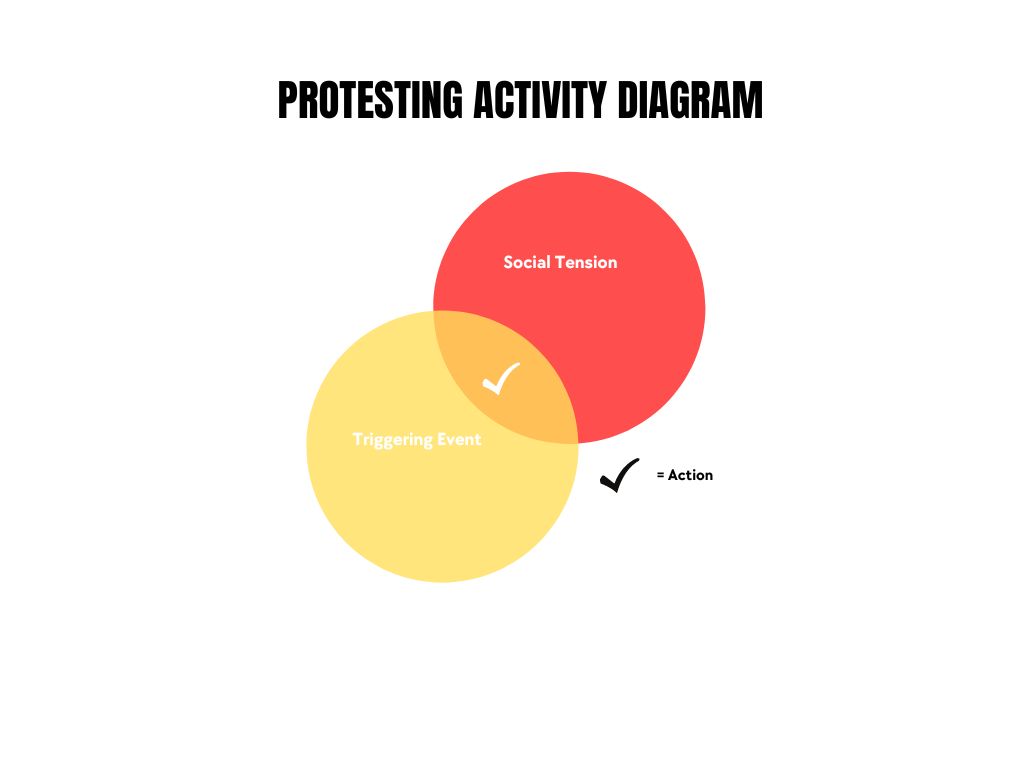}
         \caption{Protesting Activity}
         \label{Wang_Figure2}
     \end{subfigure}
     \caption{Diagrams illustrating the factors that influence the rise of protesting activity.}\label{Wang_Figure3}
\end{figure}

Prior studies have developed mathematical approaches to studying the $k_0$ positive and $k_0$ negative cases. In \cite{Wang_Yang_2020} and \cite{Wang_Yang_2021}, Yang and Rodr{\'{i}}guez have studied the traveling wave solutions both numerically and analytically for both cases. Such solutions in the tension-enhancing case have also been explored in \cite{Wang_Berestycki_2016} in heterogeneous environments where the restriction of information varies by region. The inhibitive case has been observed in \cite{Wang_Bonnasse-Gahot_2018} by fitting the 2005 French riots dataset with a Susceptible-Infected-Recovered (SIR)-type model. Moreover, a model similar to the one we are considering with a time-periodic source term has been studied in \cite{Wang_Berestycki_2018} where convergence to periodic solutions are proved for both the tension-enhancing and tension-inhibiting case.

\subsubsection{Incorporating protest management strategies}
While the original system described in equation \eqref{Wang_Equation2} exhibits intriguing behavior, it lacks integration of the significant impact that protest management has on the overall activity. In the United States, the government frequently employs law enforcement to handle protest activities, employing various protest management strategies. Chapter 2 of \cite{Wang_DellaPorta_1998} introduces five crucial characteristics of protest policing practices, outlined in Table \ref{table:protest_management}. To summarize, these characteristics encompass the degree of police concern for the First Amendment rights of protesters and their obligations to uphold and safeguard those rights, the level of police tolerance for community disruption, the dynamics of communication between police and demonstrators, the extent and method of arrests as a means of managing demonstrators, and the extent and method of using force, either independently or in conjunction with arrests, to control demonstrators.

In the 1960s, the prevailing management model in the United States was an {\bf escalated force model}, characterized by a robust and forceful approach based on deterrence theory \cite{Wang_D'Arcy_2011, Wang_Maguire_2015, Wang_DellaPorta_1998}. Deterrence theory posits that individuals seek to minimize pain and maximize pleasure, leading to the belief that the use of force and arrests would discourage action. The second column of Table \ref{table:protest_management} presents the different characteristics outlined earlier in the context of the escalated force model.  By contrast, in the 1980s and 1990s, the {\bf negotiated management model} emerged as the dominant U.S. approach \cite{Wang_DellaPorta_1998}. In stark opposition to the escalated force model, the primary goals of the negotiated management model were to safeguard protesters' First Amendment rights and preserve lives. The third column of Table \ref{table:protest_management} illustrates the various characteristics within the framework of the Negotiated Management model. While there are other models, including several variations of the Escalated Force model, a more detailed description of the various protest management strategies employed in the United States can be found in \cite{Wang_Maguire_2015}. We will concentrate on the two models highlighted above and presented in Table \ref{table:protest_management}.

\begin{table}[H]
\begin{center}
\begin{tabular}{|l|c|c|}
\hline
\multicolumn{1}{|c|}{\textbf{Characteristic}} & \textbf{Escalated Force} & \textbf{Negotiated Management}\\ \hline
Respect for $1^{st}$ Amend. Rights                                   & low                            & high                                 \\ \hline
Tolerance for disruption                                             & low                            & high                                 \\ \hline
Communication                                                         & low                            & high                                 \\ \hline
Arrest                                                               & quick                          & slow                                 \\ \hline
Use of force                                                          & quick                          & slow                                 \\ \hline
\end{tabular}
\end{center}
\caption{The main characteristics of the Escalated Force and the Negotiated Management models. \label{table:protest_management}}
\end{table}
\vspace{-12pt}
As is often the case, intriguing behaviors frequently stem from conflicting influences. While police presence may dissuade certain individuals from participating \cite{Wang_Taylor_2011}, an inappropriate protest management strategy could result in heightened tensions or increased participation \cite{Wang_Institute_2015}. Our objective is to refine model \eqref{Wang_Equation2} to encompass these opposing effects. Initially, we examine the direct impact of police presence on $u$. We posit that police presence can act as a deterrent or, conversely, motivate people to participate. This effect is represented by modifying the first evolution equation as follows: $u_t = d_1 \Delta u+H(P)r(v)G(u) - \omega u$,
where $H(P) = (1+P)^{k_1}$. Note that in this case when $k_1>0$ the presence of police enhances the level of activity when $k_1<0$ it inhibits the action, and when $k_1=0$ the equation reduces to the original evolution equation.

The presence of police and their management style can also indirectly influence social tension, impacting the dynamics of $v$. This is accomplished by adjusting the second equation as follows:
$$
v_t = d_2 \Delta v+1 - h(u)T(P)v,
$$
where $T(P) = (1+P)^{-k_2}$. It is important to note that when $k_2 > 0$, police presence enhances social tension, whereas when $k_2 < 0$, it inhibits social tension. As mentioned earlier, when $k_2 = 0$, our equation reverts to the original one. Throughout these modifications, we adhere to the convention that a positive $k_i$ ($i = 0, 1, 2$) signifies enhancement, while a negative parameter implies inhibition.  See Table \ref{table:k} for a summary of the interpretation of the sign of these parameters. 

\begin{table}[H]
\begin{tabular}{|c|c|c|}
\hline
\textbf{Parameter} & \textbf{Negative Value}                                                                          & \textbf{Positive Value}                                                                           \\ \hline
$k_0$              & Protesting action is tension-inhibiting                                                          & Protesting action is tension-enhancing                                                            \\ \hline
$k_1$              & \begin{tabular}[c]{@{}c@{}}Police presence has a deterrence \\ effect on Protesting\end{tabular} & \begin{tabular}[c]{@{}c@{}}Police presence has an incendiary \\ effect on Protesting\end{tabular} \\ \hline
$k_2$              & Escalating Force                                                                                 & Negotiated Management Model                                                                       \\ \hline
\end{tabular}
\caption{Summary of the physical interpretation of $k_i$ ($i = 0, 1, 2$).}\label{table:k}
\end{table}

Finally, we assume that the police are deployed at a proportional rate to the level of activity $u$ up to some carrying capacity.  This leads to the following equation: 
$$
P_t = u(1-P),
$$
where the level of activity serves as the growth rate of the police presence, and the carrying capacity of policing is assumed to be one.  Note that we assume no police diffusion.  

Putting everything together, the system that we are interested in is
\begin{align}
\label{Wang_Equation3}
    \begin{cases}
        u_t = d_1 u_{xx} + f(u,v,P),\\
        v_t = d_2 v_{xx} + g(u,v,P), \\
        P_t = s(u,v,P),
    \end{cases}
\end{align}
where the source functions are:
\begin{align}
\label{Wang_Equation4}
\begin{cases}
    f(u,v,P) = \left(1+P\right)^{k_1} \dfrac{1}{1 + e^{-\beta \left(v-\alpha\right)}}\Gamma u \left(1-u\right) - \omega u, \vspace{3pt}\\
    g(u,v,P) = 1 - \left(1+P\right)^{-k_2} \theta \left(1+u\right)^{-k_0}v, \vspace{3pt}\\
    s(u,v,P) = u(1-P).
\end{cases}
\end{align}

\begin{table}[H]
\begin{center}
\begin{tabular}{ |c|c| } 
 \hline
 Parameter & Description \\ 
 \hline\hline
 $d_1$ & diffusivity of protest activity \\ 
 $d_2$ & diffusivity of social tension \\ 
 $\alpha$ & critical level of social tension to transition from without to with self-reinforcement \\ 
 $\omega$ & rate of decay of protest activity \\ 
 $\theta$ & scaling factor for rioting level that scales social tension \\
 $\Gamma$ & rate of rioting self-reinforcement \\ 
 $\beta$ & rate of transition of without to with self-reinforcement \\
 \hline
\end{tabular}
\end{center}
\caption{The description of parameters in system \eqref{Wang_Equation3}.\label{table:parameters-label}}
\end{table}

The descriptions of model parameters are given in Table \ref{table:parameters-label}. For physical reasons, we limit the domains for the parameters to be $\alpha, \beta, \Gamma, \omega > 0.$

\subsubsection{Main Results}

In this work, we investigate the impact of various police management strategies on traveling wave solutions' quantitative and qualitative features. These solutions depict a mobile front where a high level of protest activity continuously spreads into regions with no activity. The traveling wave solutions to system \eqref{Wang_Equation2} consistently leave behind a significantly elevated level of activity and/or social tension.
Our focus involves proving the existence and stability of traveling wave solutions for system \eqref{Wang_Equation3}. The examination of wave existence is addressed in Section \ref{sec:Existence of traveling wave solutions}. We handle cases of monostable and bistable sources separately. We establish the existence of a monotone traveling wave satisfying system \eqref{Wang_Equation3} with a unique wave speed in bistable regimes. For monostable regimes, we demonstrate the existence of monotone traveling waves as long as the wave speed surpasses a specified threshold. The stability analysis is provided in Section \ref{sec:Stability and asymptotic decay rates}, where we also present asymptotic approach rates to the steady-state solutions. Our findings indicate that traveling waves with a bistable source are asymptotically stable with a shift according to the sup-norm. In contrast, waves with a monostable source are asymptotically stable according to a weighted norm, as detailed in the upcoming theorem.

We complement the theoretical results with numerical experiments across various parameter combinations and analyze the outcomes. We present key metrics to assess how parameter alterations influence the system's dynamics quantitatively. These metrics encompass the maximum and cumulative levels of activity and tension experienced and the level of protest activity at the interface infiltrating the region not yet exposed to the protest traveling wave. Additionally, we partition the parameter space into distinct regions based on specific simulation outcomes.  We identify regions in the parameter space where, after the traveling wave has long passed, there are either (1) persistent nonzero levels of protesting with the maximum density of police present, or (2) protest levels have diminished to near zero, and police densities are below maximum levels. Furthermore, we pinpoint regions where protesting and policing strategies have elevated tension following the dynamics of the traveling wave and regions where these strategies have successfully alleviated tensions—whether the tension relief is immediate following the protesting or occurs after a brief rise but ultimately results in a decline in tension.

\section{Background Matters}
In this section, we discuss the spatially homogeneous system, its equilibrium solutions, and the linearized stability of the system about these equilibrium solutions. Moreover, we introduce some relevant definitions and notations for our analysis of traveling wave solutions.  
\subsection{The spatially homogeneous system}
\label{sec:Fixed point analysis on the spatially homogeneous system}
First, we find and analyze the stationary points of the spatially homogeneous system
\begin{align}
\label{Wang_Equation5}
    \begin{cases}
        u_t = H(P)r(v)G(u) - \omega u, \\
        v_t = 1 - h(u)T(P)v,  \\
        P_t = u(1-P).
    \end{cases}
\end{align}
The nullclines for the $P$-equation are $u=0$ or $P=1$. If $u=0$, then we have a hyperplane of fixed points
$(0, (1+P)^{k_2}, P)$
for $0<P<1$, which correspond to the trivial steady states. If $P=1$, then we have fixed points given by the intersection of the three nullclines
\begin{align}
\label{Wang_Equation6}
\begin{cases}
    2^{k_1}r(v)G(u) - \omega u = 0, \\
    2^{k_2}(1+u)^{k_0} = v, \\
    P = 1,
\end{cases}
\end{align}
which correspond to the non-trivial steady states. From these equations of nullclines, we can see that the $u$ value at non-trivial steady states should always be strictly between zero and one. 

The Jacobian of $F = [f, g, s]^T$ is given by
\begin{align}
\label{Wang_Equation7}
DF(\bm{u})  
=  \left[\begin{array}{ccc}
    H(P)r(v)G'(u)-\omega & H(P)r'(v)G(u) & H'(P)r(v)G(u)\\
    -h'(u)T(P)v & -h(u)T(P)& -h(u)T'(P)v\\
    1-P & 0 & -u 
\end{array}\right].  
\end{align}
Then, the characteristic polynomial of the Jacobian is
\begin{align}
\label{Wang_Equation8}
\begin{split}
    & -\lambda^3 + (f_u + g_v + s_P)\lambda^2 + (-f_u(g_v + s_P) - g_v s_P + f_v g_u + f_P s_u)\lambda \\
    & + f_u g_v s_P - f_v g_u s_P + f_v g_P s_u - f_P s_u g_v = 0.
\end{split}
\end{align}
We find the eigenvalues and evaluate them at different steady states.

At the trivial steady state, $f_v,f_P,s_P = 0$, so the characteristic polynomial ~\eqref{Wang_Equation8} reduces to $-\lambda^3 + (f_u + g_v)\lambda^2 - f_u g_v\lambda = 0$, yielding the eigenvalues $0,f_u,g_v$. Evaluated at the trivial steady state, the three eigenvalues become
\begin{align}
\label{Wang_Equation9}
\begin{cases}
    \lambda_1 = 0, \\
    \lambda_2 = (1+P)^{k_1} \dfrac{\Gamma}{1 + e^{-\beta(v^* - \alpha)}} - \omega, \\
    \lambda_3 = -(1+P)^{-k_2},
\end{cases}
\end{align}
where $v^* = (1+P)^{k_2}$. Notice that $\lambda_3 < 0$ always, and the sign of $\lambda_2$ depends on the parameters.

At the nontrivial steady states, the eigenvalues are
\begin{align}
\label{Wang_Equation10}
\begin{cases}
    \lambda_1 = \dfrac{f_u + g_v + \sqrt{(f_u+g_v)^2+4f_vg_u-4f_ug_v}}{2},\vspace{3pt} \\
    \lambda_2 = \dfrac{f_u + g_v - \sqrt{(f_u+g_v)^2+4f_vg_u-4f_ug_v}}{2}, \\
    \lambda_3 = -u^*,
\end{cases}
\end{align}
where the partial derivatives are evaluated at the nontrivial steady states $(u^*,v^*,1)$ which can be obtained by solving ~\eqref{Wang_Equation6}. Note that $\lambda_3$ evaluated at the nontrivial steady states is always real and negative. For the monotone system, it turns out that all eigenvalues are real.  For $\lambda_1,\lambda_2$ to be negative, we need $f_vg_u - f_ug_v < 0$ and $f_u + g_v < 0$, the latter of which can be shown using some algebra, the former of which simplifies to
\begin{align}
\label{Wang_Equation11}
    \omega \beta k_0 e^{-\beta(v^*-\alpha)}\frac{v^*}{1+u^*} < 2^{k_1}\Gamma.
\end{align}

Condition ~\eqref{Wang_Equation11} will be checked numerically for the proofs later in the paper.

For the theoretical results, we focus on analyzing the case when system \eqref{Wang_Equation3} is monotone. However, we will explore more general parameter sets numerically.
By definition, our parabolic system ~\eqref{Wang_Equation3} is monotone if for any $\bm{u} \in \mathbb{R}^3$,
\begin{align*}
    \frac{\partial F_i}{\partial u_j} \ge 0, \quad i,j=1,2,3, \quad i \neq j,
\end{align*}
where $F_i = f,g,s$ for $i=1,2,3$, respectively. One can verify that system \eqref{Wang_Equation3} is monotone when $k_0,k_1,k_2, \beta, \Gamma \ge 0$, $0 \le u, P \le 1.$

We are interested in studying the system with different $k_0,k_1,k_2$ values. Given that changing the rest of the parameters has essentially the same effect on the number of steady states as changing $k_0, k_1, k_2$, we will fix $\alpha = 1, \beta = 1, \Gamma = 1, \omega = 0.5$ from now on unless specified otherwise. For example, varying $\alpha$ changes the number of steady states the same way as varying $k_1$ does.

Now, we divide the $k_1-k_2$ space into different regions depending on how many steady states there are. We consider the tension-enhancing ($k_0\geq 0$) and inhibiting ($k_0<0$) cases separately.

\subsubsection{Tension-enhancing case: $k_0 > 0$}
When $k_0 \ge 0$, the $v$-nullcline is monotone increasing in the positive $u-v$ region. Since we assume $\omega, \Gamma, \beta >0$, the $u$-nullcline is always monotone increasing. In this case, there are either zero, one, or two non-trivial steady states. Figure \ref{Wang_Figure7} shows the number of steady states in the $k_1-k_2$ regime. Given any constant initial condition for $P$, $P$ evolves to reach a specific constant value between zero and one at the trivial steady state, so the red regions indeed represent the regimes with one steady state instead of a hyperplane of them. 
\begin{comment}
    Figure ~\eqref{fig:ProfilesOfSteadyStateRegions} shows profiles of states at $t=500$ for regions 1, 2, and 3 in Figure \ref{Wang_Figure6} with different initial conditions.
\end{comment}

\begin{figure} [H]
     \centering
     \begin{subfigure}[b]{0.32\textwidth}
         \centering
         \includegraphics[width=\textwidth]{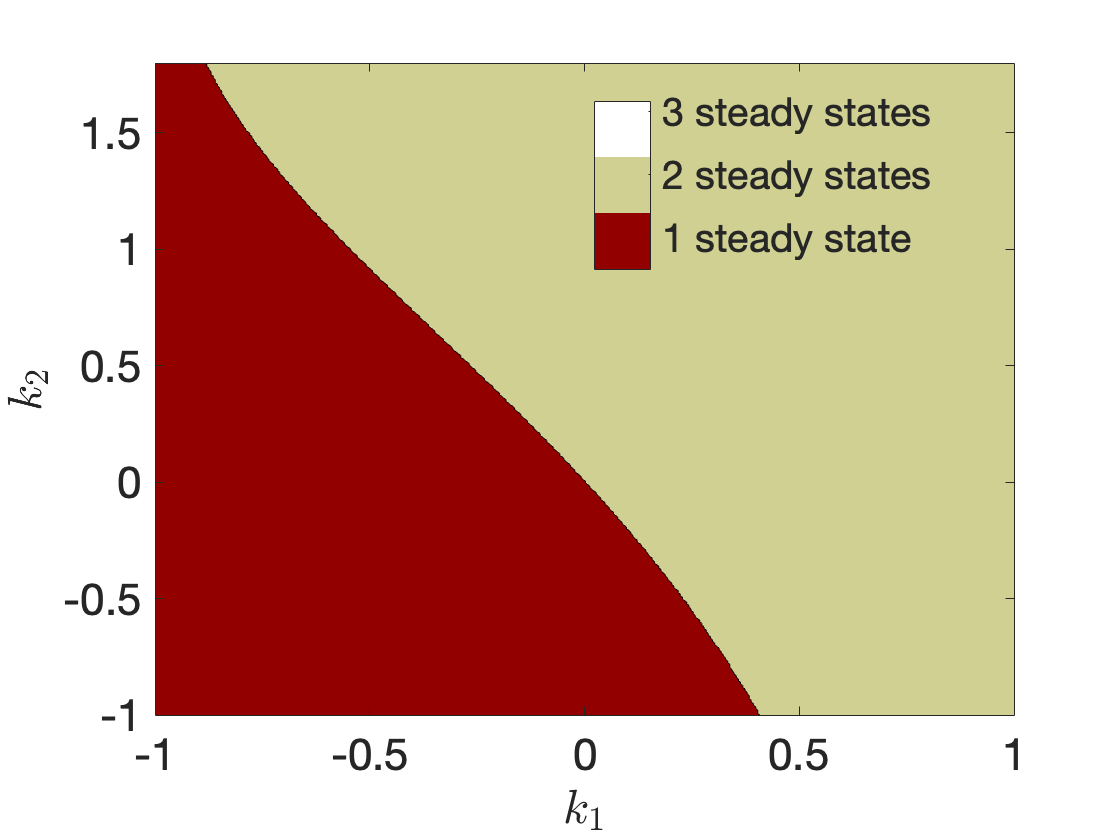}
         \caption{$k_0 = 1$}
         \label{Wang_Figure4}
     \end{subfigure}
     \hfill
     \begin{subfigure}[b]{0.32\textwidth}
         \centering
         \includegraphics[width=\textwidth]{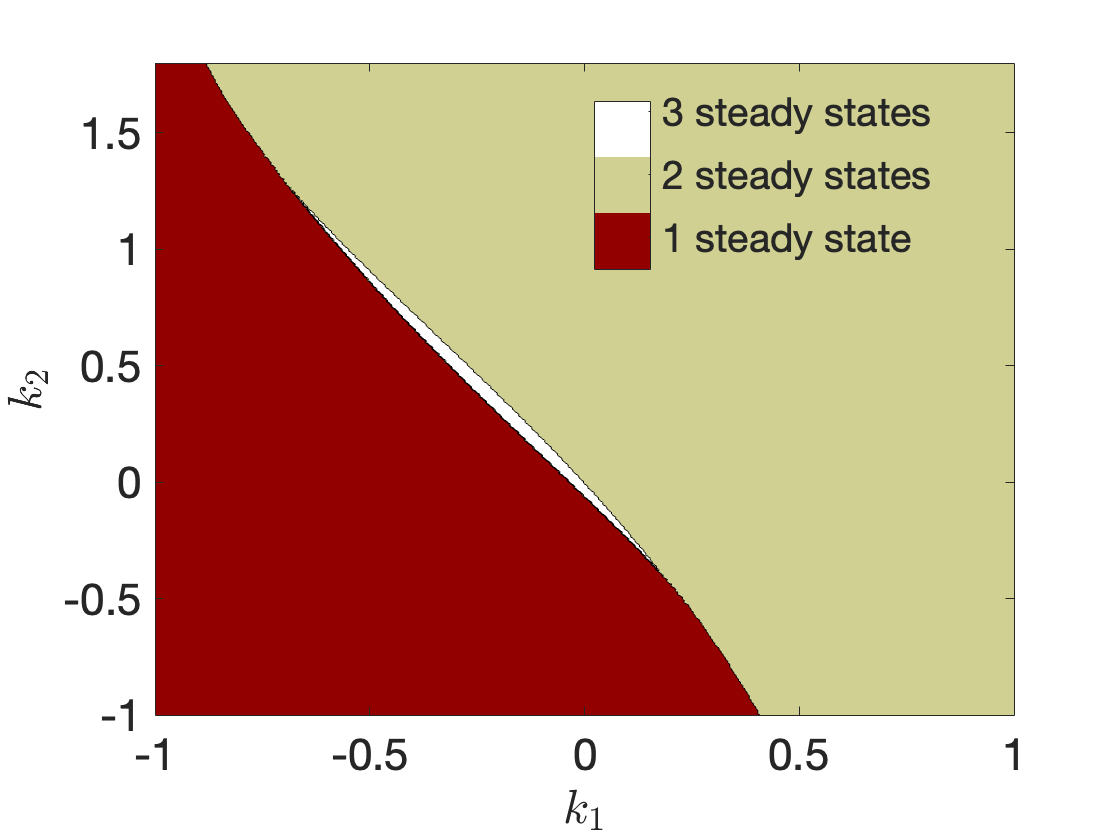}
         \caption{$k_0 = 2.5$}
         \label{Wang_Figure5}
     \end{subfigure}
     \hfill
     \begin{subfigure}[b]{0.32\textwidth}
         \centering
         \includegraphics[width=\textwidth]{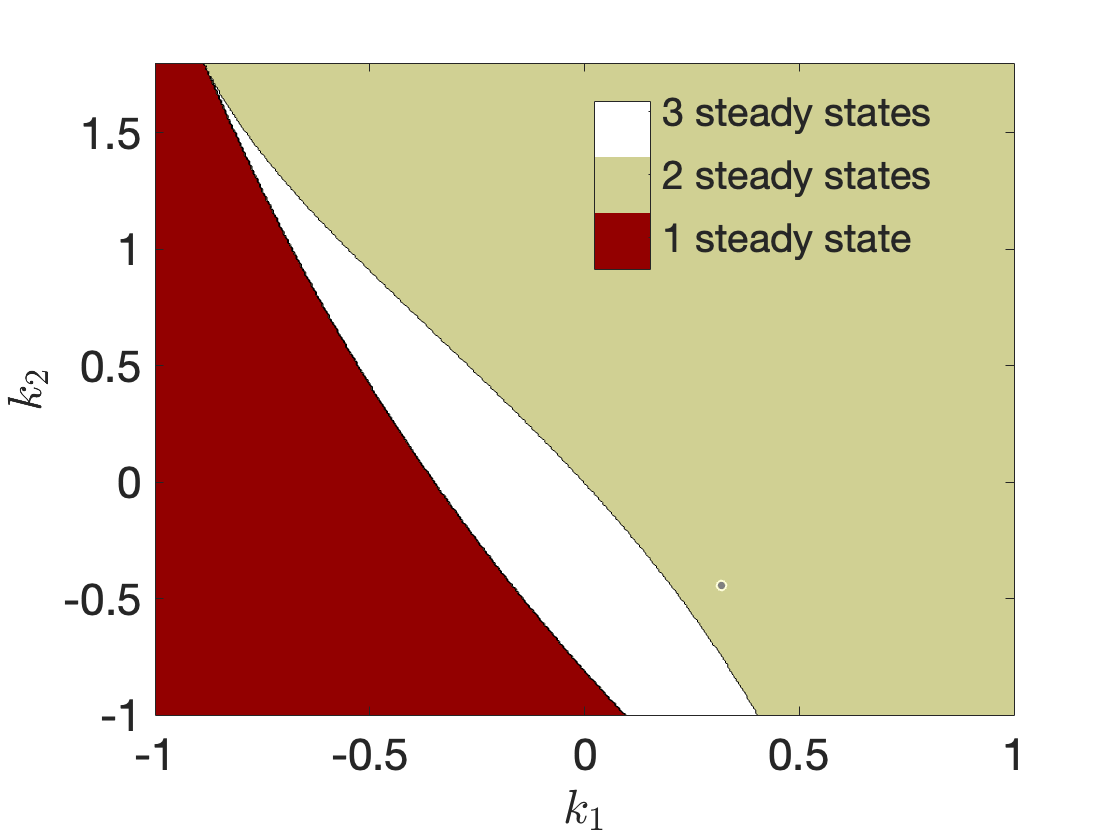}
         \caption{$k_0 = 4.5$}
         \label{Wang_Figure6}
     \end{subfigure}
     \hfill
        \caption{The $k_1-k_2$ space separated into regimes with different numbers of steady states with positive $k_0$, where $\alpha=1, \beta=1, \Gamma=1, \omega=0.5.$}
        \label{Wang_Figure7}
\end{figure}

\subsubsection{Tension-inhibiting case: $k_0 < 0$}
When $k_0 < 0$, the $v$-nullcline is always monotone decreasing in the positive $u-v$ region. Again, the $u$-nullcline is always monotone increasing. In this case, there is either zero or one non-trivial steady state. Figure \ref{Wang_Figure10} shows a plot separating the regions with different numbers of steady states in the $k_1-k_2$ regime. Notice that the magnitude of $k_0$ does not change the diagram as long as $k_0 < 0$.

\begin{figure} [H]
     \centering
     \begin{subfigure}[b]{0.45\textwidth}
         \centering
         \includegraphics[width=\textwidth]{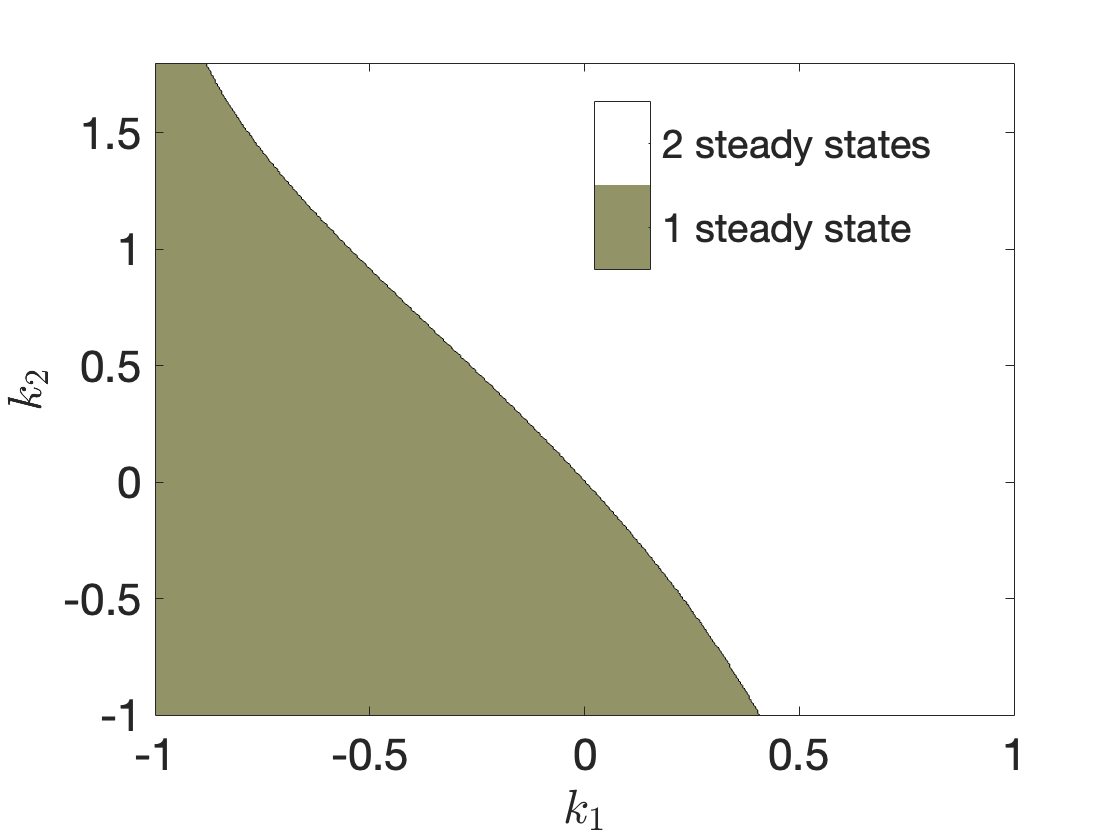}
         \caption{$k_0 = -0.5$}
         \label{Wang_Figure8}
     \end{subfigure}
     \hfill
     \begin{subfigure}[b]{0.45\textwidth}
         \centering
         \includegraphics[width=\textwidth]{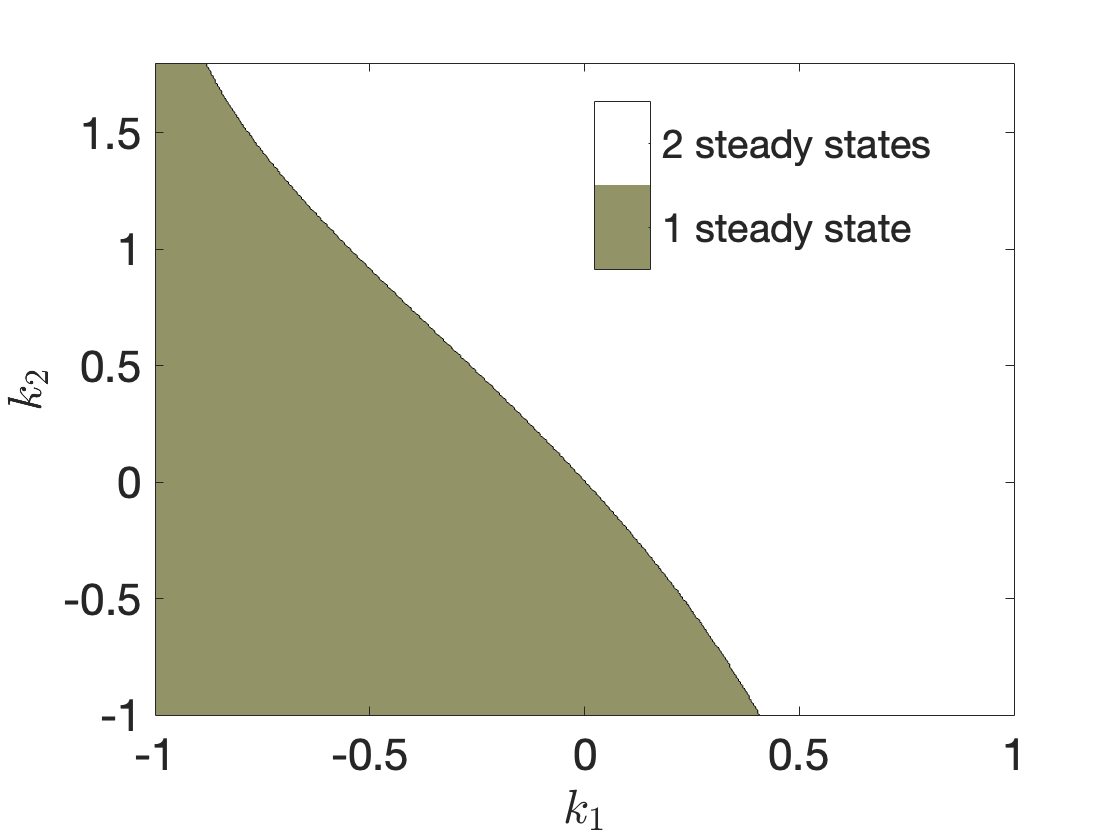}
         \caption{$k_0 = -4.5$}
         \label{Wang_Figure9}
     \end{subfigure}
     \hfill
        \caption{The $k_1-k_2$ space separated into regimes with different numbers of steady states with negative $k_0$, where $\alpha=1, \beta=1, \Gamma=1, \omega=0.5.$}
        \label{Wang_Figure10}
\end{figure}

\subsection{Existence and stability of traveling waves}
We are concerned with the existence and stability of traveling wave solutions to our system. 
Traveling wave solutions to our reaction-diffusion system \eqref{Wang_Equation3} have the vector form  $\bm{u}(x,t) = \hat{\bm{u}}(z)$ where $z=x-ct$ for some wave speed $c \in \mathbb{R}$ and satisfy
\begin{align}
\label{Wang_Equation12}
\begin{cases}
    A \hat{\bm{u}}'' + C \hat{\bm{u}}' + F (\hat{\bm{u}}) = 0, \\
    \lim_{z \to \pm \infty} \hat{\bm{u}}(z) = \hat{\bm{u}}_{\pm}.
\end{cases}
\end{align}
In system ~\eqref{Wang_Equation12}, $F = \begin{bmatrix}
    f \\ g \\ s
\end{bmatrix}$ with $f, g, s$ defined in ~\eqref{Wang_Equation3}, $A = \begin{bmatrix}
    1&0&0 \\ 0&1&0 \\ 0&0&0
\end{bmatrix}$ is the diffusion matrix, and $C = cI$.
Depending on the stability of the stationary points $\hat{\bm{u}}_{\pm}$, three types of sources are possible. The source is bistable if both the vectors $\hat{\bm{u}}_+$ and $\hat{\bm{u}}_-$ are stable. The source is monostable if one of $\hat{\bm{u}}_+$ and $\hat{\bm{u}}_-$ is stable, and the other is unstable. The source is unstable if both $\hat{\bm{u}}_+$ and $\hat{\bm{u}}_-$ are unstable. It can be seen in section \ref{sec:Existence of traveling wave solutions} that the existence and uniqueness of traveling wave solutions depend on the source type.

We are interested in monotone waves, as non-monotone waves for monotone systems are unstable \cite{Wang_Volpert_1994}. For definiteness, we assume that the waves are monotonically decreasing. Consequently, we have $\hat{\bm{u}}_+ < \hat{\bm{u}}_-$ (this inequality should be interpreted component wise).

Let $E$ be a Banach space to study the stability of waves, and we consider small perturbations from a Banach space $H$, lying in $E$. We denote the norm in $H$ by $\lVert \cdot \rVert _{H}$. These spaces will be made precise in each theorem.

\begin{definition}[\cite{Wang_Volpert_1994}]
    A wave $\hat{\bm{u}}(x)$ is \textit{asymptotically stable with shift} according to the norm $\lVert \cdot \rVert _{H}$, if there exists a positive number $\epsilon$ such that for an arbitrary real vector-valued function $\bm{u}_0 \in E$ with $\bm{u}_0 - \hat{\bm{u}} \in H$ and $\lVert \bm{u}_0 - \hat{\bm{u}} \rVert _{H} < \epsilon$, the solution $\bm{u}(x,t)$ of system \eqref{Wang_Equation3} with initial condition $\bm{u}(x,0) = \bm{u}_0 (x)$ exists for all $t>0$, is unique, $\bm{u}(x,t) - \hat{\bm{u}}(x) \in H$, and satisfies the estimate
    \begin{align}
    \label{Wang_Equation13}
        \lVert \bm{u}(x,t) - \hat{\bm{u}}(x+h) \rVert _{H} \le Me^{-bt},
    \end{align}
    where $h$ is a number depending on $\bm{u}_0(x)$; $M>0$ and $b>0$ are independent of $t,h$ and $\bm{u}_0(x)$.
\end{definition}

\begin{definition}[\cite{Wang_Volpert_1994}]
    A wave $\hat{\bm{u}}(x)$ is \textit{asymptotically stable} according to the norm $\lVert \cdot \rVert _{H}$, if there exists a positive number $\epsilon$ such that for an arbitrary real vector-valued function $\bm{u}_0 \in E$ with $\bm{u}_0 - \hat{\bm{u}} \in H$ and $\lVert \bm{u}_0 - \hat{\bm{u}} \rVert _{H} < \epsilon$, the solution $\bm{u}(x,t)$ of system \eqref{Wang_Equation3} with initial condition $\bm{u}(x,0) = \bm{u}_0 (x)$ exists for all $t>0$, is unique, $\bm{u}(x,t) - \hat{\bm{u}}(x) \in H$, and satisfies the estimate
    \begin{align}
    \label{Wang_Equation14}
        \lVert \bm{u}(x,t) - \hat{\bm{u}}(x) \rVert _{H} \le Me^{-bt},
    \end{align}
    where $M>0$ and $b>0$ are independent of $t$ and $\bm{u}_0(x)$.
\end{definition}

\section{Traveling waves in the monotone system: $k_0, k_1, k_2 \ge 0, 0 \le u \le 1$}

\subsection{Existence of traveling wave solutions}
\label{sec:Existence of traveling wave solutions}
To prove the existence of traveling wave solutions to ~\eqref{Wang_Equation3}, we need to determine the signs of the eigenvalues of Jacobian given in \eqref{Wang_Equation7}. 
Let us first look at the eigenvalues at the trivial steady states. Recall that the eigenvalues are given by ~\eqref{Wang_Equation9}, where $\lambda_1 = 0$ and $\lambda_3 < 0$. It is left to determine the sign of $\lambda_2$. Let us denote by $k_2(k_1,P)$ the boundary separating $\lambda_2 > 0$ from $\lambda_2 < 0$. By plotting in the $k_1-k_2$ plane, we find that the curve $k_2(k_1,1)$, represented by the blue dotted line in figure ~\eqref{Wang_Figure14}, coincides with the curve separating the region with two steady states from other numbers of steady states, which we denote by $k_2(k_1)$. More precisely, we see in figure ~\eqref{Wang_Figure14} that, for $P=1$, $\lambda_2 > 0$ when there are two steady states, and $\lambda_2 < 0$ when there are one or three steady states. For convenience, we denote the boundary separating the region with one steady state from other numbers of steady states by $k_2(k_0,k_1)$. Note that $k_2(k_1)$ coincides with $k_2(k_0,k_1)$ when $k_0$ is small (see Figures \ref{Wang_Figure11} and 
\ref{Wang_Figure15}).

\begin{figure} [H]
     \centering
     \begin{subfigure}[t]{0.32\textwidth}
         \centering
         \includegraphics[width=\textwidth]{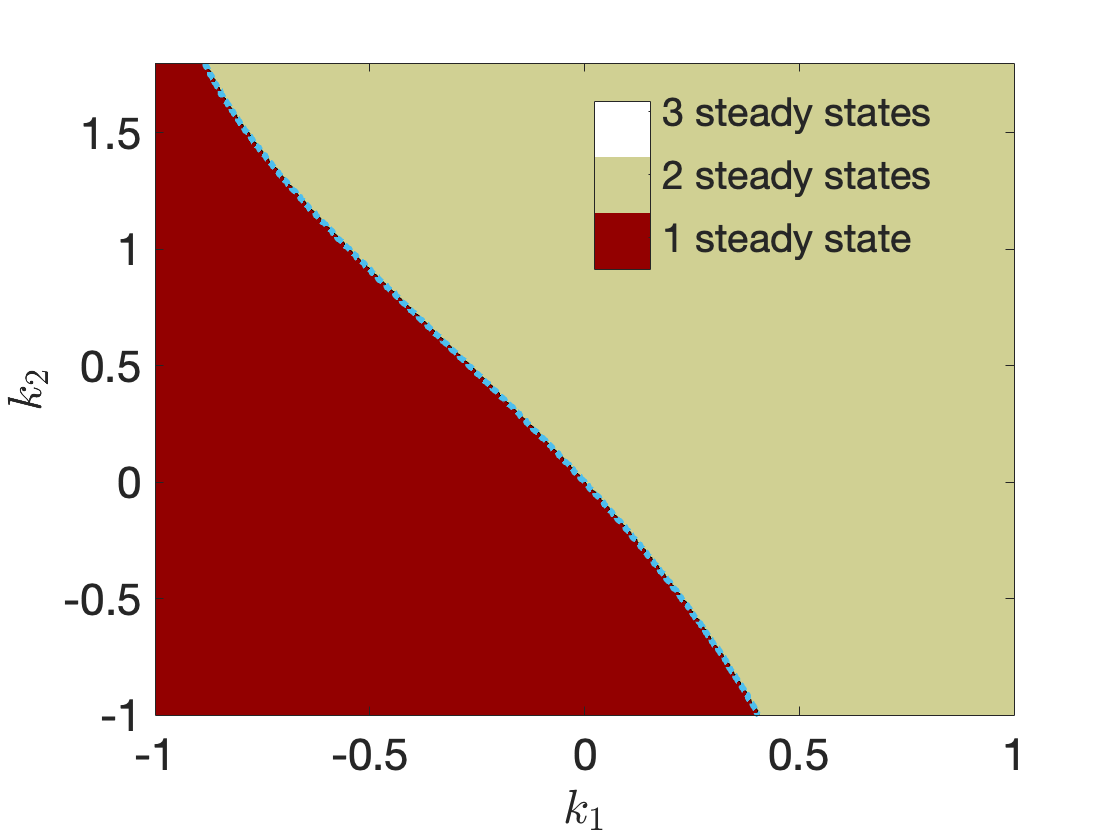}
         \caption{$k_0=1, P=1$}     \label{Wang_Figure11}
     \end{subfigure}
     \hfill
      \begin{subfigure}[t]{0.32\textwidth}
         \centering
         \includegraphics[width=\textwidth]{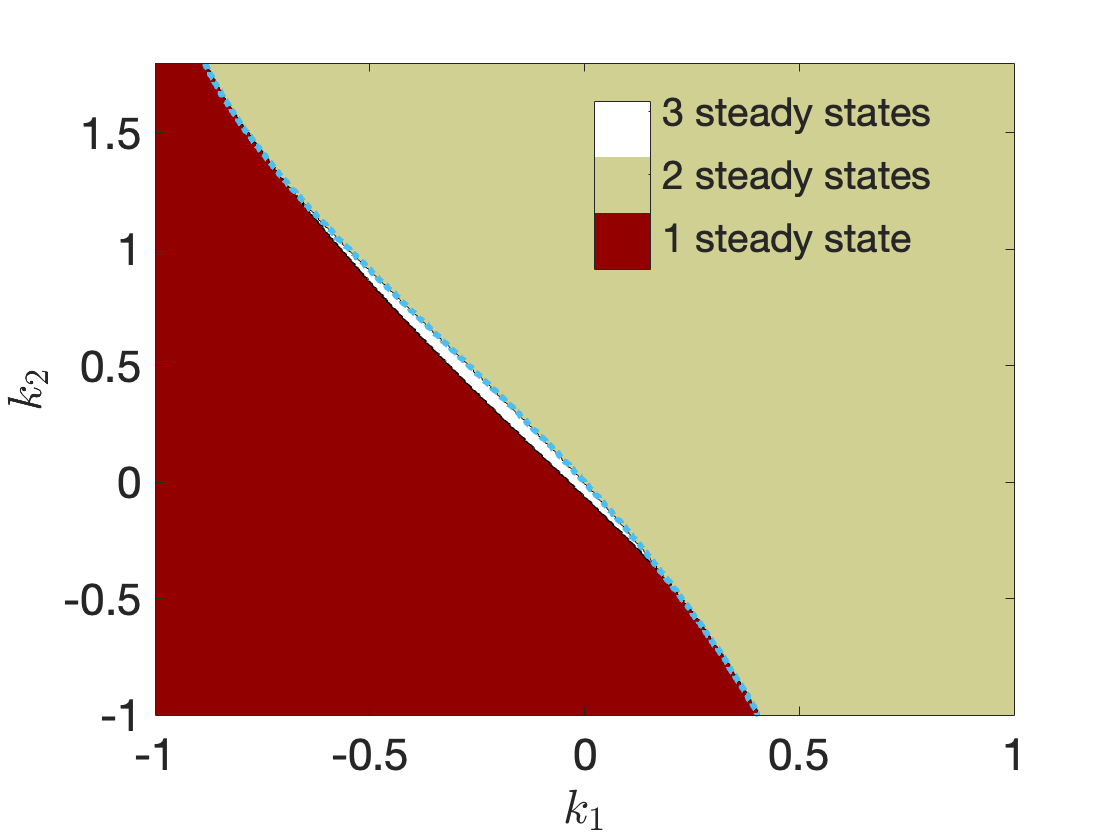}
         \caption{$k_0=2.5, P=1$}         \label{Wang_Figure12}
     \end{subfigure}
     \hfill
     \begin{subfigure}[t]{0.32\textwidth}
         \centering
         \includegraphics[width=\textwidth]{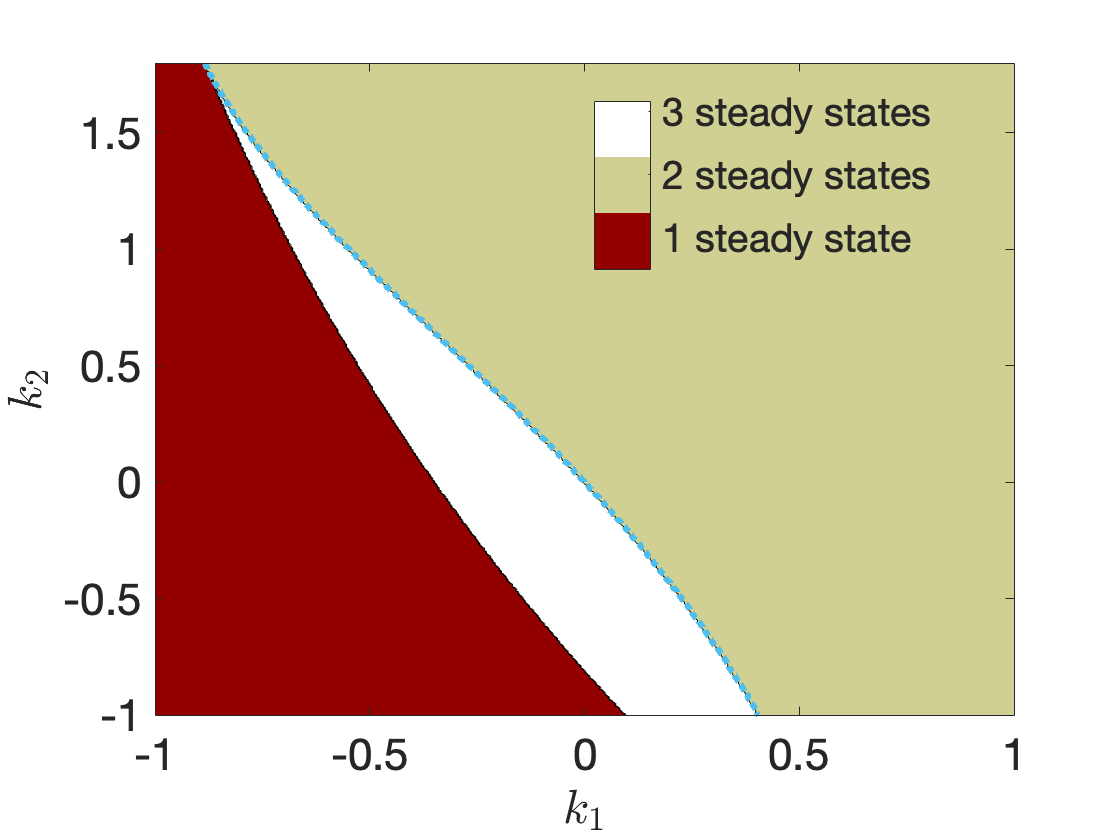}
         \caption{$k_0=4.5, P=1$}         \label{Wang_Figure13}
     \end{subfigure}
     
    \caption{The blue dotted line represents when $\lambda_2 = 0$, denoted by $k_2(k_1,1)$, which coincides with $k_2(k_1)$. When $k_2 > k_2(k_1,1)$, $\lambda_2 > 0$, and vice versa. For plot (a)-(c), $\omega=0.5, \alpha=1, \Gamma=1, \beta=1.$}
    \label{Wang_Figure14}
\end{figure}

When $P \neq 1$, the curve $k_2(k_1,P)$ does not necessarily coincide with $k_2(k_1)$ (see Figure \ref{Wang_Figure18}). In this more general case, we have up to six possible regions to analyze (see Figure \ref{Wang_Figure19}).
\begin{figure} [H]
     \centering
     \begin{subfigure}[t]{0.32\textwidth}
         \centering
         \includegraphics[width=\textwidth]{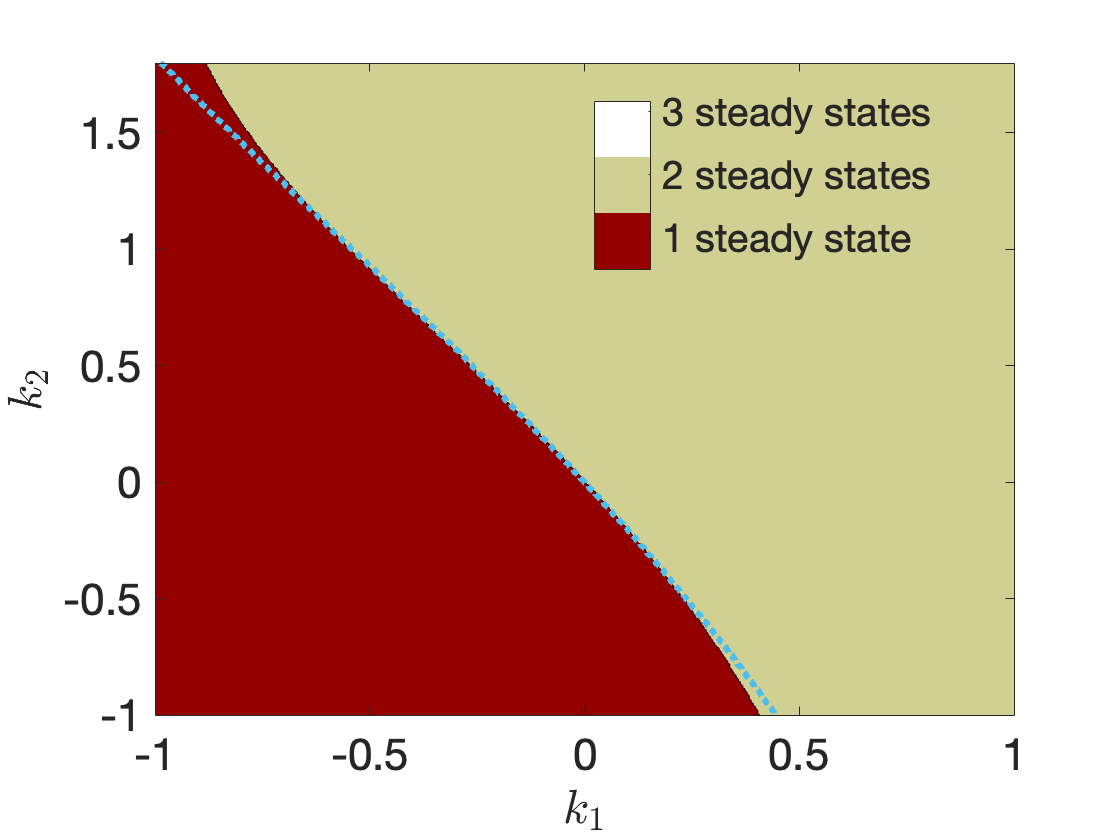}
         \caption{$k_0=1, P=0.5$}     \label{Wang_Figure15}
     \end{subfigure}
     \hfill
      \begin{subfigure}[t]{0.32\textwidth}
         \centering
         \includegraphics[width=\textwidth]{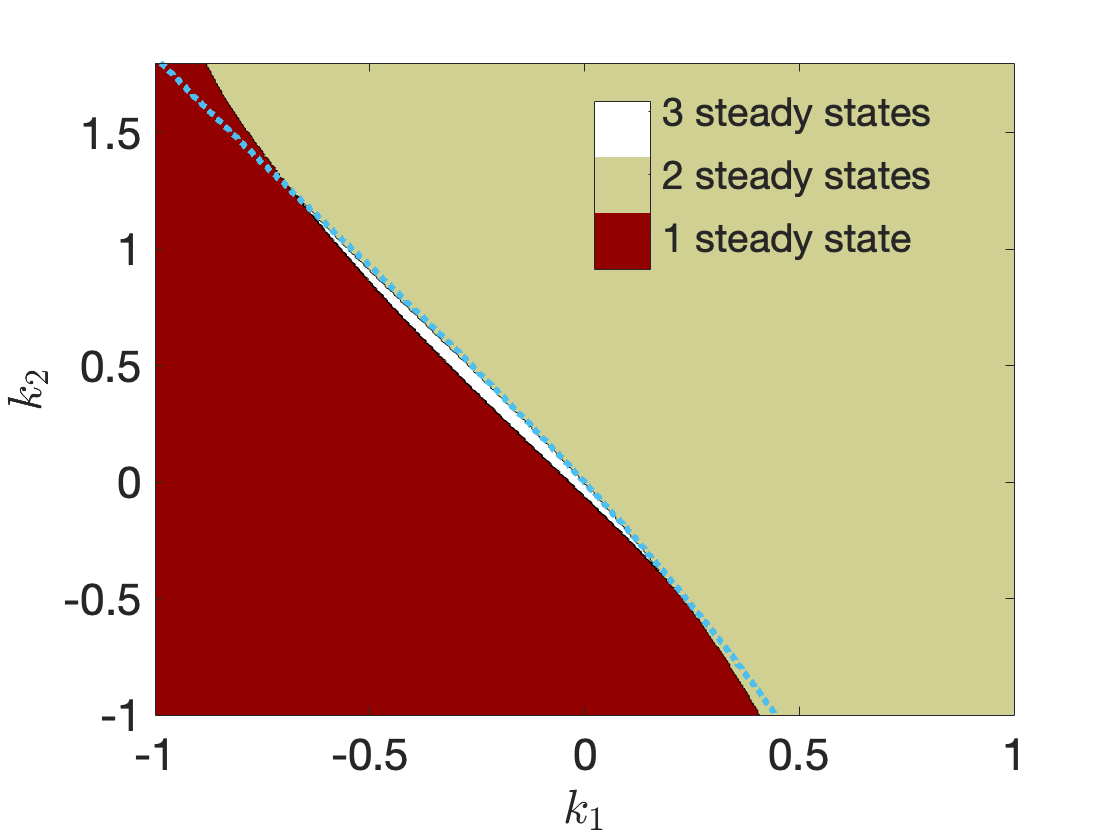}
         \caption{$k_0=2.5, P=0.5$}         \label{Wang_Figure16}
     \end{subfigure}
     \hfill
     \begin{subfigure}[t]{0.32\textwidth}
         \centering
         \includegraphics[width=\textwidth]{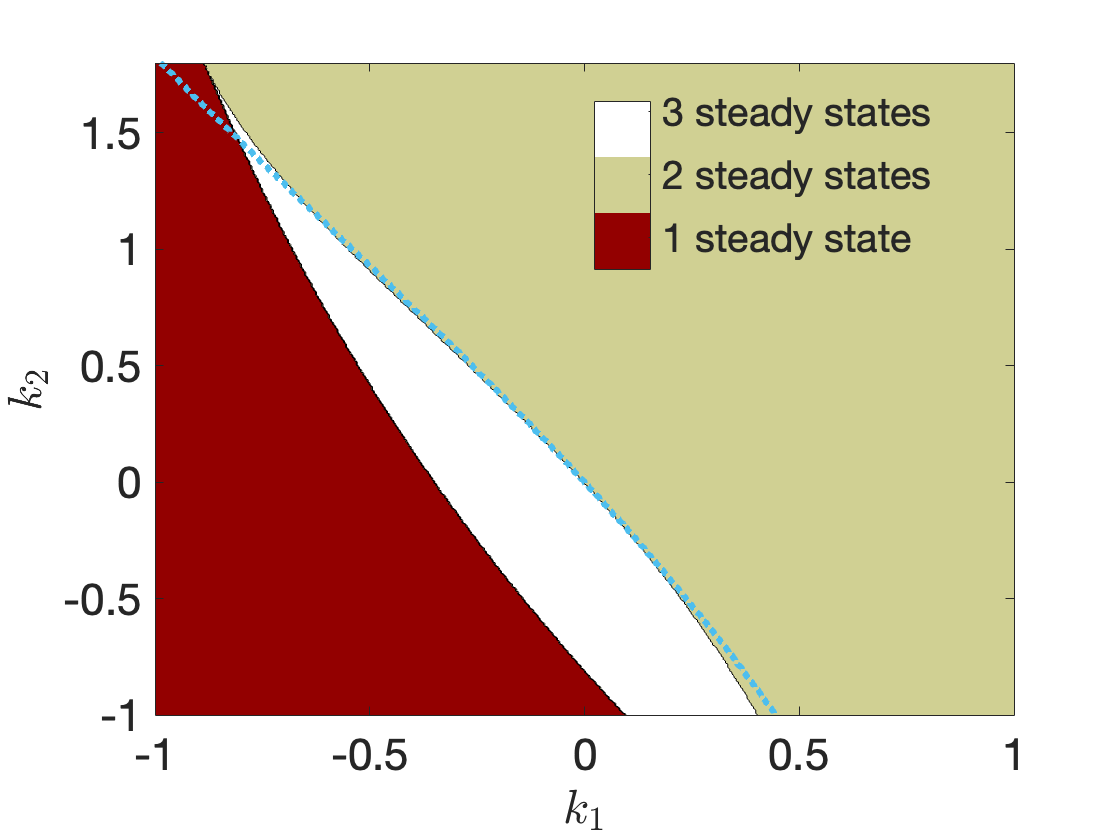}
         \caption{$k_0=4.5, P=0.5$}         \label{Wang_Figure17}
     \end{subfigure}
     
    \caption{The blue dotted line represents when $\lambda_2 = 0$, denoted by $k_2(k_1,0.5)$. When $k_2 > k_2(k_1,0.5)$, $\lambda_2 > 0$, and vice versa. For plot (a)-(c), $\omega=0.5, \alpha=1, \Gamma=1, \beta=1.$}
    \label{Wang_Figure18}
\end{figure}

\begin{figure} [H]
     \centering
        \includegraphics[scale=0.3]{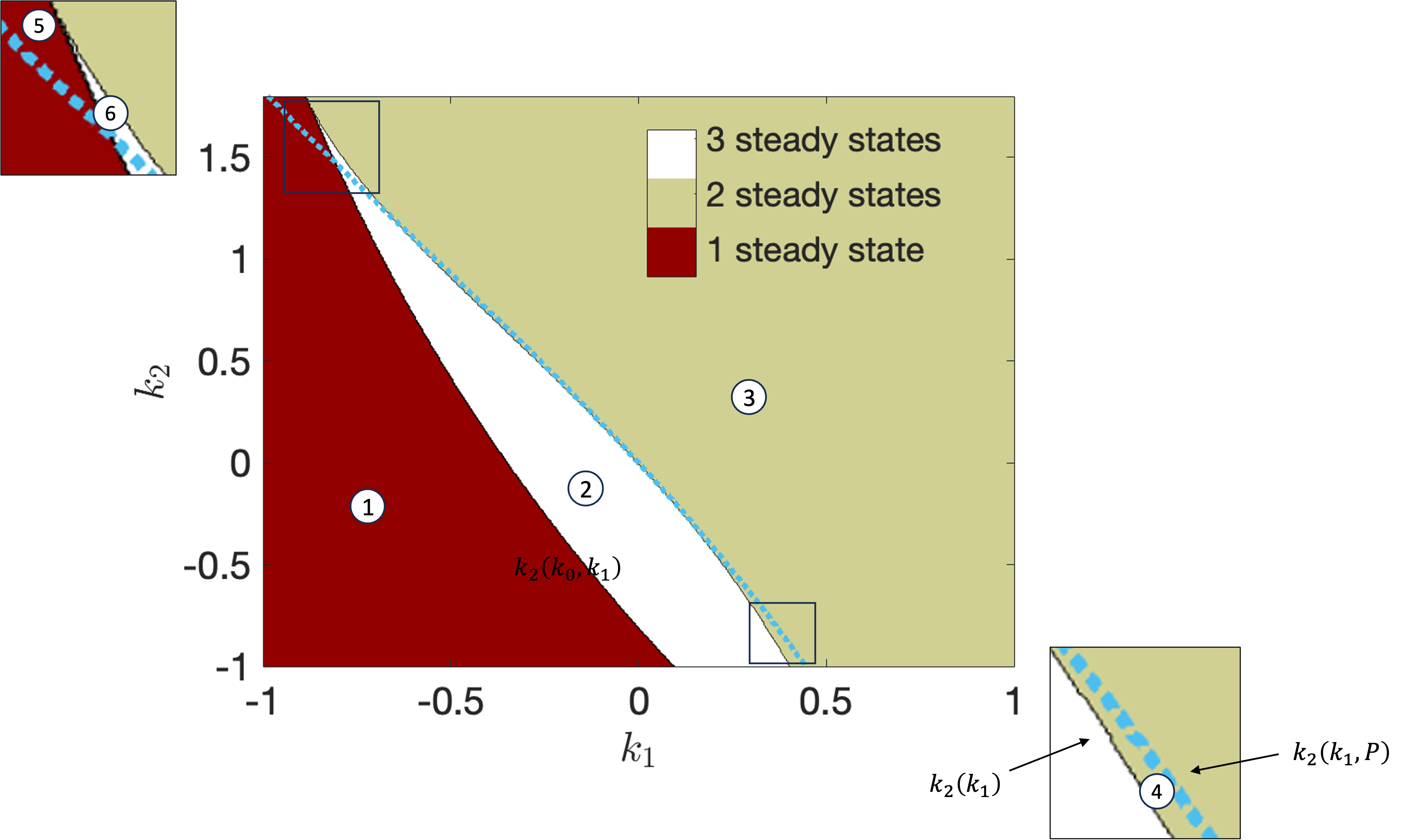}
        \caption{Six possible regions in the $k_1-k_2$ regime for $k_0 \ge 0$. 
        }
        \label{Wang_Figure19}
\end{figure}

Depending on the parameter set, regions five and six may be empty (see Figure \ref{Wang_Figure22}). 

\begin{figure} [H]
     \centering
     \begin{subfigure}[b]{0.45\textwidth}
         \centering
         \includegraphics[width=\textwidth]{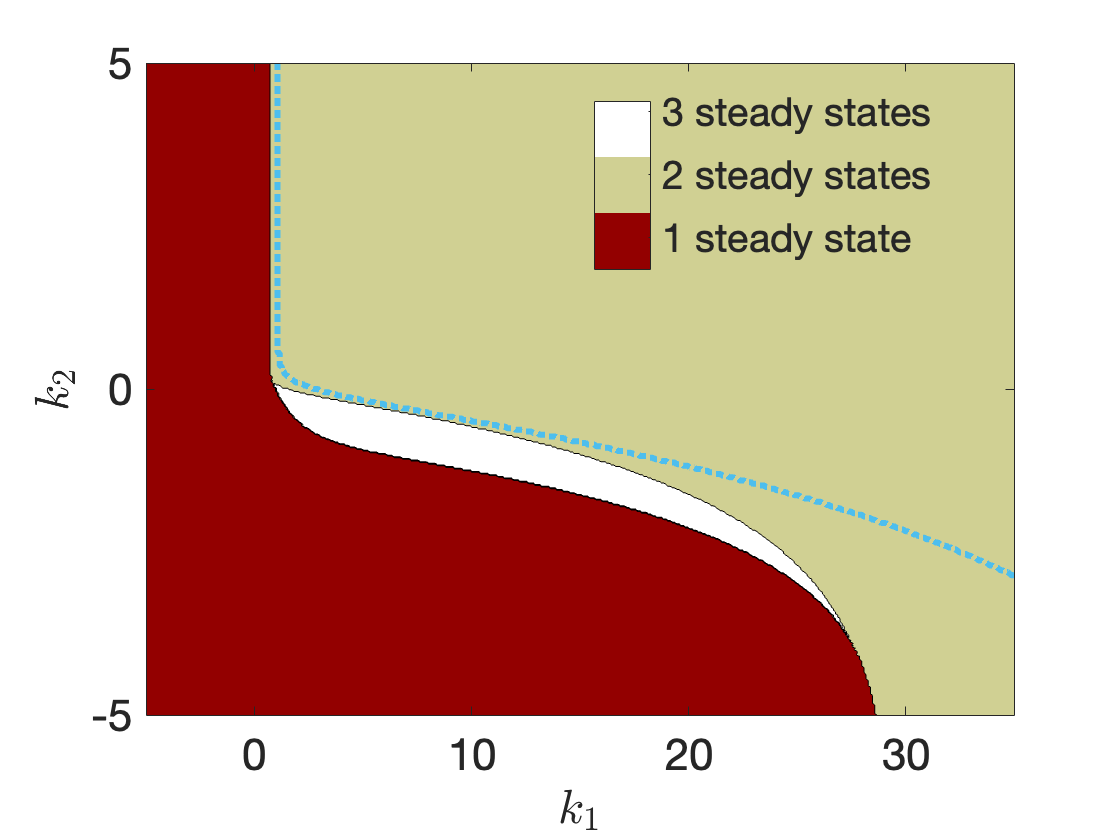}
         \caption{$k_0=1, P=0.5$}
         \label{Wang_Figure20}
     \end{subfigure}
     \hfill
     \begin{subfigure}[b]{0.45\textwidth}
         \centering
         \includegraphics[width=\textwidth]{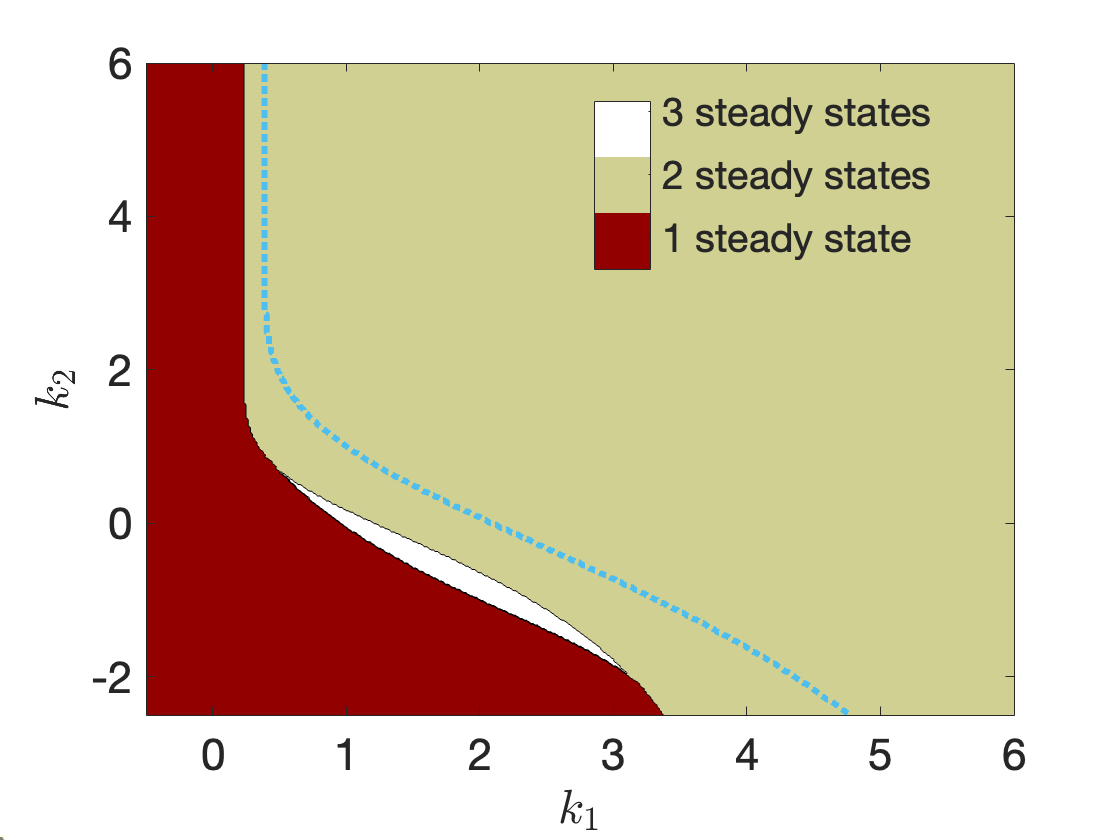}
         \caption{$k_0=2, P=0.5$}
         \label{Wang_Figure21}
     \end{subfigure}
     \hfill
        \caption{Possible parameter for which region 
         %\circled{5} and \circled{6} do not exist. For plot (a), $\alpha=1, \beta=20, \Gamma=160, \omega=250.$ For plot (b), $\alpha=1, \beta=2.5, \Gamma=210, \omega=245.$
        }
        \label{Wang_Figure22}
\end{figure}

The signs of eigenvalues at non-trivial steady states need to be checked separately for different values of $k_1, k_2$.

We now present the results for the existence of the traveling wave solutions.

\begin{theorem}[Monostable]
\label{thm:monostable-existence}
Let $k_0, k_1, k_2 \ge 0$. For $k_2 > k_2(k_0,k_1)$ and $k_2 > k_2(k_1,P)$, there exists a positive constant $c_*$ such that for all $c \ge c_*$, there exist monotone traveling wave solutions to ~\eqref{Wang_Equation12}. When $c < c_*$, such waves do not exist.
\end{theorem}

\begin{proof}
In the region where $k_0, k_1, k_2 \ge 0$, $k_2 > k_2(k_0,k_1)$ and $k_2 > k_2(k_1,P)$, there are either two or three steady states, corresponding to region \circled{3} or \circled{6}, respectively. We denote by $\hat{\bm{u}}_+, \hat{\bm{u}}^1, \hat{\bm{u}}_-$ the trivial, middle non-trivial, and largest non-trivial steady state, respectively, so that $\hat{\bm{u}}_+ < \hat{\bm{u}}^1 < \hat{\bm{u}}_-$. Since $k_2 > k_2(k_1, P)$, the second eigenvalue of $DF(\hat{\bm{u}}_+)$ given in \eqref{Wang_Equation9} lies in the right half-plane in this parameter regime, as shown in figure ~\eqref{Wang_Figure18} or ~\eqref{Wang_Figure14}. For the nontrivial steady states, one can check numerically that all three eigenvalues of $DF(\hat{\bm{u}}_-)$ are real and negative, and, if we are in region \circled{6}, that at least one eigenvalue of $DF(\hat{\bm{u}}^1)$ lies in the right half plane. Finally, by Theorem $2.2$ in Introduction of \cite{Wang_Volpert_1994}, we obtain the existence of waves. 
\end{proof}

\begin{theorem}[Bistable]
\label{thm:bistable-existence}
Let $k_0, k_1, k_2 \ge 0$. Additionally, let $k_0 > k_0^*$ for some $k_0^* \ge 0$ such that $k_2(k_1)$ and $k_2(k_0,k_1)$ do not coincide, guaranteeing a region with three steady states. For $k_2(k_0,k_1) < k_2 < k_2(k_1,P)$, there exists a monotonically decreasing traveling wave solution to ~\eqref{Wang_Equation12} with a unique wave speed $c$ if either there are two steady states, or there are three steady states and one of the two following conditions is satisfied:
\begin{enumerate}
    \item $k_2 \neq 0, u_1 \neq 1$,
    \item $k_2 = 0, k_1 \neq 0, u_1 \neq 1$,
\end{enumerate}
where $u_1$ is the $u$-value at the middle stationary point.
\end{theorem}

\begin{proof}
In the region where $k_0, k_1, k_2 \ge 0$ and $k_2(k_0,k_1) < k_2 < k_2(k_1,P)$, there are either two or three steady states, corresponding to region \circled{4} or \circled{2}, respectively. Same as in the proof above, we denote the three steady states by $\hat{\bm{u}}_+, \hat{\bm{u}}^1, \hat{\bm{u}}_-$. We want to prove that all the eigenvalues of the matrices $DF(\hat{\bm{u}}_+)$ and $DF(\hat{\bm{u}}_-)$ lie in the left half-plane. For the trivial steady state, recall that the first and third eigenvalue of $DF(\hat{\bm{u}}_+)$  given in \eqref{Wang_Equation9} are both negative. Since $k_2 < k_2(k_1,P)$, the second eigenvalue of $DF(\hat{\bm{u}}_+)$ lies in the left half-plane, as shown in figure ~\eqref{Wang_Figure18} or ~\eqref{Wang_Figure14}. For the nontrivial steady state $\hat{\bm{u}}_-$, one can check numerically that all three eigenvalues of $DF(\hat{\bm{u}}_-)$ are real and negative. If we are in region \circled{4}, we are done. If we are in region \circled{2}, then it remains to show that there exists a non-negative vector $\bm{q}_1 = 
\begin{bmatrix}
    q_1 & q_2 & q_3
\end{bmatrix}$ such that $\bm{q}_1 DF(\hat{\bm{u}}^1) > 0$. If $\hat{\bm{u}}^1 =
\begin{bmatrix}
    u_1 & v_1 & P_1
\end{bmatrix}$, then the inequality can be written as
\begin{align}
\label{Wang_Equation15}
\begin{cases}
    -q_1 |f_u| + q_2 |g_u| > 0, \\
    q_1 |f_v| - q_2 |g_v| > 0, \\
    q_1 |f_P| + q_2 |g_P| - q_3 |s_P| > 0,
\end{cases}
\end{align}
where we have used the fact that $s_u(\hat{\bm{u}}^1) = 0$ and $ f_u(\hat{\bm{u}}^1), g_v(\hat{\bm{u}}^1), s_P(\hat{\bm{u}}^1) < 0$. Note that if $k_0 = 0$ then there are only two stationary points, so we only need to consider the case where $k_0 \neq 0$, which guarantees $g_u(\hat{\bm{u}}^1) \neq 0$.  Since $k_2 = 0$ implies and $g_P(\hat{\bm{u}}^1) = 0$, we want to discuss by cases depending on whether $k_2$ is zero or not. Note that $g_v(\hat{\bm{u}}^1) \neq 0$ for any parameter values.

Case 1: if $k_2 > 0$, then inequality ~\eqref{Wang_Equation15} can be further written as
\begin{align*}
\begin{cases}
    q_2 > \cfrac{q_1 |f_u|}{|g_u|}, \vspace{3pt}\\
    q_2 < \cfrac{q_1 |f_v|}{|g_v|}, \vspace{3pt}\\
    q_2 > -\cfrac{q_1 |f_P| - q_3 |s_P|}{|g_P|}.
\end{cases}
\end{align*}
Such $\bm{q}_1$ exists if there are $q_1, q_3$ such that
\begin{align*}
\begin{cases}
    \cfrac{q_1 |f_v|}{|g_v|} > \cfrac{q_1 |f_u|}{|g_u|}, \vspace{3pt}\\
    \cfrac{q_1 |f_v|}{|g_v|} > -\cfrac{q_1 |f_P| - q_3 |s_P|}{|g_P|},
\end{cases}
\end{align*}
which leads to
\begin{align*}
\begin{cases}
    q_1 \left (\cfrac{|f_v|}{|g_v|} - \cfrac{|f_u|}{|g_u|} \right) > 0, \vspace{3pt}\\
    q_3 < q_1 \cfrac{|g_P|}{|s_P|} \left (\cfrac{|f_v|}{|g_v|} + \cfrac{|f_P|}{|g_P|} \right).
\end{cases}
\end{align*}
Note that such $q_1, q_3$ exist only if
\begin{align*}
\begin{cases}
    \cfrac{|f_v|}{|g_v|} - \cfrac{|f_u|}{|g_u|} > 0, \vspace{3pt}\\
    \cfrac{|g_P|}{|s_P|} \left (\cfrac{|f_v|}{|g_v|} + \cfrac{|f_P|}{|g_P|} \right) > 0,
\end{cases}
\end{align*}
yielding the conditions
\begin{align*}
\begin{cases}
    f_v g_u - f_u g_v > 0, \\
    u_1 \neq 1. 
\end{cases}
\end{align*}

Case 2: if $k_2 = 0$, then a similar derivation as above yields the conditions
\begin{align*}
\begin{cases}
    f_v g_u - f_u g_v > 0, \\
    k_1 \neq 0, u_1 \neq 1. \\
\end{cases}
\end{align*}

One can check that the above strict inequality for both cases is satisfied at $\hat{\bm{u}} = \hat{\bm{u}}^1$.

Therefore, we obtain the existence and uniqueness of monotone traveling waves by Theorem $3.2$ in Chapter $3$ of \cite{Wang_Volpert_1994}. 
\end{proof}

In Table \ref{tab:source-types-in-6-regions} we present a summary of different regions in the $k_1-k_2$ regime with the according numbers of steady states and source types. The regions correspond to the ones labeled in Figure \ref{Wang_Figure19}.

\begin{table} [H]
\begin{center}
\caption{Six regions and their source types.\label{tab:source-types-in-6-regions}}
\begin{tabular}{ ||c|c|c|c|| } 
 \hline
 Label & Region & Number of steady states & Source type \\ 
 \hline\hline
 \circled{1} & $k_2<k_2(k_0,k_1), k_2<k_2(k_1,P)$ & 1 & near-pulse \\ 
 \hline
 \circled{2} & $k_2(k_0,k_1)<k_2<k_2(k_1), k_2<k_2(k_1,P)$ & 3 & bistable \\ 
 \hline
 \circled{3} & $k_2>k_2(k_1), k_2>k_2(k_1,P)$ & 2 & monostable \\
 \hline
 \circled{4} & $k_2(k_1)<k_2<k_2(k_1,P)$ & 2 & bistable \\
 \hline
 \circled{5} & $k_2(k_1,P)<k_2<k_2(k_0,k_1)$ & 1 & near-pulse \\
 \hline
 \circled{6} & $k_2(k_0,k_1)<k_2<k_2(k_1), k_2>k_2(k_1,P)$ & 3 & monostable \\
 \hline
\end{tabular}
\end{center}
\end{table}
Figure \ref{Wang_Figure25} illustrates traveling wave profiles in the regimes discussed in Theorem \ref{thm:monostable-existence} and \ref{thm:bistable-existence}.
\begin{figure} [H]
     \centering
     \begin{subfigure}[t]{0.48\textwidth}
         \centering
         \includegraphics[width=\textwidth]{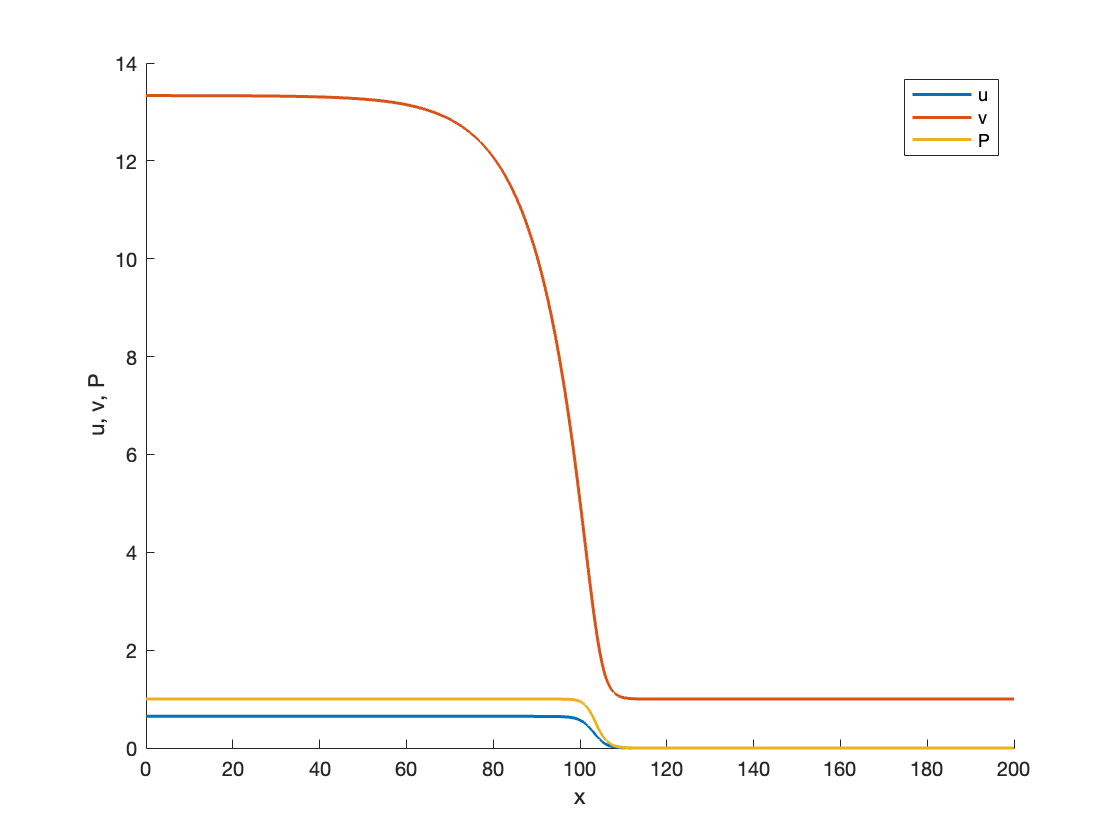}
         \caption{The monostable region with $k_2=0.5$ and $\Gamma=1$.}    \label{Wang_Figure23}
     \end{subfigure}
     \hfill
     \begin{subfigure}[t]{0.48\textwidth}
         \centering
         \includegraphics[width=\textwidth]{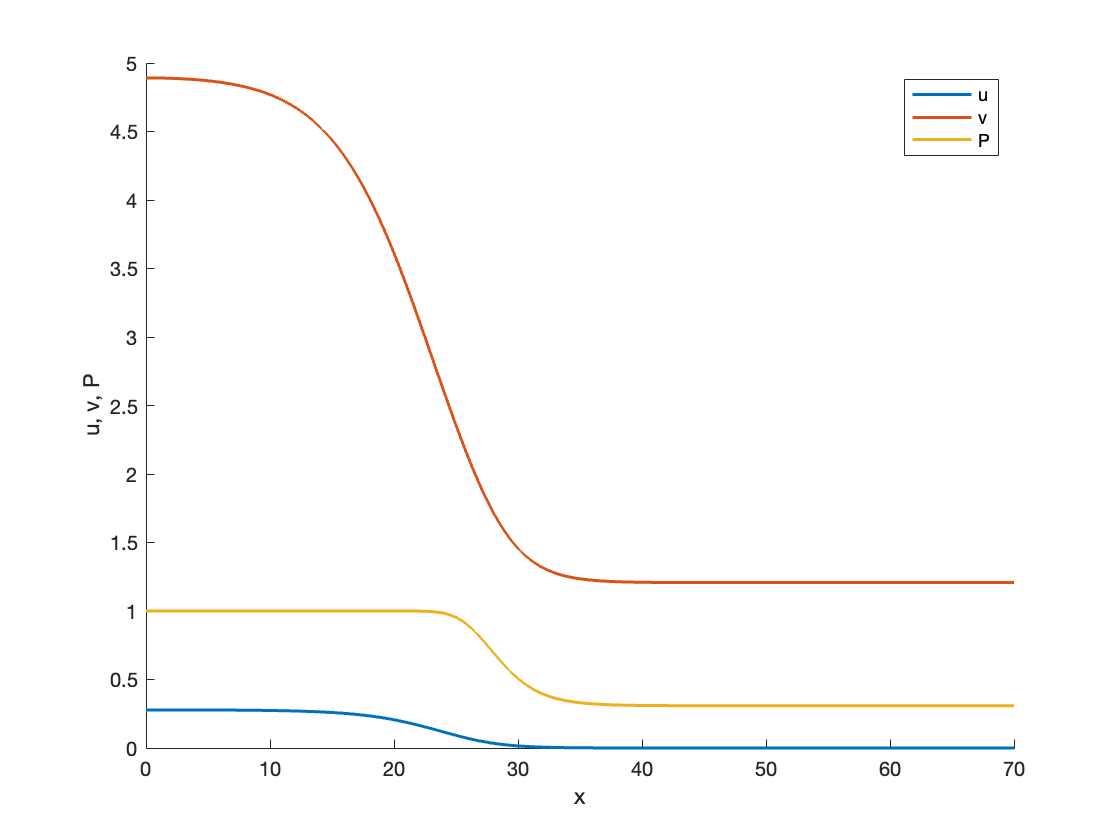}
         \caption{The bistable region with $k_2=0.7$ and $\Gamma=0.5$.}  \label{Wang_Figure24}
     \end{subfigure}
     
    \caption{Examples of traveling wave profiles with $k_0=4.5, k_1=0.5, \alpha=1, \beta=1,$ and $\omega=0.5.$}
    \label{Wang_Figure25}
\end{figure}

\subsection{Stability and asymptotic decay rates of traveling wave solutions}
\label{sec:Stability and asymptotic decay rates}
In this section, we explore the stability of the traveling wave solutions whose existence is proved in the previous section. We want to determine the types of perturbations allowed to obtain stable waves in different parameter regimes.

\subsubsection{Asymptotic decay rates}
\label{sec:Asymptotic rates}
We first present some preliminary results on the asymptotic decay rates of the waves, which will assist us in proving the stability of the waves in the monostable region. These rates show that the traveling waves and their derivatives respectively approach the steady states and zero exponentially fast as $|z| \to \infty$. Exact proofs for the following propositions can be found in Appendix \ref{sec:appendix-asymptotic-rates}.

\begin{proposition} [Asymptotic approach rates for $k_0 \ge 0$] 
\label{prop:asymp-decay-rates-triv}
Let $\hat{\bm{u}}(z) = 
\begin{bmatrix}
    \hat{u}(z) \\
    \hat{v}(z) \\
    \hat{P}(z)
\end{bmatrix}$ 
be the traveling wave solution to ~\eqref{Wang_Equation3} that satisfies
\begin{align*}
\begin{cases}
    \lim_{z \to +\infty} \hat{\bm{u}}(z) = \hat{\bm{u}}_+, \quad
    \lim_{z \to -\infty} \hat{\bm{u}}(z) = \hat{\bm{u}}_-, \\
    \hat{u}(z), \hat{v}(z) \ge 0, \quad 0 \le \hat{P}(z) \le 1, \\
    \hat{u}'(z), \hat{v}'(z) < 0,
\end{cases}
\end{align*}
where
$\hat{\bm{u}}_+ = 
\begin{bmatrix}
    0 \\
    (1+P)^{-k_2} \\
    P
\end{bmatrix},
\hat{\bm{u}}_- = 
\begin{bmatrix}
    u^* \\
    v^* \\
    1
\end{bmatrix}$.

Define $\mu_{1,2} = -(f_u(\hat{\bm{u}}_-) + g_v(\hat{\bm{u}}_-)) \pm \sqrt{(f_u(\hat{\bm{u}}_-) + g_v(\hat{\bm{u}}_-))^2 + 4f_v(\hat{\bm{u}}_-) g_u(\hat{\bm{u}}_-) - 4f_u(\hat{\bm{u}}_-) g_v(\hat{\bm{u}}_-)},$ where $\mu_1, \mu_2$ take the plus, minus sign, respectively.

(a) There are nonnegative vectors $\bm{k}_c, \bm{K}_c, \bm{k}_{\epsilon}, \bm{K}_{\epsilon}, \bm{m}_c, \bm{M}_c, \bm{m}_{\epsilon}, \bm{M}_{\epsilon}$, negative numbers $\lambda_5, \lambda_3, \lambda_2$, and positive numbers $\epsilon, \Tilde{\lambda}_1, \Tilde{\lambda}_2$ such that the nonzero elements of $\bm{K}_{\epsilon}, \bm{M}_{\epsilon}$ go to $\infty$ as $\epsilon \to 0$, and
\begin{align*}
    \begin{cases}
    \bm{k}_c e^{\lambda_5 z} \le \hat{\bm{u}}(z) - \hat{\bm{u}}_+ \le \bm{K}_c e^{\lambda_5 z},
    & (z \ge 0), f_u(\hat{\bm{u}}_+) < g_v(\hat{\bm{u}}_+), \\
    \bm{k}_{\epsilon} e^{(\lambda_5 - \epsilon)z} \le \hat{\bm{u}}(z) - \hat{\bm{u}}_+ \le \bm{K}_{\epsilon} e^{(\lambda_5 + \epsilon)z},
    & (z \ge 0), f_u(\hat{\bm{u}}_+) = g_v(\hat{\bm{u}}_+), \\
    \bm{k}_c e^{\lambda_3 z} \le \hat{\bm{u}}(z) - \hat{\bm{u}}_+ \le \bm{K}_c e^{\lambda_3 z},
    & (z \ge 0), g_v(\hat{\bm{u}}_+) < f_u(\hat{\bm{u}}_+) < 0, \\
    \bm{k}_c e^{\lambda_2 z} \le \hat{\bm{u}}(z) - \hat{\bm{u}}_+ \le \bm{K}_c e^{\lambda_2 z},
    & (z \ge 0), f_u(\hat{\bm{u}}_+) > 0, c > 2\sqrt{f_u(\hat{\bm{u}}_+)}, \\
    \bm{k}_{\epsilon} e^{(\lambda_2 - \epsilon)z} \le \hat{\bm{u}}(z) - \hat{\bm{u}}_+ \le \bm{K}_{\epsilon} e^{(\lambda_2 + \epsilon)z},
    & (z \ge 0), f_u(\hat{\bm{u}}_+) > 0, c = 2\sqrt{f_u(\hat{\bm{u}}_+)}, \\
    \bm{m}_c e^{\Tilde{\lambda}_1 z} \le \hat{\bm{u}}_- - \hat{\bm{u}}(z) \le \bm{M}_c e^{\Tilde{\lambda}_1 z},
    & (z \le 0), \frac{u^*}{c} < \frac{-c + \sqrt{c^2 + 2\mu_2}}{2}, \\
    \bm{m_{\epsilon}}e^{(\Tilde{\lambda}_1 + \epsilon) z} \le \hat{\bm{u}}_- - \hat{\bm{u}}(z) \le \bm{M_{\epsilon}}e^{(\Tilde{\lambda}_1 - \epsilon) z}
    & (z \le 0), \frac{u^*}{c} = \frac{-c + \sqrt{c^2 + 2\mu_2}}{2}, \\
    \bm{m_{\epsilon}}e^{(\Tilde{\lambda}_2 + \epsilon) z} \le \hat{\bm{u}}_- - \hat{\bm{u}}(z) \le \bm{M_{\epsilon}}e^{(\Tilde{\lambda}_2 - \epsilon) z}
    & (z \le 0), \frac{-c + \sqrt{c^2 + 2\mu_2}}{2} < \frac{u^*}{c}, \mu_1 = \mu_2, \\
    \bm{m}_c e^{\Tilde{\lambda}_2 z} \le \hat{\bm{u}}_- - \hat{\bm{u}}(z) \le \bm{M}_c e^{\Tilde{\lambda}_2 z},
    & (z \le 0),  \frac{-c + \sqrt{c^2 + 2\mu_2}}{2} < \frac{u^*}{c}, \mu_1 \neq \mu_2.
\end{cases}
\end{align*}

(b) There are nonpositive vectors $\bm{h}_c, \bm{H}_c, \bm{h}_{\epsilon}, \bm{H}_{\epsilon}$, nonnegative vectors $\bm{n}_c, \bm{N}_c, \bm{n}_{\epsilon}, \bm{N}_{\epsilon}$, negative numbers $\lambda_5, \lambda_3, \lambda_2$, and positive numbers $\epsilon, \Tilde{\lambda}_1, \Tilde{\lambda}_2$ such that the nonzero elements of $\bm{h}_{\epsilon}$ and $\bm{N}_{\epsilon}$ respectively go to $-\infty$ and $\infty$ as $\epsilon \to 0$, and
\begin{align*}
    \begin{cases}
    \bm{h}_c e^{\lambda_5 z} \le \hat{\bm{u}}'(z) \le \bm{H}_c e^{\lambda_5 z},
    & (z \ge 0), f_u(\hat{\bm{u}}_+) < g_v(\hat{\bm{u}}_+), \\
    \bm{h}_{\epsilon} e^{(\lambda_5 + \epsilon)z} \le \hat{\bm{u}}'(z) \le \bm{H}_{\epsilon} e^{(\lambda_5 - \epsilon)z},
    & (z \ge 0), f_u(\hat{\bm{u}}_+) = g_v(\hat{\bm{u}}_+), \\
    \bm{h}_c e^{\lambda_3 z} \le \hat{\bm{u}}'(z) \le \bm{H}_c e^{\lambda_3 z},
    & (z \ge 0), g_v(\hat{\bm{u}}_+) < f_u(\hat{\bm{u}}_+) < 0, \\
    \bm{h}_c e^{\lambda_2 z} \le \hat{\bm{u}}'(z) \le \bm{H}_c e^{\lambda_2 z},
    & (z \ge 0), f_u(\hat{\bm{u}}_+) > 0, c > 2\sqrt{f_u(\hat{\bm{u}}_+)}, \\
    \bm{h}_{\epsilon} e^{(\lambda_2 + \epsilon)z} \le \hat{\bm{u}}'(z) \le \bm{H}_{\epsilon} e^{(\lambda_2 - \epsilon)z},
    & (z \ge 0), f_u(\hat{\bm{u}}_+) > 0, c = 2\sqrt{f_u(\hat{\bm{u}}_+)}, \\
    \bm{n}_c e^{\Tilde{\lambda}_1 z} \le - \hat{\bm{u}}'(z) \le \bm{N}_c e^{\Tilde{\lambda}_1 z},
    & (z \le 0), \frac{u^*}{c} < \frac{-c + \sqrt{c^2 + 2\mu_2}}{2}, \\
    \bm{n_{\epsilon}}e^{(\Tilde{\lambda}_1 + \epsilon) z} \le - \hat{\bm{u}}'(z) \le \bm{N_{\epsilon}}e^{(\Tilde{\lambda}_1 - \epsilon) z}
    & (z \le 0), \frac{u^*}{c} = \frac{-c + \sqrt{c^2 + 2\mu_2}}{2}, \\
    \bm{n_{\epsilon}}e^{(\Tilde{\lambda}_2 + \epsilon) z} \le - \hat{\bm{u}}'(z) \le \bm{N_{\epsilon}}e^{(\Tilde{\lambda}_2 - \epsilon) z}
    & (z \le 0), \frac{-c + \sqrt{c^2 + 2\mu_2}}{2} < \frac{u^*}{c}, \mu_1 = \mu_2, \\
    \bm{n}_c e^{\Tilde{\lambda}_2 z} \le - \hat{\bm{u}}'(z) \le \bm{N}_c e^{\Tilde{\lambda}_2 z},
    & (z \le 0),  \frac{-c + \sqrt{c^2 + 2\mu_2}}{2} < \frac{u^*}{c}, \mu_1 \neq \mu_2.
\end{cases}
\end{align*}

\end{proposition}

\subsubsection{Stability}
\label{sec:Stability}
We now present the stability results for the bistable and monostable regions separately. We will see that the stability in the monostable regions allows for perturbations from a family of Banach spaces embedded with exponentially weighted norms, which requires a faster decay rate at infinity compared to the bistable case.

\begin{theorem}[Bistable]
\label{thm:bistable-stability}
    Let $k_0, k_1 \ge 0, k_2 > 0$, or $k_0, k_1 > 0, k_2 \ge 0$. For $k_2(k_0,k_1) < k_2 < k_2(k_1,P)$, the traveling waves whose existence is guaranteed by Theorem ~\eqref{thm:bistable-existence} are asymptotically stable with shift in the sup-norm $\lVert \cdot \rVert$ defined by $\lVert \bm{u} \rVert = \sup_{x \in \mathbb{R}} |\bm{u}(x)|$ for all vector-valued, continuous and bounded function $\bm{u}$ defined on $\mathbb{R}$.
\end{theorem}

\begin{proof}
    Let us denote the monotone wave solution of system ~\eqref{Wang_Equation3} by $\hat{\bm{u}}(z)$. Since the system is monotone for $k_0, k_1, k_2 \ge 0$, its Jacobian $DF(\hat{\bm{u}}(z))$ has nonnegative off-diagonal elements. One can also check that it is functionally irreducible so long as $\Gamma, \beta \neq 0$, and either $k_0, k_1 \ge 0, k_2 > 0$ or $k_0, k_1 > 0, k_2 \ge 0$. For the eigenvalue conditions, recall from the proof for Theorem ~\eqref{thm:bistable-existence} that matrices $DF(\hat{\bm{u}}_-)$ and $DF(\hat{\bm{u}}_+)$ have all their eigenvalues in the left half-plane, where, again, $\hat{\bm{u}}_-$ and $\hat{\bm{u}}_+$ are the two steady states that $\hat{\bm{u}}(z)$ approaches as $z \to -\infty$ and $z \to +\infty$, respectively. The last condition $\lim_{z \to \infty} \hat{\bm{u}}'(z) = 0$ results from the monotonicity assumption $\hat{\bm{u}}'(z) < 0$ and the definition $\lim_{z \to \infty} \hat{\bm{u}}(z) = \hat{\bm{u}}_+$. Therefore, the wave $\hat{\bm{u}}(z)$ is asymptotically stable with a shift in the sup-norm by Theorem 4.1 in chapter 5 of \cite{Wang_Volpert_1994}.
\end{proof}

\begin{theorem}[Monostable]
\label{thm:monostable-stability}
    Let $k_0, k_1 \ge 0, k_2 > 0$, or $k_0, k_1 > 0, k_2 \ge 0$. For $k_2 > k_2(k_0,k_1) $ and $k_2 > k_2(k_1,P)$, the traveling waves whose existence is guaranteed by Theorem ~\eqref{thm:monostable-existence} are asymptotically stable in the norm $\lVert \cdot \rVert _{\sigma}$ for some $\sigma \ge 0$. The norm $\lVert \cdot \rVert _{\sigma}$ is  defined by $\lVert \bm{u} \rVert _{\sigma} = \sup_{x \in \mathbb{R}} |\bm{u}(x)(1 + e^{\sigma x})|$ for all vector-valued, continuous and bounded function $\bm{u}$ defined on $\mathbb{R}$ such that $\lim_{|x| \to \infty} \bm{u}(x)(1 + e^{\sigma x}) = 0$.
\end{theorem}

\begin{proof}
    Again, the Jacobian $DF(\hat{\bm{u}}(z))$ is functionally irreducible in the given region where either $k_0, k_1 \ge 0, k_2 > 0$ or $k_0, k_1 > 0, k_2 \ge 0$. We have checked in the proof for ~\eqref{thm:monostable-existence} that $DF(\hat{\bm{u}}_-)$ has all its eigenvalues in the left half-plane. Now, we want to find a nonnegative $\sigma$ that shifts all eigenvalues of the matrix $A \sigma^2 - C\sigma + DF(\hat{\bm{u}}_+)$ to the left half-plane. The eigenvalues of this matrix are
    \[
    \lambda_1 = \sigma^2 - c\sigma + f_{u}(\hat{\bm{u}}_+), \quad
    \lambda_2 = \sigma^2 - c\sigma + g_{v}(\hat{\bm{u}}_+), \quad
    \lambda_3 = -c\sigma.
    \]
    Since $g_{v}(\hat{\bm{u}}_+) < f_{u}(\hat{\bm{u}}_+)$ always holds in the monostable regions, we only need to require that $\lambda_1 < 0$. This is satisfied as long as
    \[
    \frac{c - \sqrt{c^2 - 4f_u(\hat{\bm{u}}_+)}}{2} < \sigma < \frac{c + \sqrt{c^2 - 4f_u(\hat{\bm{u}}_+)}}{2},
    \]
    with $c > 2\sqrt{f_u(\hat{\bm{u}}_+)}$.
    
    The last part of the proof is to show that $\hat{\bm{u}}'(z)e^{\sigma z}$ diverges as $z \to \infty$. By part (b) of Proposition \ref{prop:asymp-decay-rates-triv}, we have
    \[
    \bm{h}_c e^{(\lambda_2+\sigma)z} \le \hat{\bm{u}}'(z) e^{\sigma z} \le \bm{H}_c e^{(\lambda_2+\sigma)z}.
    \]
    Since $\lambda_2 = \dfrac{-c + \sqrt{c^2 - 4f_u(\hat{\bm{u}}_+)}}{2}$, using the lower bound for $\sigma$ found above, we see that $\lambda_2 + \sigma > 0$ and conclude divergence.

    By Theorem 4.1 in chapter 5 of \cite{Wang_Volpert_1994}, the waves are asymptotically stable in the norm $\lVert \cdot \rVert _{\sigma}$.
\end{proof}

The waves in the monostable region in the monotone system in \cite{Wang_Yang_2021} should be asymptotically stable without a shift due to the divergence condition.

%%%%%%%%%%%%%%%%%%%%%%%%%%%%%%%%%%%%%%%%%%%%%%%%%%%%%%%%%%%%%%%%%%%%%%%%%%%%%%%%%%%%%%%%%%%%%%%%%%%%
%%%%%%%%%%%%%%%%%%%%%%%%%%%%%%%%%%%%%%%%%%%%%%%%%%%%%%%%%%%%%%%%%%%%%%%%%%%%%%%%%%%%%%%%%%%%%%%%%%%%
%%%%%%%%%%%%%%%%%%%%%%%%%%%%%%%%%%%%%%%%%%%%%%%%%%%%%%%%%%%%%%%%%%%%%%%%%%%%%%%%%%%%%%%%%%%%%%%%%%%%
%%%%%%%%%%%%%%%%%%%%%%%%%%%%%%%%%%%%%%%%%%%%%%%%%%%%%%%%%%%%%%%%%%%%%%%%%%%%%%%%%%%%%%%%%%%%%%%%%%%%
%%%%%%%%%%%%%%%%%%%%%%%%%%%%%%%%%%%%%%%%%%%%%%%%%%%%%%%%%%%%%%%%%%%%%%%%%%%%%%%%%%%%%%%%%%%%%%%%%%%%
%%%%%%%%%%%%%%%%%%%%%%%%%%%%%%%%%%%%%%%%%%%%%%%%%%%%%%%%%%%%%%%%%%%%%%%%%%%%%%%%%%%%%%%%%%%%%%%%%%%%
%%%%%%%%%%%%%%%%%%%%%%%%%%%%%%%%%%%%%%%%%%%%%%%%%%%%%%%%%%%%%%%%%%%%%%%%%%%%%%%%%%%%%%%%%%%%%%%%%%%%
%%%%%%%%%%%%%%%%%%%%%%%%%%%%%%%%%%%%%%%%%%%%%%%%%%%%%%%%%%%%%%%%%%%%%%%%%%%%%%%%%%%%%%%%%%%%%%%%%%%%
%%%%%%%%%%%%%%%%%%%%%%%%%%%%%%%%%%%%%%%%%%%%%%%%%%%%%%%%%%%%%%%%%%%%%%%%%%%%%%%%%%%%%%%%%%%%%%%%%%%%
%%%%%%%%%%%%%%%%%%%%%%%%%%%%%%%%%%%%%%%%%%%%%%%%%%%%%%%%%%%%%%%%%%%%%%%%%%%%%%%%%%%%%%%%%%%%%%%%%%%%

\section{Model outcomes from numerical experiments}
In this section, we explore parameter space to study important features that are difficult to analyze theoretically. Our emphasis is on studying the effect of $k_0, k_1$, and $k_2$ on these features; refer to Table \ref{table:k} for a refresher on the physical interpretation of these parameters. The specific values for all parameters employed in these simulations are detailed in Table \ref{table:ModelVals_NumericalSim} within Appendix \ref{app:tables}.

\subsection{Quantitative outcomes}

We present various metrics to assess the outcomes of simulations conducted under different conditions. While many measurements can quantify the dynamics of $u$, $v$, and $P$ in a system initiated with predominantly localized protests, a moderate level of social tension everywhere, and no police presence, we intended to provide a combination of local and global measurements. It is important to note that the speed of the traveling wave remains independent of the parameters. Additionally, by $t=15$, we observed that the profiles for $u$ with the parameters outlined in Table \ref{table:ModelVals_NumericalSim} had fully formed. 
The profile for $v$ required more time to form, and by $t=100$, most wave profiles had been reached.

As mentioned in the introduction, a wave of protests is initiated when significant social tension is coupled with a triggering event. Our numerical simulations assume an initial social tension of one uniformly across space. The triggering event is assumed to occur at the leftmost boundary of the domain, resulting in an initial condition of $u(x,0)=e^{-5x}$. It is natural to conceptualize this event as a localized function with a decaying impact relative to the distance from the triggering event location. It is well-established that exponentially decaying initial data gives rise to traveling wave solutions. Furthermore, we assume an initial condition of zero throughout the domain for police presence.

\subsubsection{Invading protest levels}

A metric of interest for individuals is the degree of activity and social tension encountered as the protest wavefront passes through. Specifically, the invading protest wavefront serves as a distinct boundary separating no protest activity immediately to its right from a positive level of activity on it. Given the constant wave speed across all simulations, we observe a near-identical nature of this interface among simulations. To the left of this point, the behavior diverges from simulation to simulation for varying parameter values (refer to Figures \ref{Wang_Figure29}-\ref{Wang_Figure31}). The $x$ coordinate where the simulation-to-simulation behavior diverges is easily discernible in simulations with negative $k_0$ values, identified as the local maximum of the interface (e.g., in Figure \ref{Wang_Figure29}, it occurs at approximately $x=454$). We denote this $x$ coordinate as the location of the invading interface because it marks the trailing edge of a sharp interface distinguishing the region unexposed to protesting activity from the exposed region. In Figures \ref{Wang_Figure26}-\ref{Wang_Figure28}, we present the level of invading protest activity (i.e., the level at the invading interface) for various values of $k_0$, $k_1$, and $k_2$, demonstrating slight variations. This also underscores the point that the invading level does not necessarily correspond to the highest point of the traveling wave, as some profiles continue to grow beyond the invading interface (e.g., the profiles in Figure \ref{Wang_Figure29}).

As depicted in Figures \ref{Wang_Figure26}-\ref{Wang_Figure28}, an increase in any of the parameters $k_0$, $k_1$, or $k_2$ results in a higher level of $u$ at the invading interface. In other words, the intensity of protesting activity strictly and monotonically rises as a function of any of the parameters $k_i$. While the final level of protesting activity after the front has passed cannot be predicted \textit{a priori} solely based on the level of protest activity at the invading interface, it holds true that in regions unexposed to protesting, the highest levels of protesting activity at the invading interface occur when the values of $k_i$ are at their maximum. An intriguing observation is that $k_1$ appears to exert a greater influence on the activity level than $k_2$. In Figures \ref{Wang_Figure26}-\ref{Wang_Figure28}, we note that the bottom right quadrants, where $k_1<0$ and $k_2>0$, exhibit less activity compared to scenarios where $k_1>0$. 

\begin{figure} [H]
     \centering
     \begin{subfigure}[t]{0.32\textwidth}
         \includegraphics[width=\textwidth]{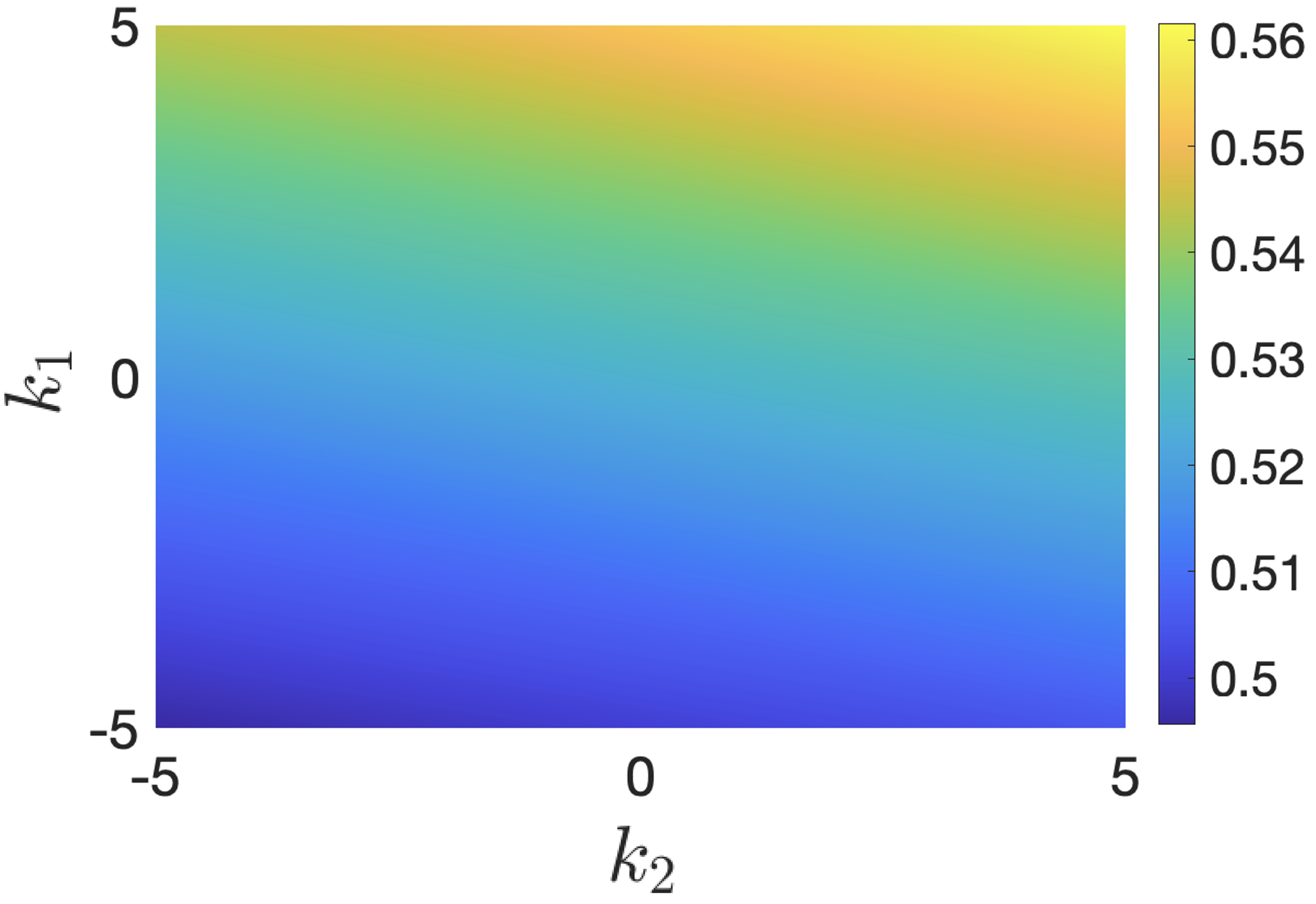}
         \caption{$k_0=-1$}
         \label{Wang_Figure26}
    \end{subfigure}
    \begin{subfigure}[t]{0.32\textwidth}
        \includegraphics[width=\textwidth]{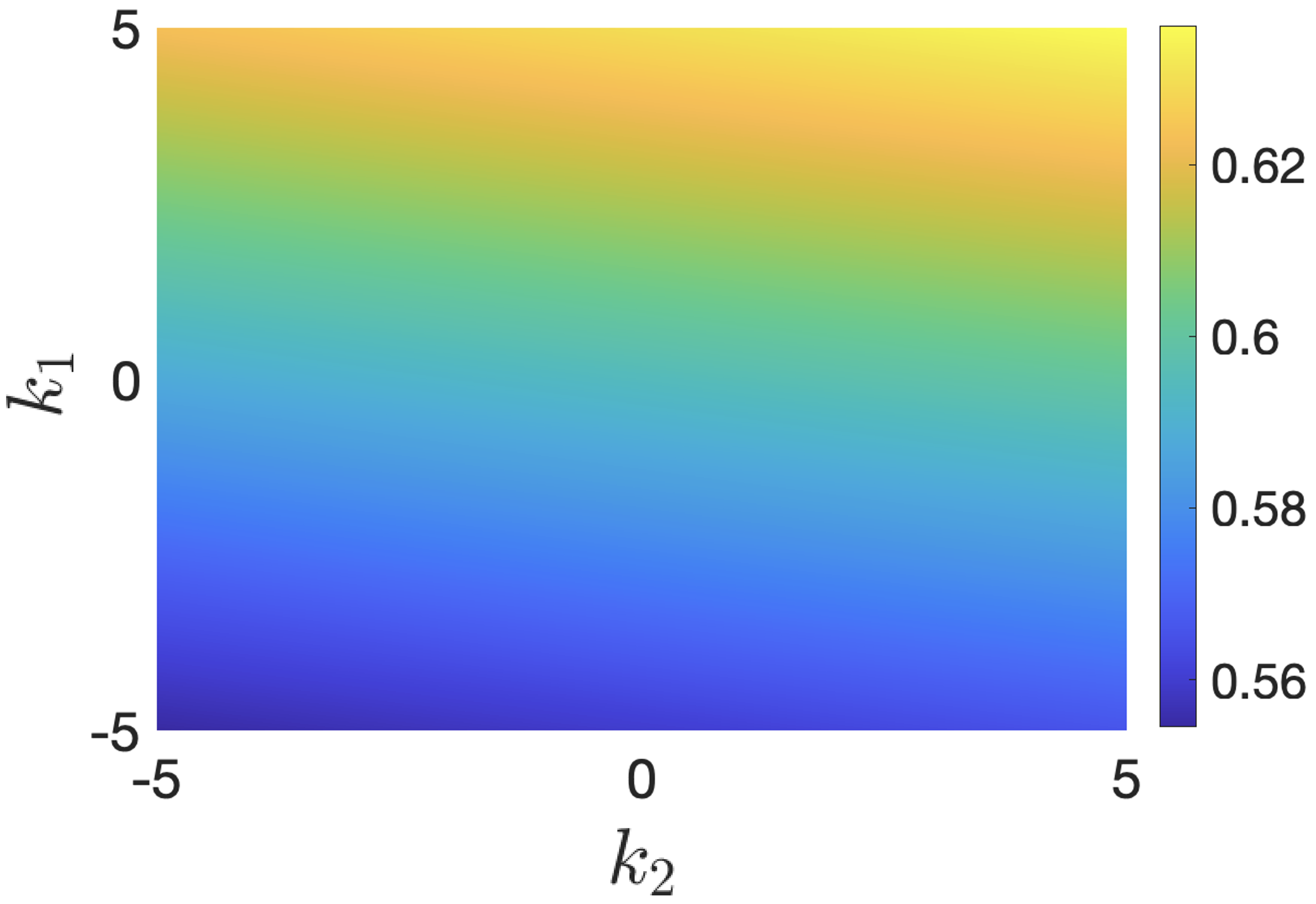}
        \caption{$k_0=0$}
        \label{Wang_Figure27}
    \end{subfigure}
    \begin{subfigure}[t]{0.32\textwidth}
        \includegraphics[width=\textwidth]{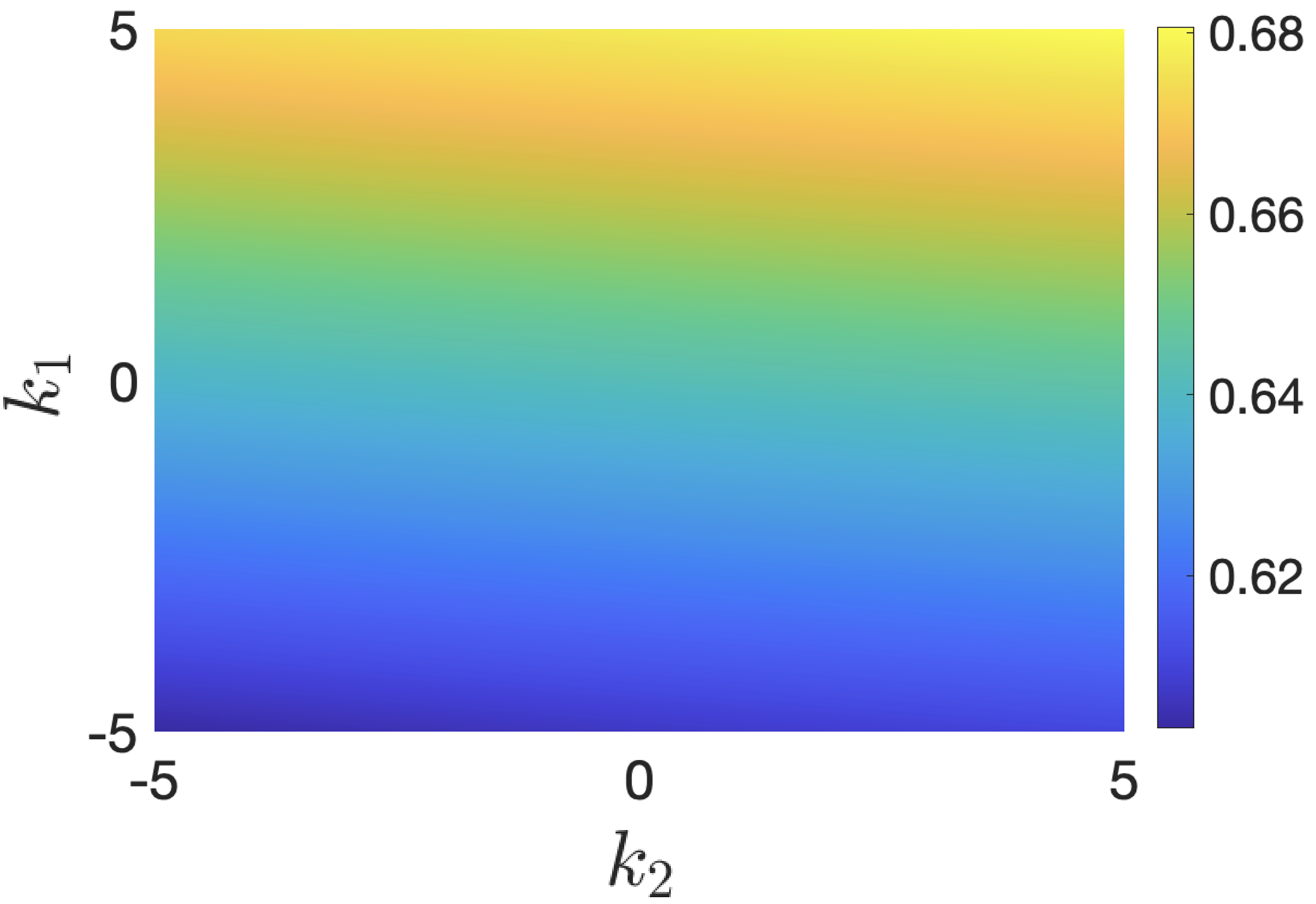}
        \caption{$k_0=1$}
        \label{Wang_Figure28}
    \end{subfigure}
    
    \begin{subfigure}[t]{0.27\textwidth}
         \includegraphics[width=\textwidth]{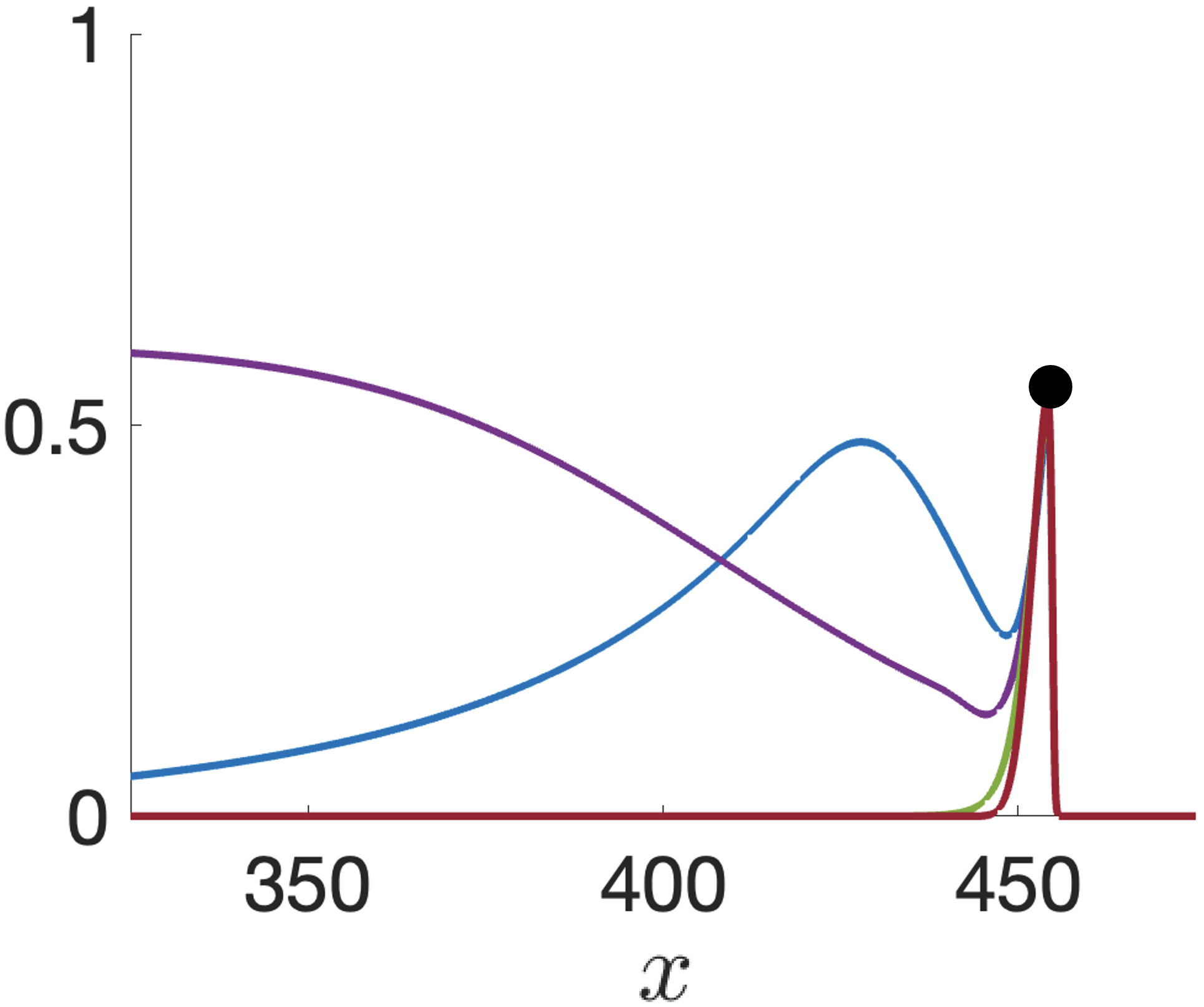}
         \caption{$k_0=-1$}
         \label{Wang_Figure29}
    \end{subfigure}
    \begin{subfigure}[t]{0.27\textwidth}
        \includegraphics[width=\textwidth]{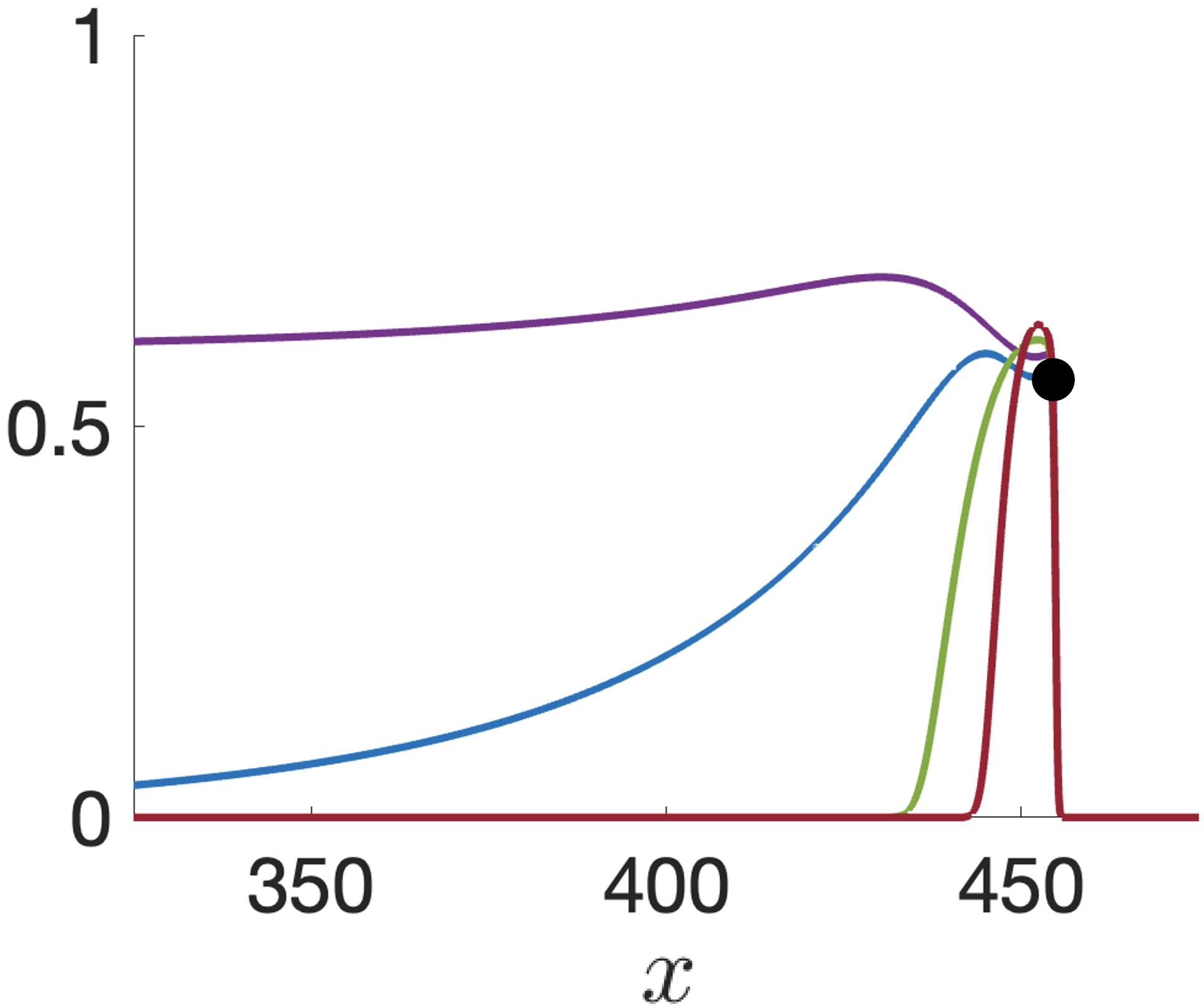}
        \caption{$k_0=0$}
        \label{Wang_Figure30}
    \end{subfigure}
    \begin{subfigure}[t]{0.27\textwidth}
        \includegraphics[width=\textwidth]{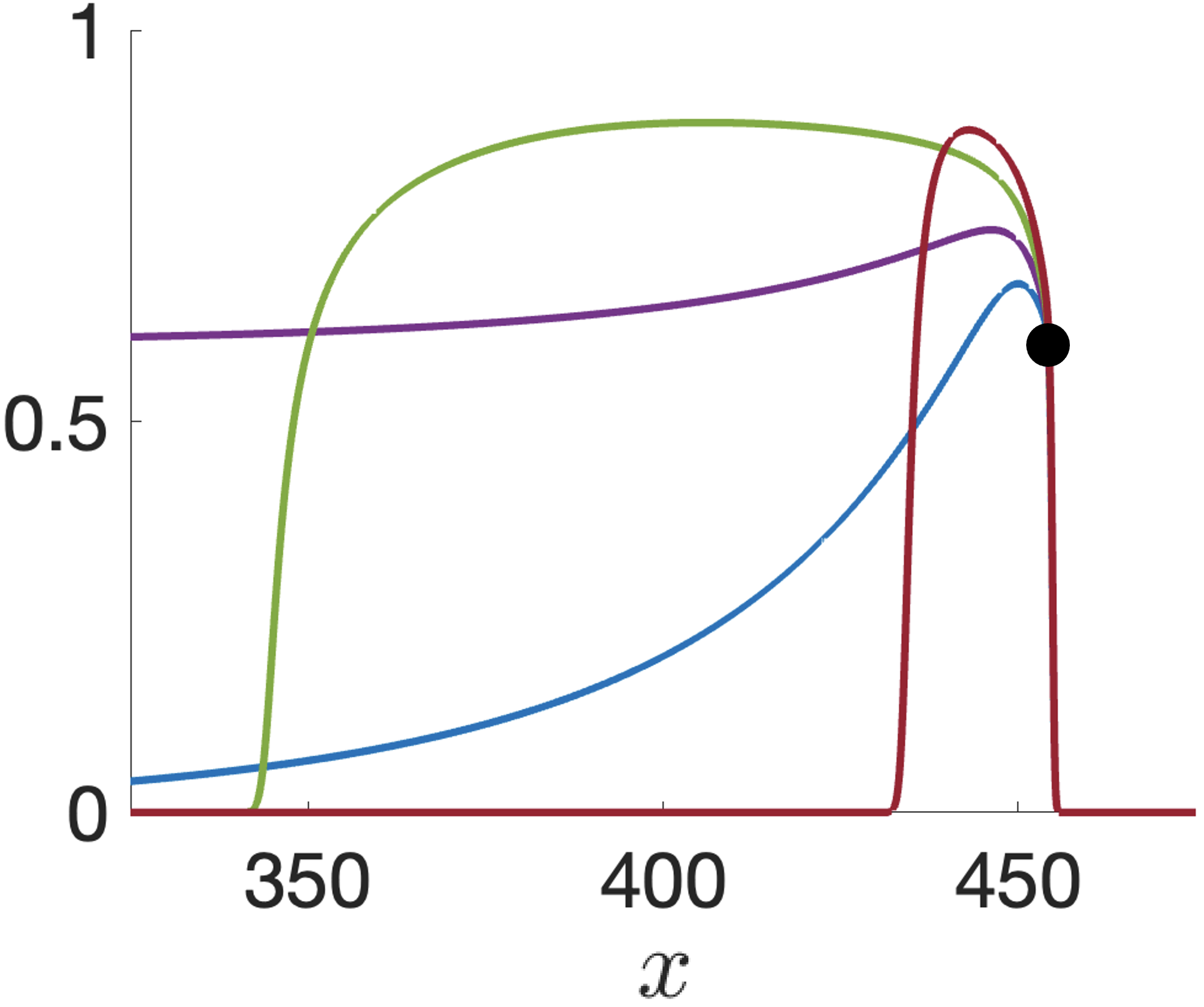}
        \caption{$k_0=1$}
        \label{Wang_Figure31}
    \end{subfigure}
    \begin{subfigure}[t]{0.15\textwidth}
        \includegraphics[width=\textwidth]{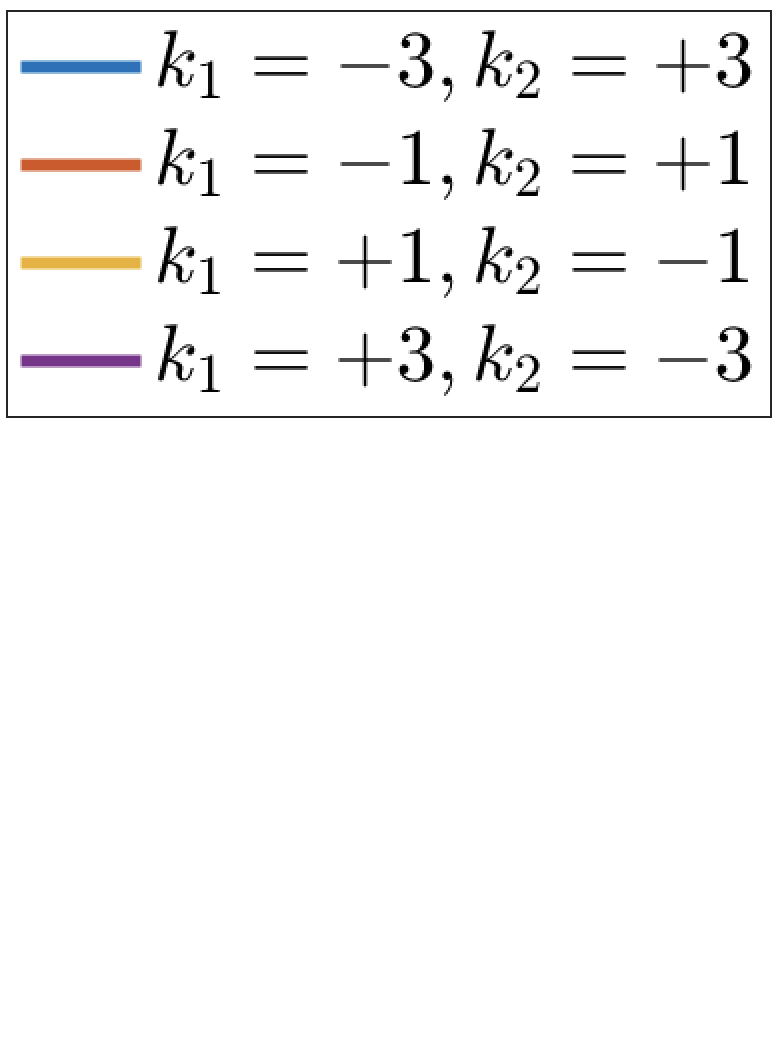}
    \label{Wang_Figure32}
    \end{subfigure}
    \caption{(a)-(c) Heatmaps for the value of the invading protest levels for simulations run with various combinations of $k_1$ and $k_2$. (d)-(f) Sample profiles from different regions from the figures (a)-(c) with the legend providing the $k_1$ and $k_2$ coordinates for the respective profiles. The black dot signifies the location of the invading interface. Note: in (d) the green curve is plotted but is mostly covered by the red curve.}
    \label{Wang_Figure33}
\end{figure}

\subsubsection{Residual levels left after the front interface passes}

Another noteworthy metric is the residual level of protest activity and social tension that persists after the interface traverses. Based on the prior sections' analysis, we know that the constant state at the fixed $x$ coordinate must have either $u$ approaching zero or $P$ approaching one. In other words, the protest traveling wave can result in a sustained nonzero level of protest activity that no longer alters social tension levels, or it may exhibit a near-pulse profile of activity that essentially diminishes after passing.

We have measured the residual levels of protest activity, tension, and police presence across various values of $k_0$, $k_1$, and $k_2$, illustrating them in a heatmap presented in Figure \ref{Wang_Figure43}. It is crucial to highlight that in all experiments, when $k_2<0$, indicative of a negotiated management strategy, the residual values are either zero or close to zero, even when $k_1>0$. By examining these values, we gain insights into the contrasting roles that police presence can play in influencing protest activity. For instance, in scenarios where $k_2>0$, signifying an escalated force management strategy, we observe a positive level of activity remaining, even if police presence results in deterrence ($k_1<0$). Furthermore, there are specific regions where $k_1<0$ and $k_2>0$, showcasing a notable persistence of tension after the front interface has passed.  

From these observations, it becomes apparent that a positive $k_2$ value outweighs an inhibitory $k_1$ value ($k_1<0$), suggesting that if police presence raises social tension, it will lead to more protest activity even if it discourages protest participation. Moreover, the parameter $k_2$ appears sensitive for small absolute values but less so for large magnitudes concerning the expected residual protest activity after the wave interface passes. This implies that the extent to which police presence enhances social tension may not matter as much as the fact that it does enhance it; a positive level of protest activity is expected to persist for a fixed $k_1$ value. A similar trend is observed for $k_1$, where a positive $k_1$ value leads to approximately the same residual activity regardless of the value, whereas, for negative values, there are more gradual changes in residual activity as $k_1$ decreases (for a fixed $k_2>0$).

\begin{figure} [H]
    \centering
    \begin{subfigure}[t]{0.32\textwidth}
        \includegraphics[width=\textwidth]{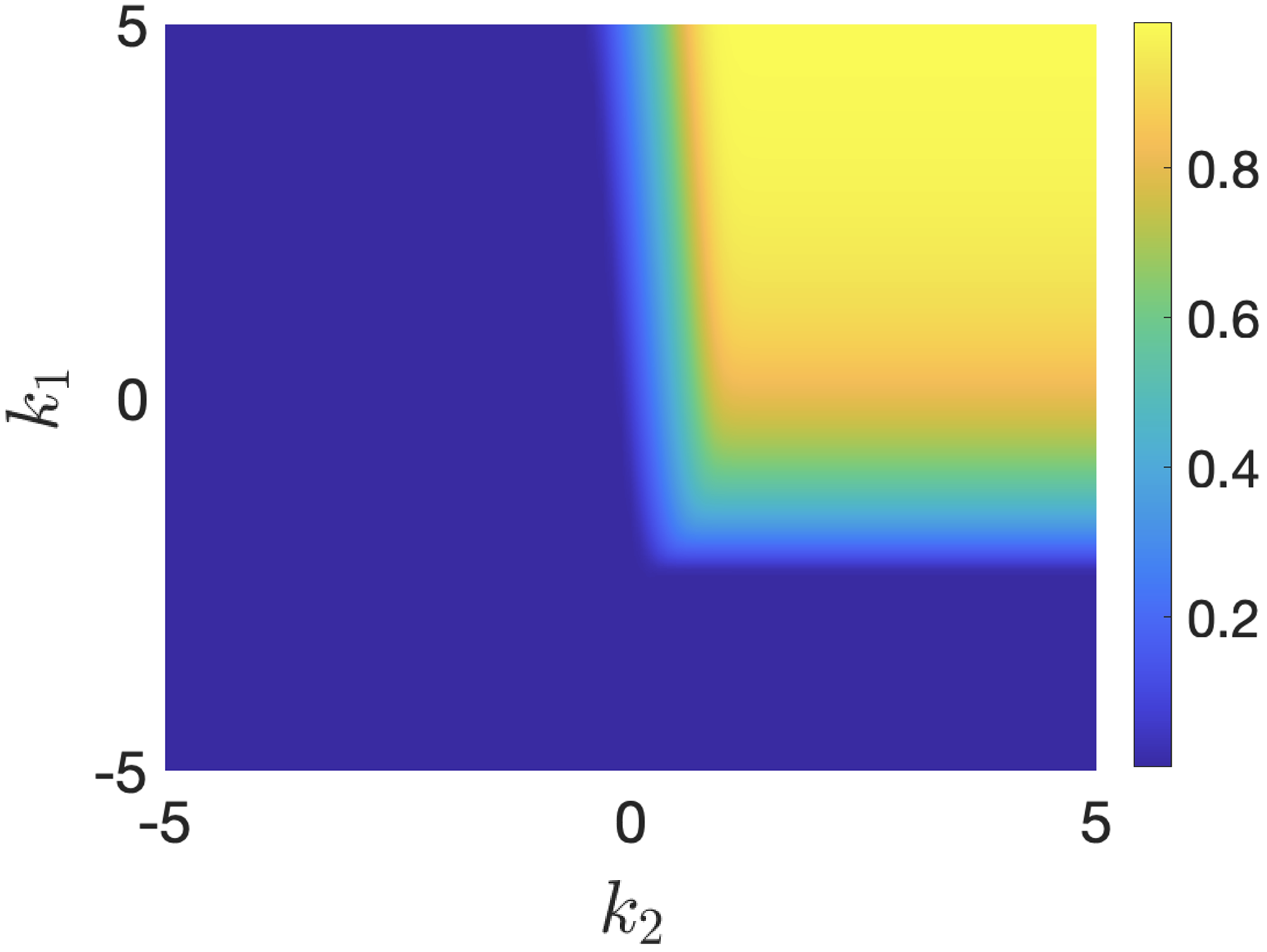}
        \caption{$u$ (with $k_0=-1$)}
        \label{Wang_Figure34}
    \end{subfigure}
    \begin{subfigure}[t]{0.32\textwidth}
        \includegraphics[width=\textwidth]{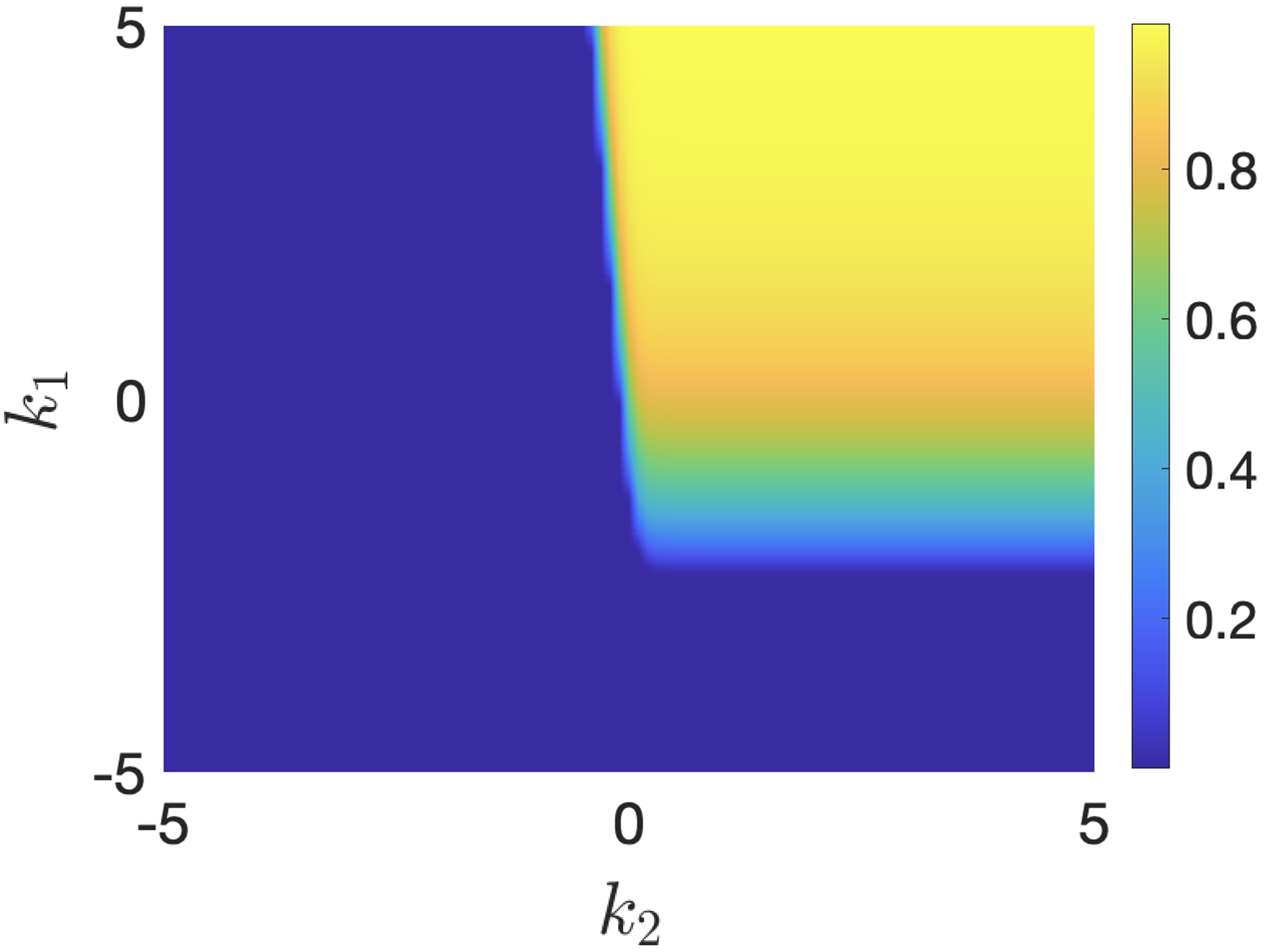}
        \caption{$u$ (with $k_0=0$)}
        \label{Wang_Figure35}
    \end{subfigure}
    \begin{subfigure}[t]{0.32\textwidth}
        \includegraphics[width=\textwidth]{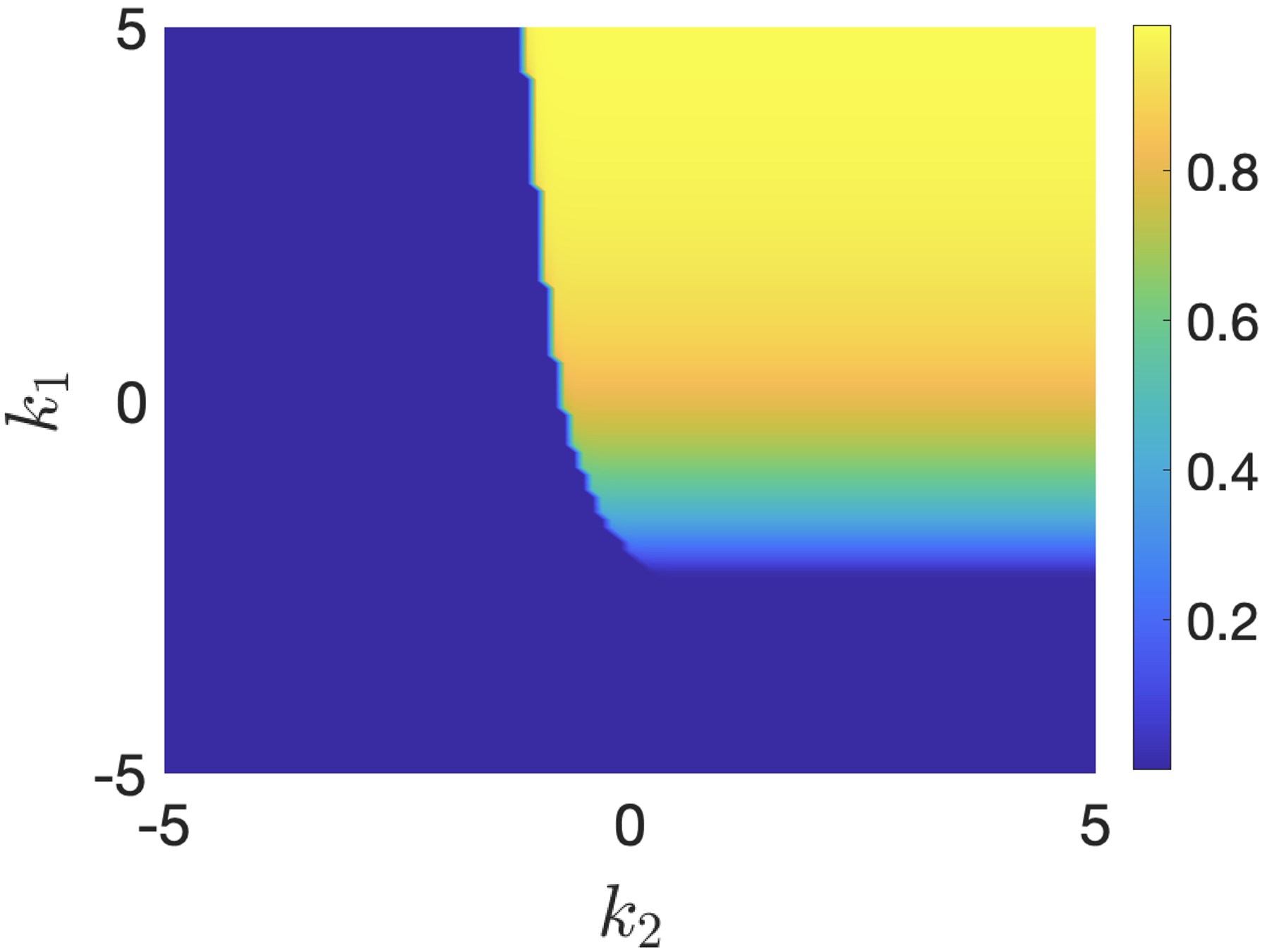}
        \caption{$u$ (with $k_0=+1$)}
        \label{Wang_Figure36}
    \end{subfigure}
    
    \begin{subfigure}[t]{0.32\textwidth}
        \includegraphics[width=\textwidth]{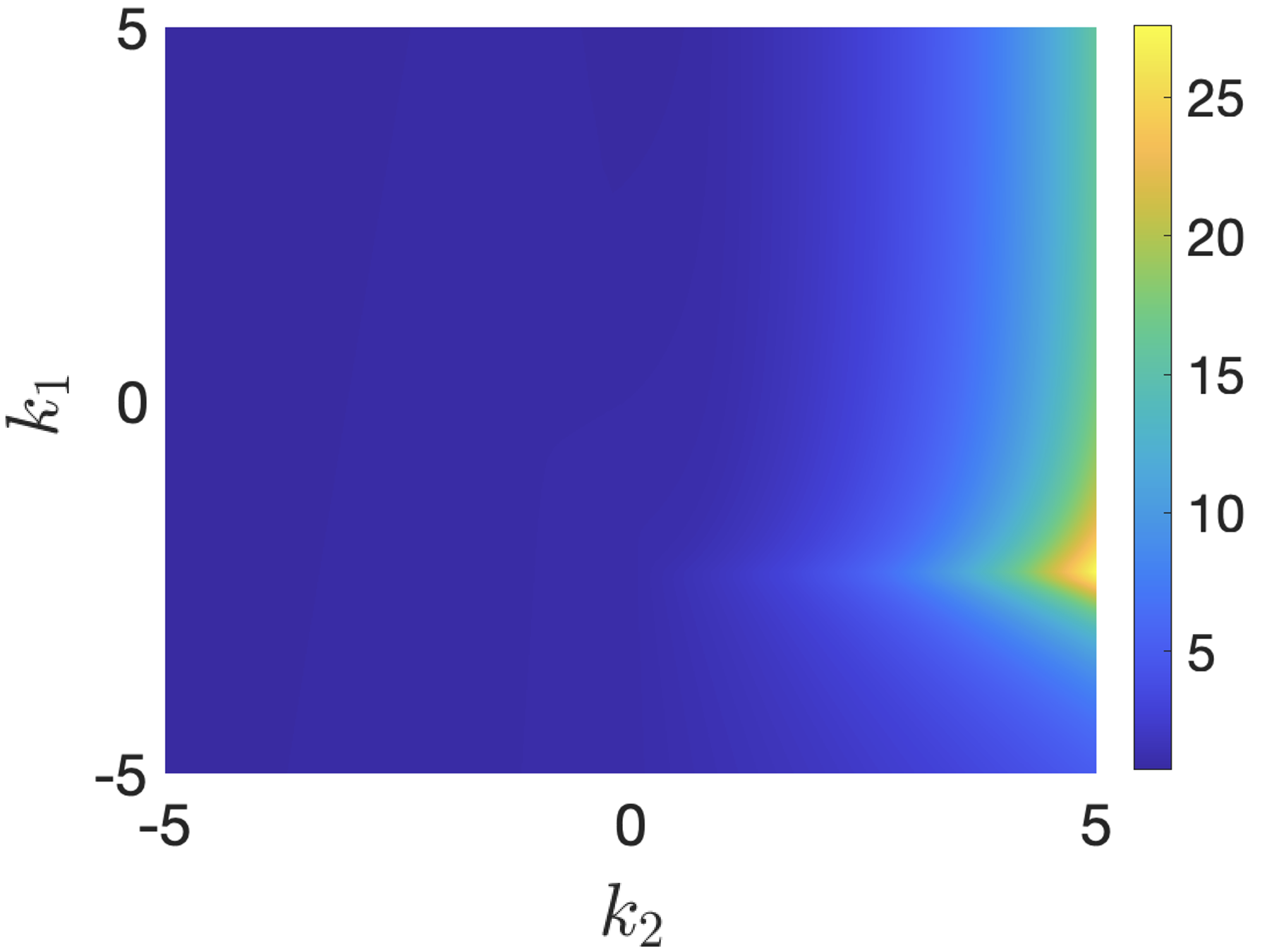}
        \caption{$v$ (with $k_0=-1$)}
        \label{Wang_Figure37}
    \end{subfigure}
    \begin{subfigure}[t]{0.32\textwidth}
        \includegraphics[width=\textwidth]{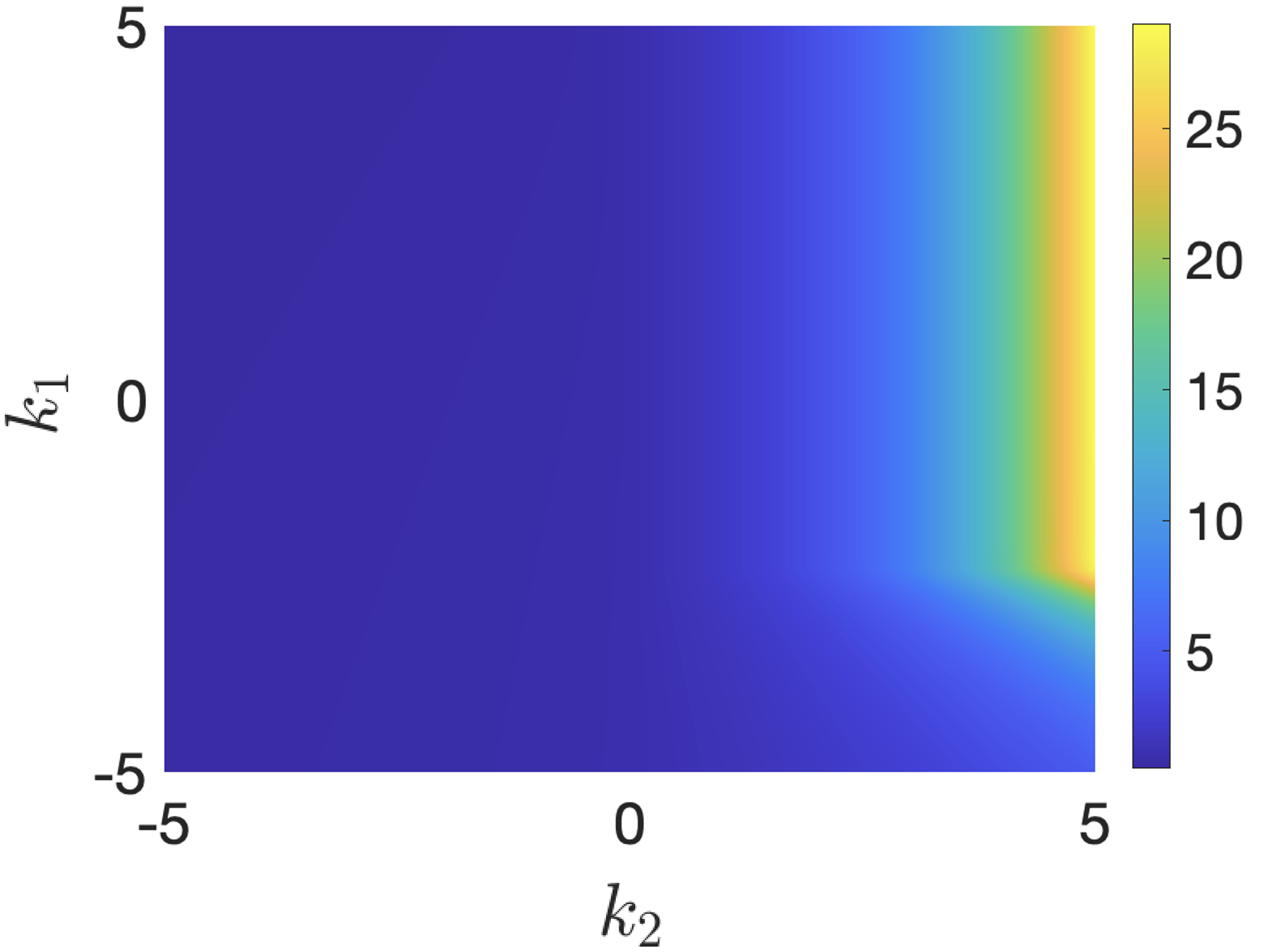}
        \caption{$v$ (with $k_0=0$)}
        \label{Wang_Figure38}
    \end{subfigure}
    \begin{subfigure}[t]{0.32\textwidth}
        \includegraphics[width=\textwidth]{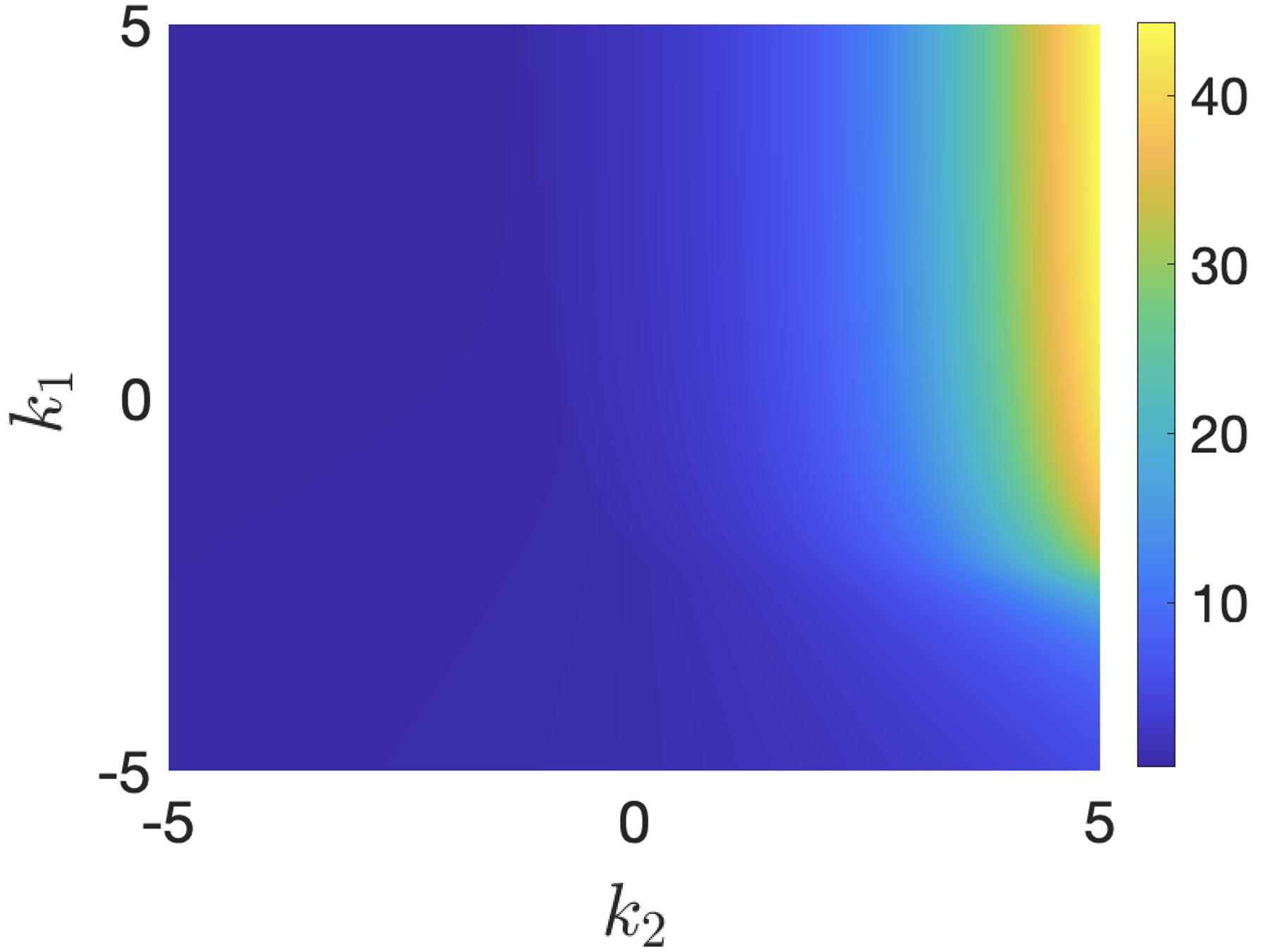}
        \caption{$v$ (with $k_0=+1$)}
        \label{Wang_Figure39}
    \end{subfigure}
    
    \begin{subfigure}[t]{0.32\textwidth}
        \includegraphics[width=\textwidth]{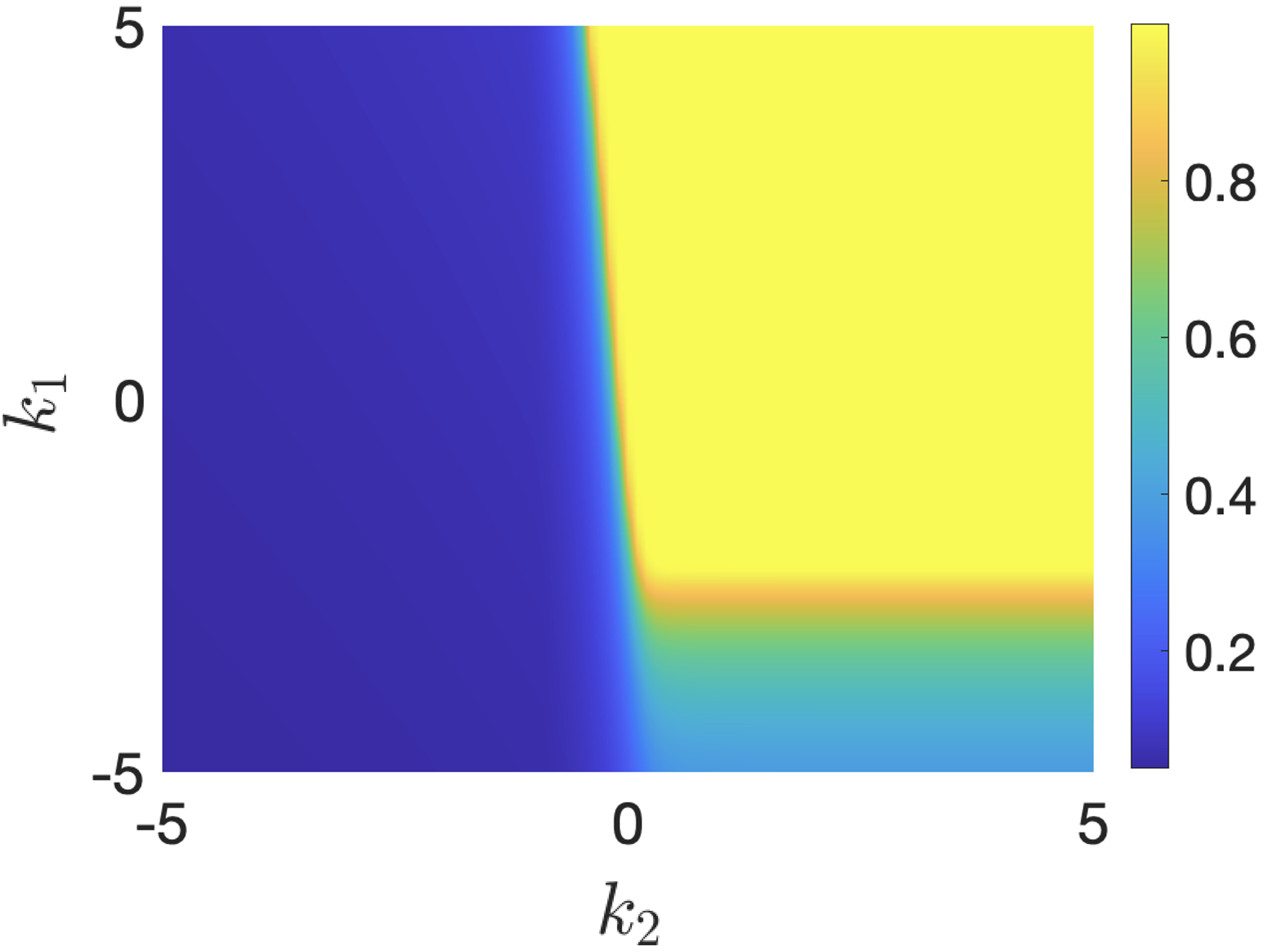}
        \caption{$P$ (with $k_0=-1$)}
        \label{Wang_Figure40}
    \end{subfigure}
    \begin{subfigure}[t]{0.32\textwidth}
        \includegraphics[width=\textwidth]{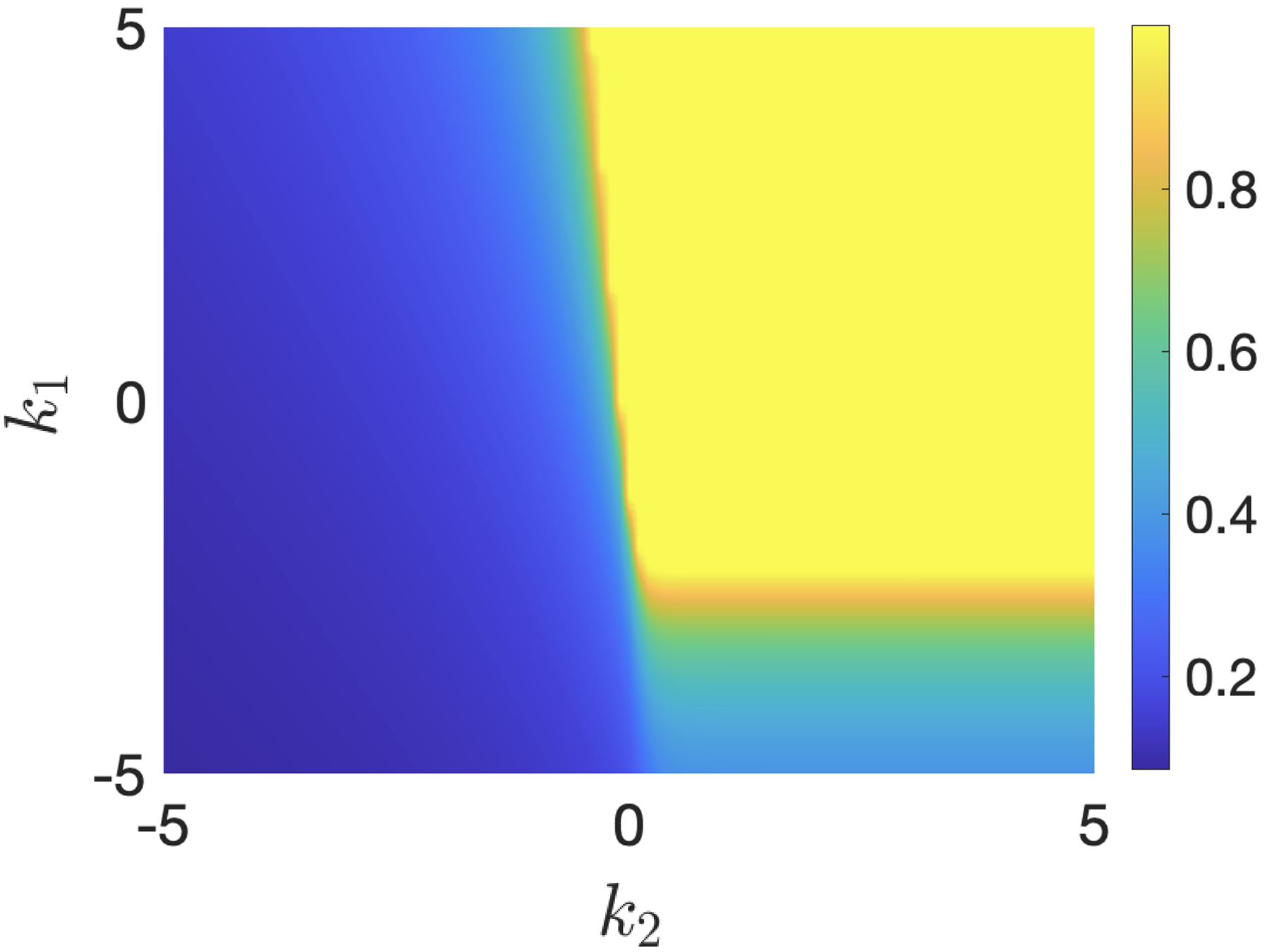}
        \caption{$P$ (with $k_0=0$)}
        \label{Wang_Figure41}
    \end{subfigure}
    \begin{subfigure}[t]{0.32\textwidth}
        \includegraphics[width=\textwidth]{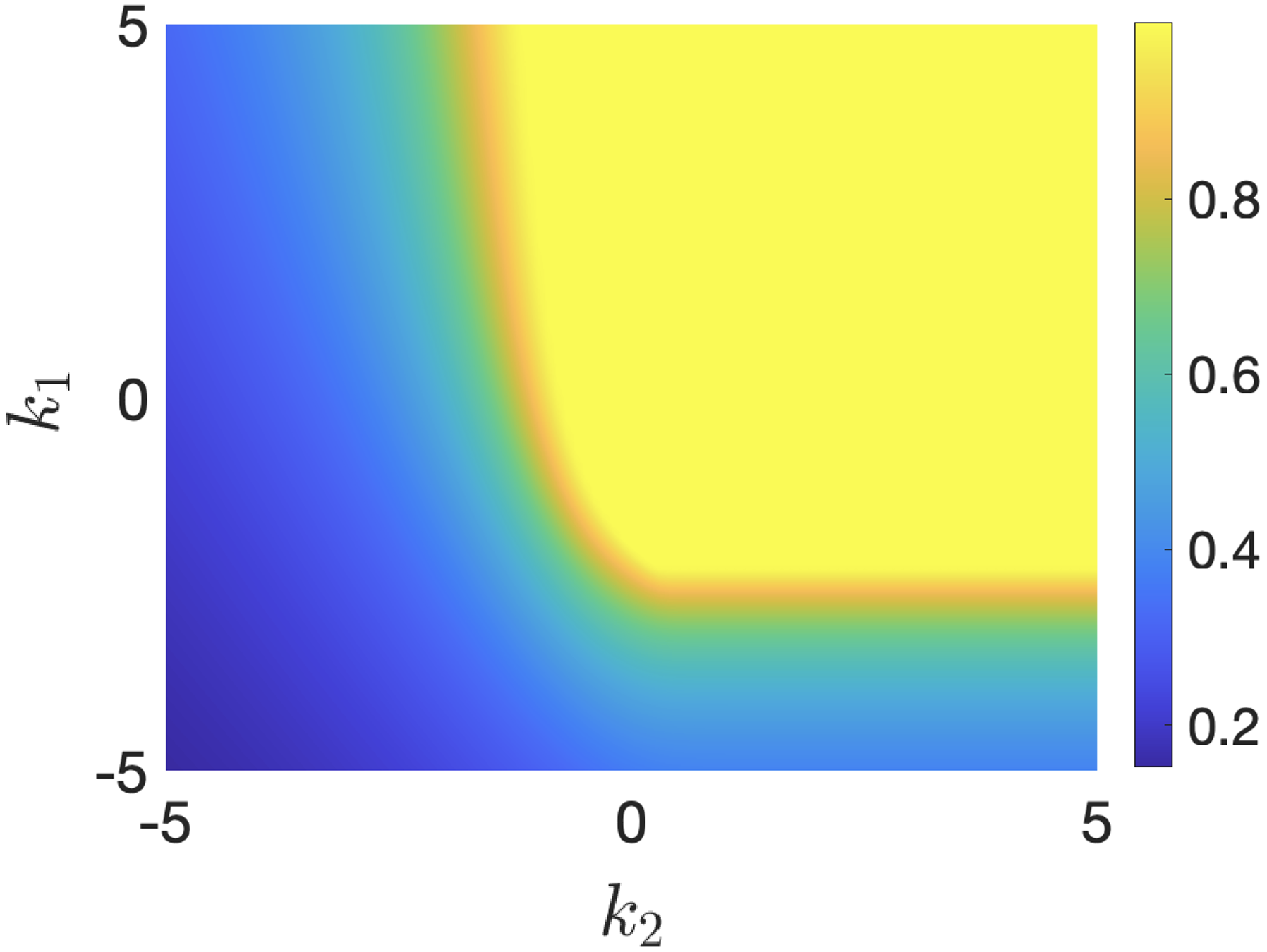}
        \caption{$P$ (with $k_0=+1$)}
        \label{Wang_Figure42}
    \end{subfigure}
    \caption{Residual levels of protest activity (a)-(c), social tension (d)-(f), and police presence (g)-(i) for various $k_0$, $k_1$, and $k_2$ parameter combinations when $t=100$.}
    \label{Wang_Figure43}
\end{figure}

\begin{figure} [H]
    \centering
    \begin{subfigure}[t]{0.4\textwidth}
        \includegraphics[width=\textwidth]{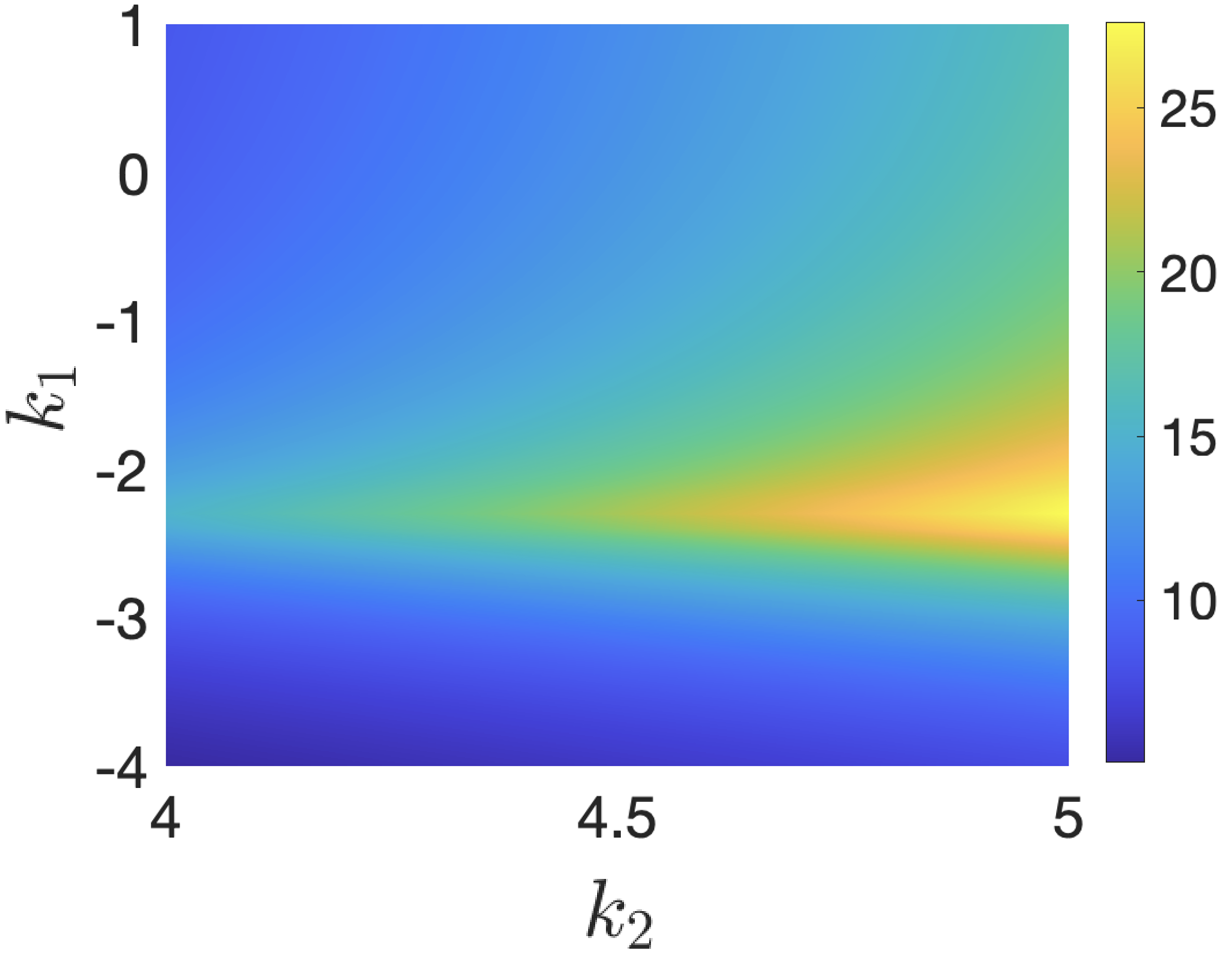}
        \caption{Zoom in of Figure \ref{Wang_Figure37}}
        \label{Wang_Figure44}
    \end{subfigure}
    \begin{subfigure}[t]{0.56\textwidth}
        \includegraphics[width=\textwidth]{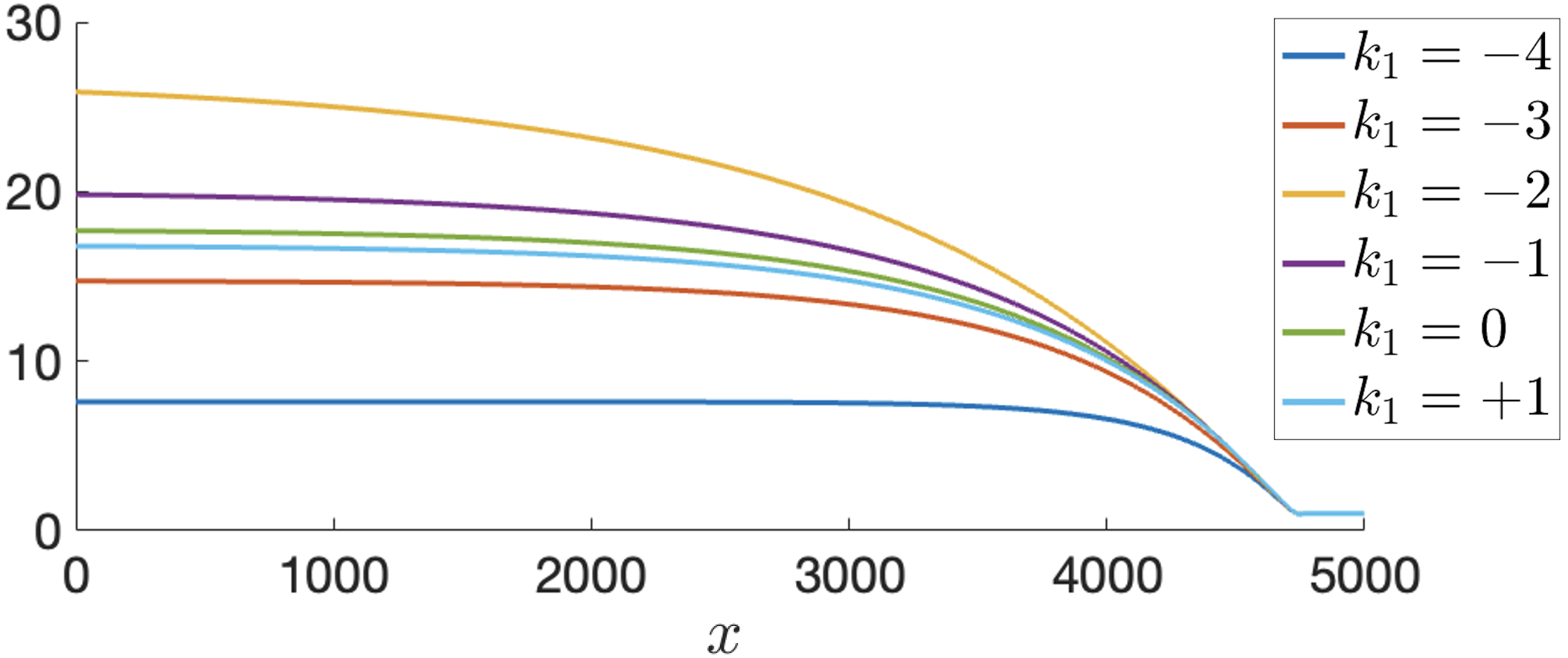}
        \caption{$v$ profiles with $k_0=-1$ and $k_2=+5$ (at $t=95$)}
        \label{Wang_Figure45}
    \end{subfigure}
    \caption{Zoom in on the residual levels of social tension and sample wave profiles.}
    \label{Wang_Figure46}
\end{figure}

In Figure \ref{Wang_Figure44}, a zoomed-in view of Figure \ref{Wang_Figure37} focusing on the third quadrant, reveals that residual values are not monotonic with respect to $k_1$. Within this range, for each fixed value of $k_2$, as $k_1$ increases from -4 to 1, the residual value of $v$ ascends to a local maximum before decreasing. Figure \ref{Wang_Figure45} displays the profiles of $v$ at $t=95$ with $k_2$ fixed at five different values of $k_1$. A comparison of different $v(x,t;k_0,k_1,k_2)$ values shows that $v(0,95;-1,-4,5)<v(0,95;-1,-3,5)<v(0,95;-1,-2,5)$, while $v(0,95;-1,-2,5)>v(0,95;-1,-1,5)>v(0,95;-1,0,5)>v(0,95;-1,1,5)$.
An interpretation of this behavior is that if a protest wave is traversing an environment in which protesting inhibits social tension ($k_0<0$) but police presence enhances social tension ($k_2>0$), then a policing policy aimed at discouraging protests ($k_1<0$) will result in stronger discouragement leading to higher social tension, up to a certain point when the discouragement becomes effective at reducing social tension.

\subsubsection{Cumulative levels of protest activity and social tension}

Perhaps one of the most important metrics is the total exposure of protesting activity and net social tension.  To quantify this metric we look at Area Under the Curve (AUC). These were computed by integrating the area underneath the profiles across space, with the calculation of each profile taking place at the same fixed point in time when wave profiles are formed for all relevant $k$ parameter combinations. Due to the nature of the calculations, the AUC values are affected by maximum values as well residual values, but the plots in Figs \ref{Wang_Figure47}-\ref{Wang_Figure50} more closely resemble their respective residual value plots than the maximum value plots.

In Figure \ref{Wang_Figure47}, we observe two predominant types of results: (1) a substantial AUC value due to $u$ approaching a non-zero residual level, and (2) a more moderate AUC value as $u$ approaches the $u=0$ regime. For all $k_2<0$, the profiles for $u$ demonstrate a narrow pulse-like behavior, indicating that protest activity is concentrated primarily at the invading interface. However, when $k_2>0$ and $k_1$ is sufficiently negative, protest activity extends beyond the invading interface, gradually diminishing to zero as the traveling wave progresses further away (see Figure \ref{Wang_Figure47}, profiles (v) and (vi)).

\begin{comment}
    The maximum value metric provides insight into the greatest exposure to protest activity and highest social tension experienced as the traveling waves progress, but to quantify the total exposure to protest activity and net social tension, we look at Area Under the Curve (AUC). These were computed by integrating the area underneath the profiles across space, with calculation of each profile taking place at the same fixed point in time when wave profiles are all largely mature for all relevant $k$ parameter combinations. Due to the nature of the calculations, the AUC values are affected by max values as well residual values, but the plots in Figs \ref{Wang_Figure47}-\ref{Wang_Figure50} more closely resemble their respective residual value plots than the maximum value plots.

    In Fig \ref{Wang_Figure47}, we see that there are essentially two types of results: (1) considerable AUC because $u$ approaches a non-zero residual level, and (2) modest AUC because $u$ approaches the $u=0$ regime. For all $k_2<0$, the $u$ profiles exhibit narrow pulse-like behavior, so protest activity occurs almost entirely at the invading interface. For $k_2>0$ and a sufficiently negative $k_1$, the protest activity is not limited to just at the invading interface, but will still dwindle towards 0 as the traveling wave passes and continues to move progressively further and further away (Fig \ref{Wang_Figure47} profiles vi and vii).
\end{comment}

\begin{figure} [H]
    \centering
        \includegraphics[width=.9\textwidth]{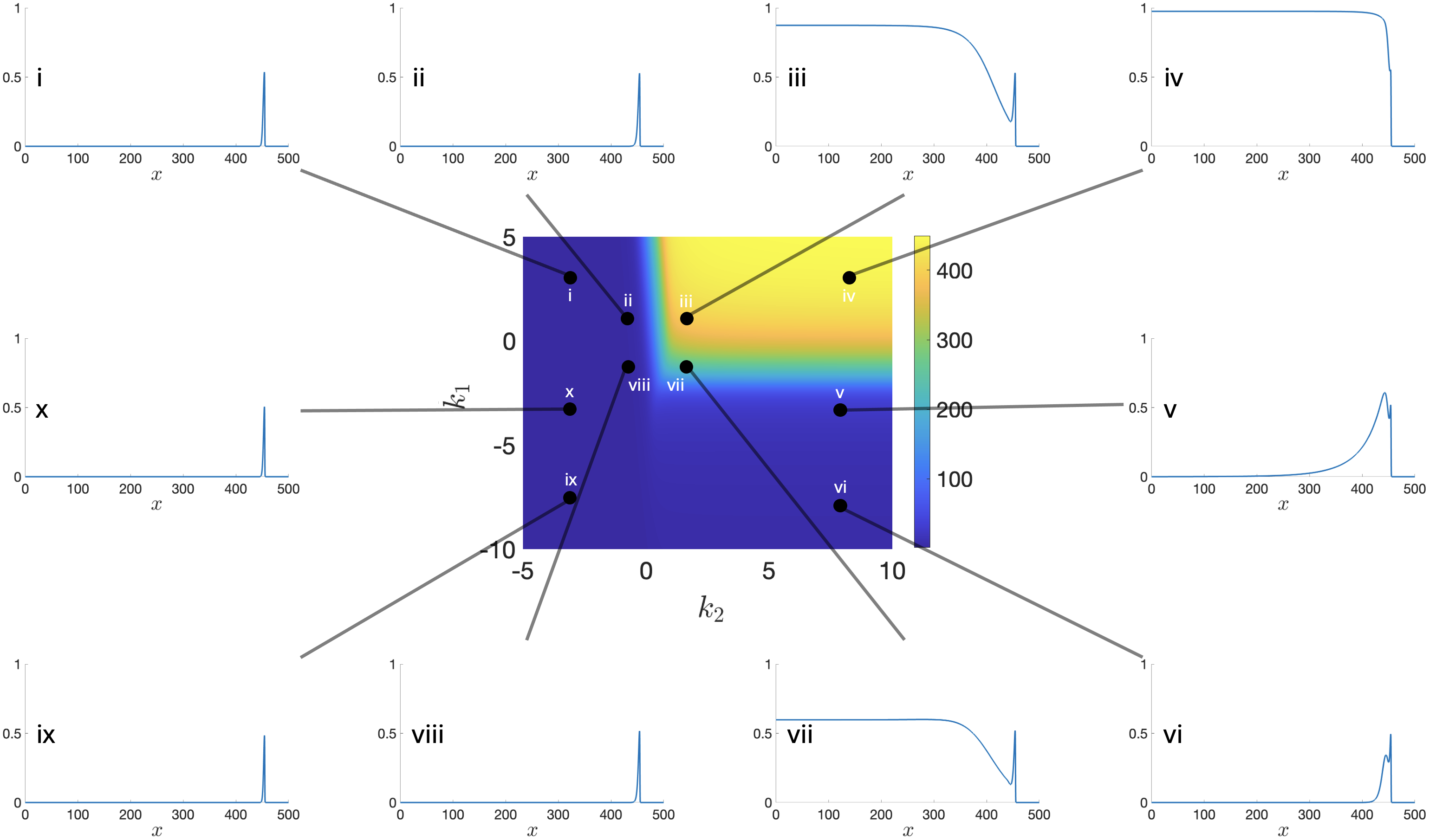}
    \caption{AUC of $u$ with $k_0=-1$ at $t=15$ and profiles at select coordinates. Moving clockwise from the top left corner the coordinates $(k_1,k_2)$ are: i (3,-3), ii (1,-1), iii (1,1), iv (3,8), v (-3,8), vi (-8,8), vii (-1,1), viii (-1,-1), ix (-8,-3), and x (-3,-3).}
    \label{Wang_Figure47}
\end{figure}

\begin{figure} [H]
    \centering
        \includegraphics[width=.9\textwidth]{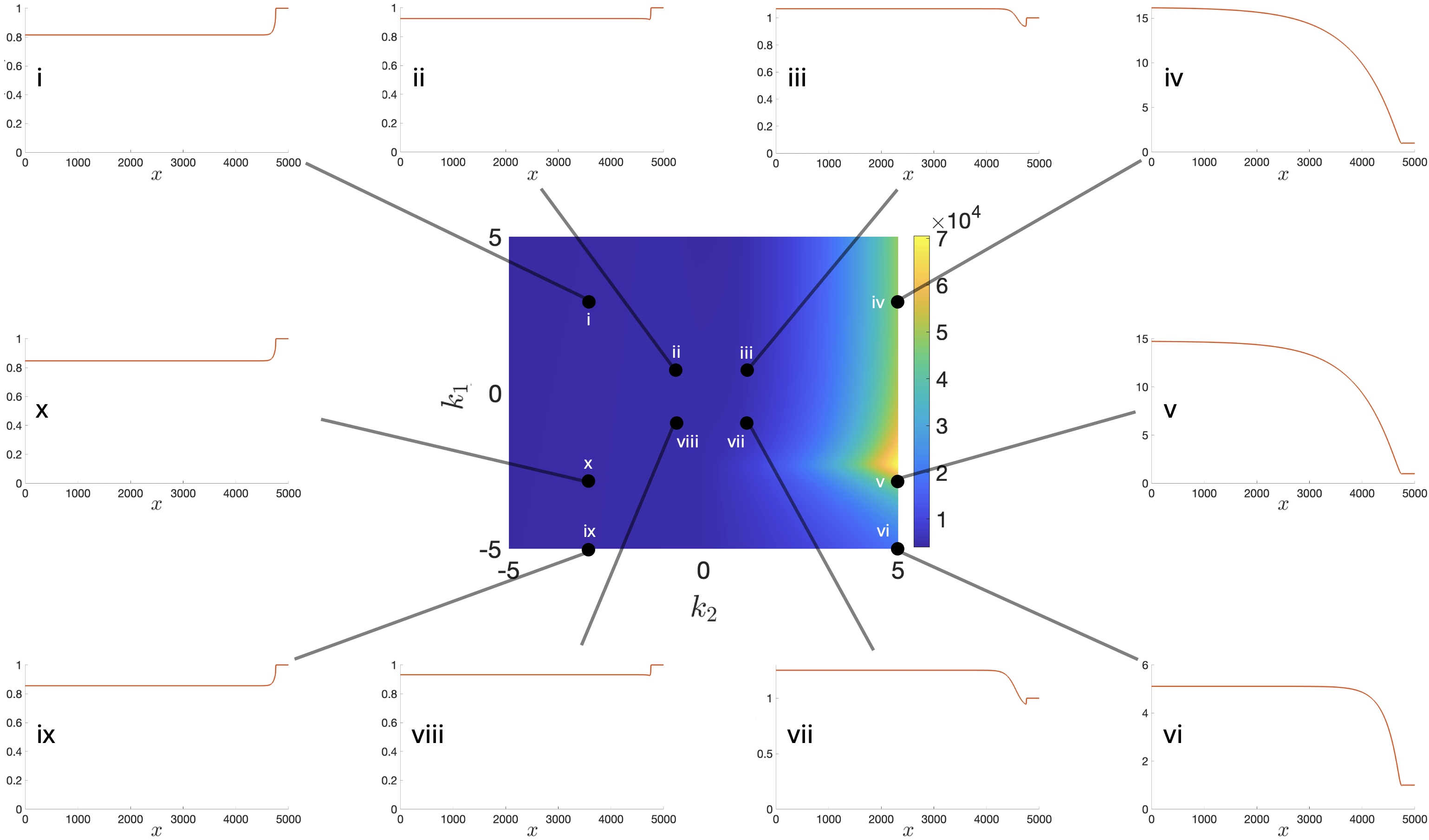}
    \caption{AUC of $v$ with $k_0=-1$ at $t=100$ and profiles at select coordinates (at $t=95$). Moving clockwise from the top left corner the coordinates $(k_1,k_2)$ are: i (3,-3), ii (1,-1), iii (1,1), iv (3,5), v (-3,5), vi (-5,5), vii (-1,1), viii (-1,-1), ix (-5,-3), and x (-3,-3).}
    \label{Wang_Figure48}
\end{figure}

Figure \ref{Wang_Figure48} illustrates the social tension corresponding to the protest activity depicted in Fig \ref{Wang_Figure47}. When $k_0=-1$ and $k_2<0$, which encompasses scenarios where the residual tension is less than one, the reductions in AUC are relatively gradual as both $k_1$ and $k_2$ decrease. In this parameter space, the profiles quickly approach the residual value, attributing the AUC decrease primarily to the diminishing residual value as $k_1$ and $k_2$ decrease. Notably, the heatmap in Figure \ref{Wang_Figure48} reveals a similar nonmonotonic trend in AUC as observed in the residual values in Figure \ref{Wang_Figure37}.

\begin{figure} [H]
    \centering
        \includegraphics[width=.9\textwidth]{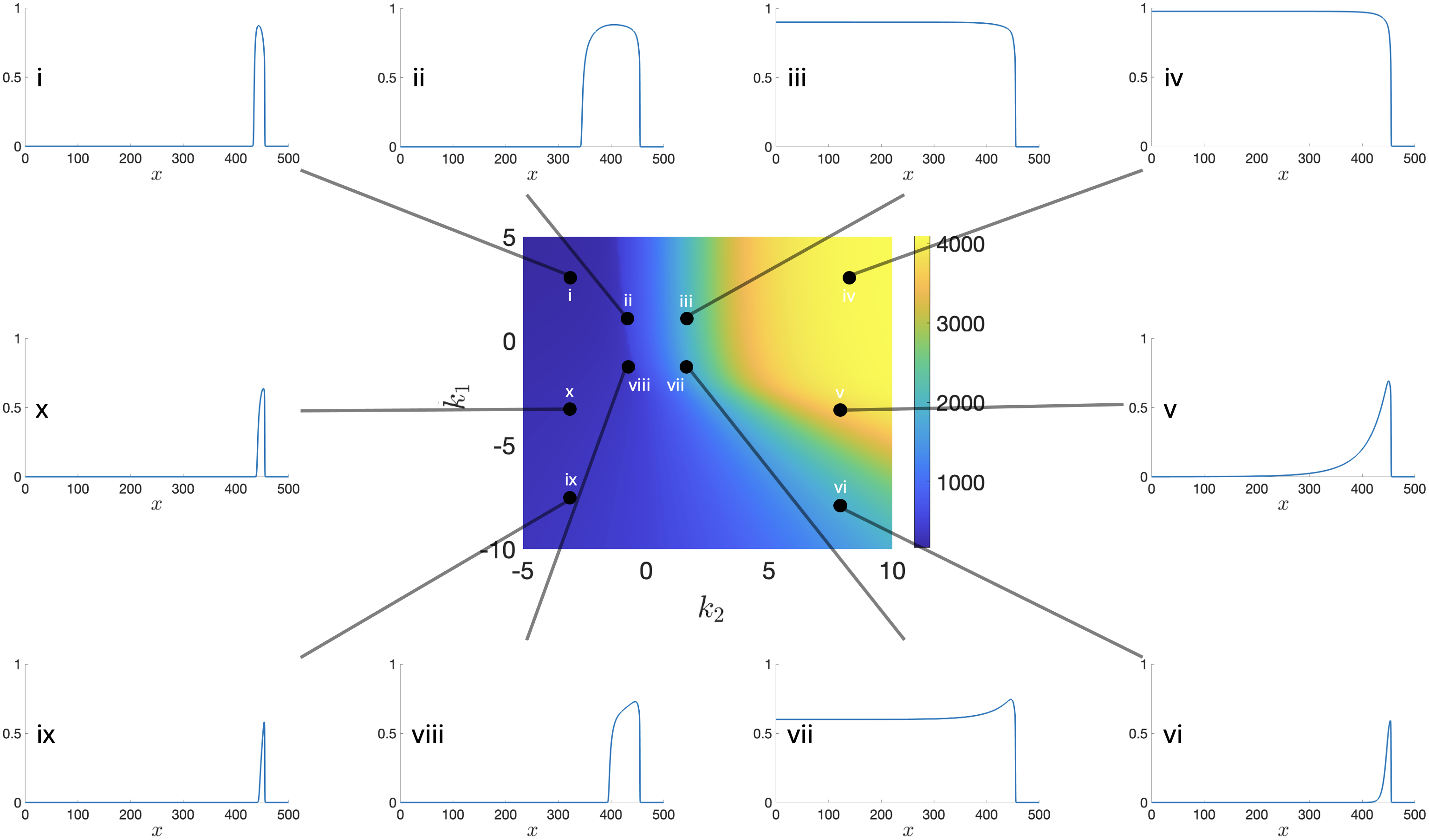}
    \caption{AUC of $u$ with $k_0=1$ at $t=15$ and profiles at select coordinates. Moving clockwise from the top left corner the coordinates $(k_1,k_2)$ are: i (3,-3), ii (1,-1), iii (1,1), iv (3,8), v (-3,8), vi (-8,8), vii (-1,1), viii (-1,-1), ix (-8,-3), and x (-3,-3).}
    \label{Wang_Figure49}
\end{figure}

The heatmap for $k_0=1$ in Figure \ref{Wang_Figure49} reveals a more diverse range of behaviors than the predominantly observed two types (narrow pulse-like profiles of $u$ and profiles trending towards non-zero steady states) in the $k_0=-1$ heatmap depicted in Figure \ref{Wang_Figure47}. While profiles can still be categorized as near-pulse or non-pulse, the near-pulse waves, particularly when $k_1>0$, may exhibit plateaus near $u=1$ for a significant portion of the $x$ domain before dropping sharply towards zero. This plateauing behavior results in a relatively large AUC value, even though the residual value is roughly zero. Generally, in comparison to the $k_0=-1$ scenario, protest activity is not confined solely to the invading interface when $k_0=1$, leading to a higher overall level of protest activity.

\begin{figure} [H]
    \centering
        \includegraphics[width=\textwidth]{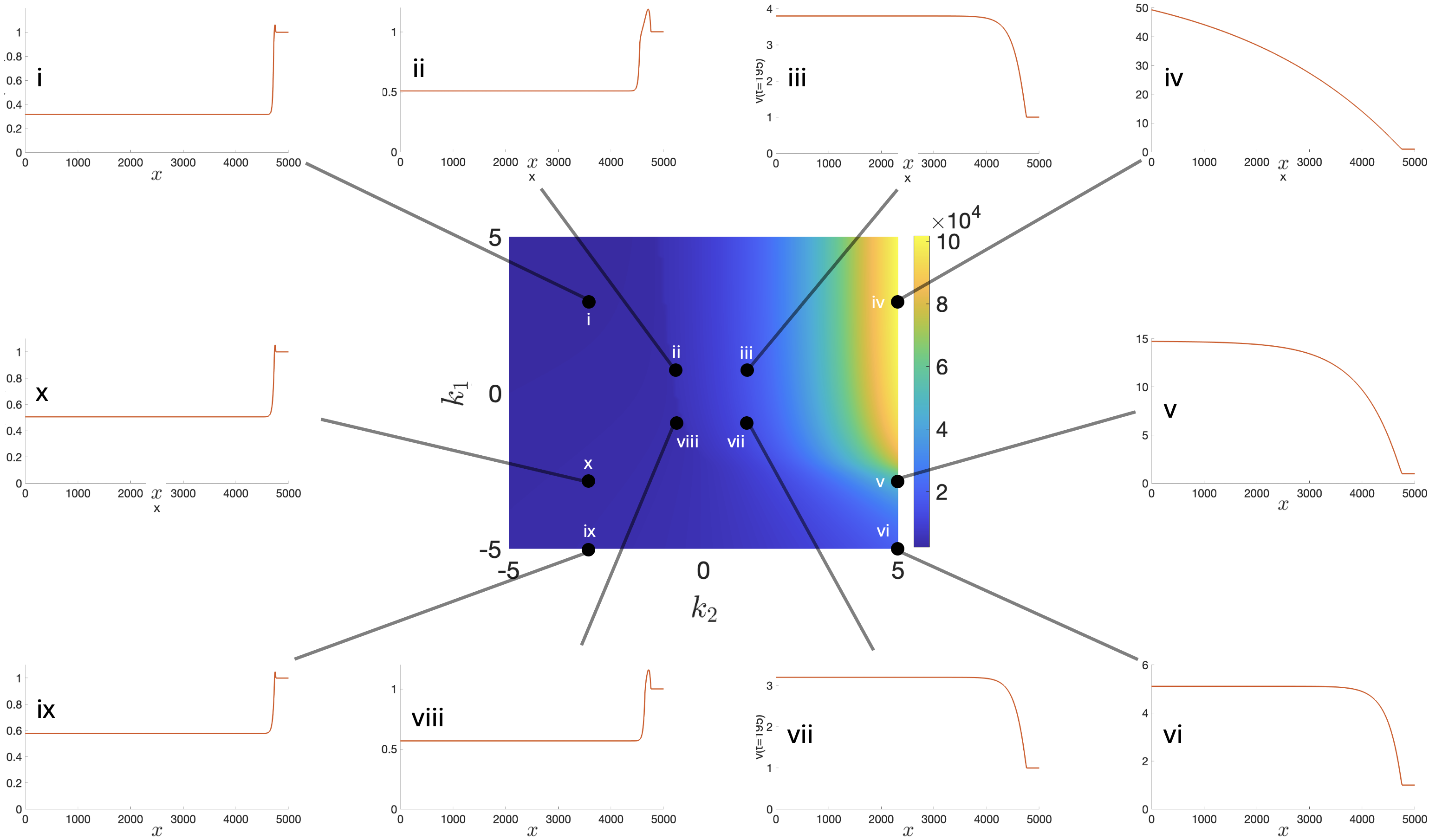}
    \caption{AUC of $v$ with $k_0=1$ at $t=100$ and profiles at select coordinates (at $t=95$). Moving clockwise from the top left corner the coordinates $(k_1,k_2)$ are: i (3,-3), ii (1,-1), iii (1,1), iv (3,5), v (-3,5), vi (-5,5), vii (-1,1), viii (-1,-1), ix (-5,-3), and x (-3,-3).}
    \label{Wang_Figure50}
\end{figure}

%fig analysis???

Figures \ref{Wang_Figure49} and \ref{Wang_Figure50} emphasize the significance of considering qualitative behaviors in conjunction with quantitative measurements. Solely focusing on maximum protest levels and global protest levels might overlook crucial insights into what is experienced as a protest wave approaches and passes a fixed point.

\subsection{Qualitative Outcomes}

In this section, we delve into the qualitative outcomes of simulations, considering the viewpoint of a resident fixed at a specific location. Taking the perspective of a town or neighborhood, the encountered activity varies significantly based on the environment (parameters) through which the protesting wave passes. The fluctuations in protest activity, tension, and police presence play a crucial role in shaping the experience of residents. They may witness scenarios where there is no protest activity, followed by a sudden surge, only to witness a rapid and complete cessation of all activity. Alternatively, there may be instances of no protest activity, a sudden surge, followed by a swift decline before resurging and persisting at levels even higher than the initial invading levels. Although these experiences may initially appear similar, they can evolve into vastly different systems. Given the diverse outcomes, we have identified a few key features and the specific regions where these features manifest.

\subsubsection{Productive and unproductive strategies for relieving tension}

Arguably, the most crucial metric for assessing the effectiveness of protests and policing strategies involves comparing tension levels before and after the protests. In an environment characterized by a baseline tension level that experiences a wave of protest activity, evaluating the productivity of the protests and policing strategy in alleviating or exacerbating the underlying tension is vital. This assessment is made by comparing tension levels before and after the complete passage of the wave of protests through a region. We quantified residual tension and, in Figures \ref{Wang_Figure33} to \ref{Wang_Figure62}, juxtaposed these residual levels against the initial level ($v(x,0)=1$). Regions, where the wave of activity escalated social tension, are shaded purple, while those where tension decreased immediately and persistently are shaded green. If the tension increases at some point before ultimately being successfully reduced, the region is shaded light blue.  We denote the quantity $v_0$ to be the initial tension and $v_{ss}$ to be the difference between the residual tension and $v_0$.

\begin{figure} [H]
    \centering
    \begin{subfigure}[t]{0.6\textwidth}
        \includegraphics[width=\textwidth]{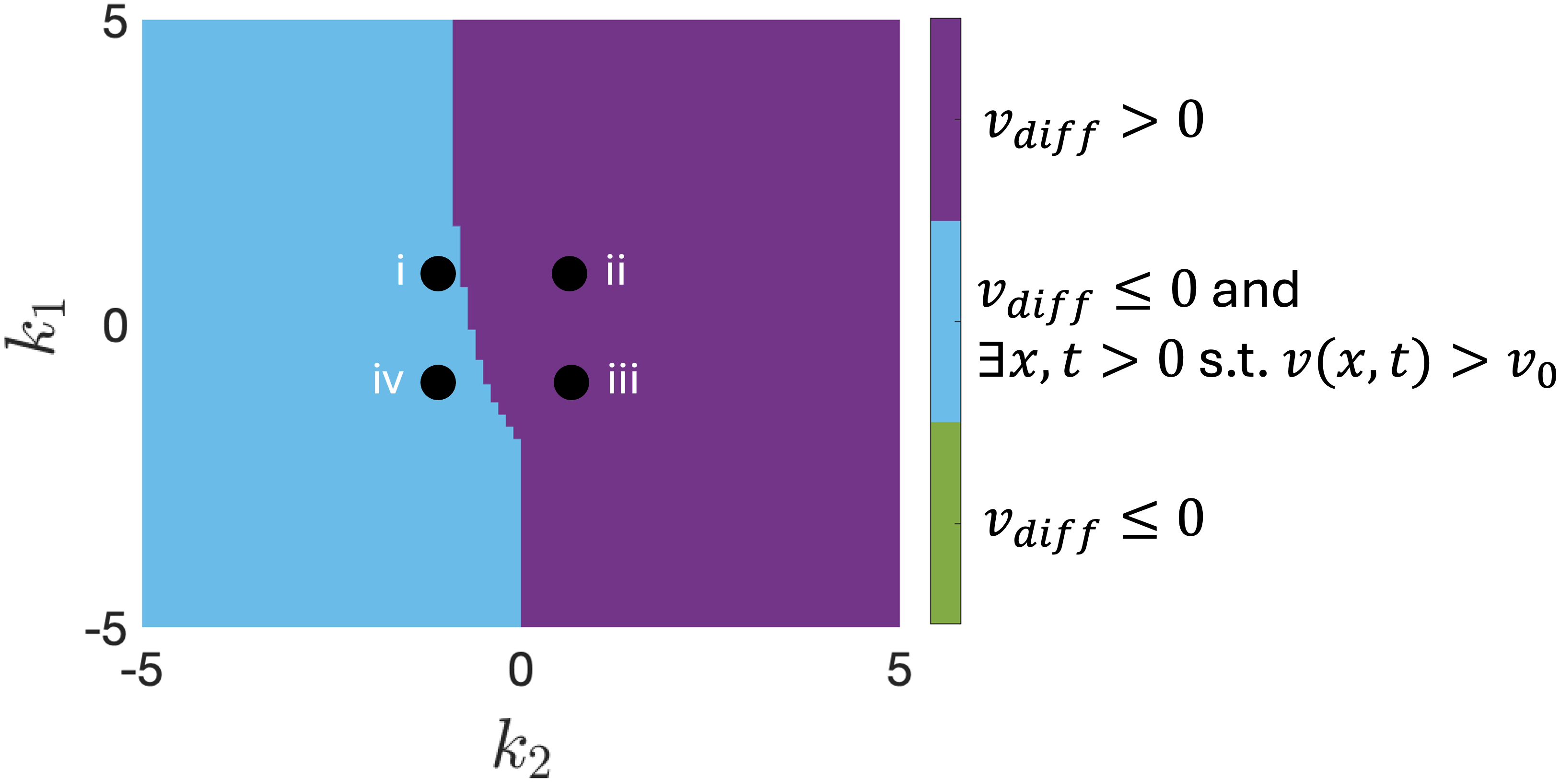}
        \caption{$k_0=1$}
        \label{Wang_Figure51}
    \end{subfigure}
    
    \begin{subfigure}[t]{0.24\textwidth}
        \includegraphics[width=\textwidth]{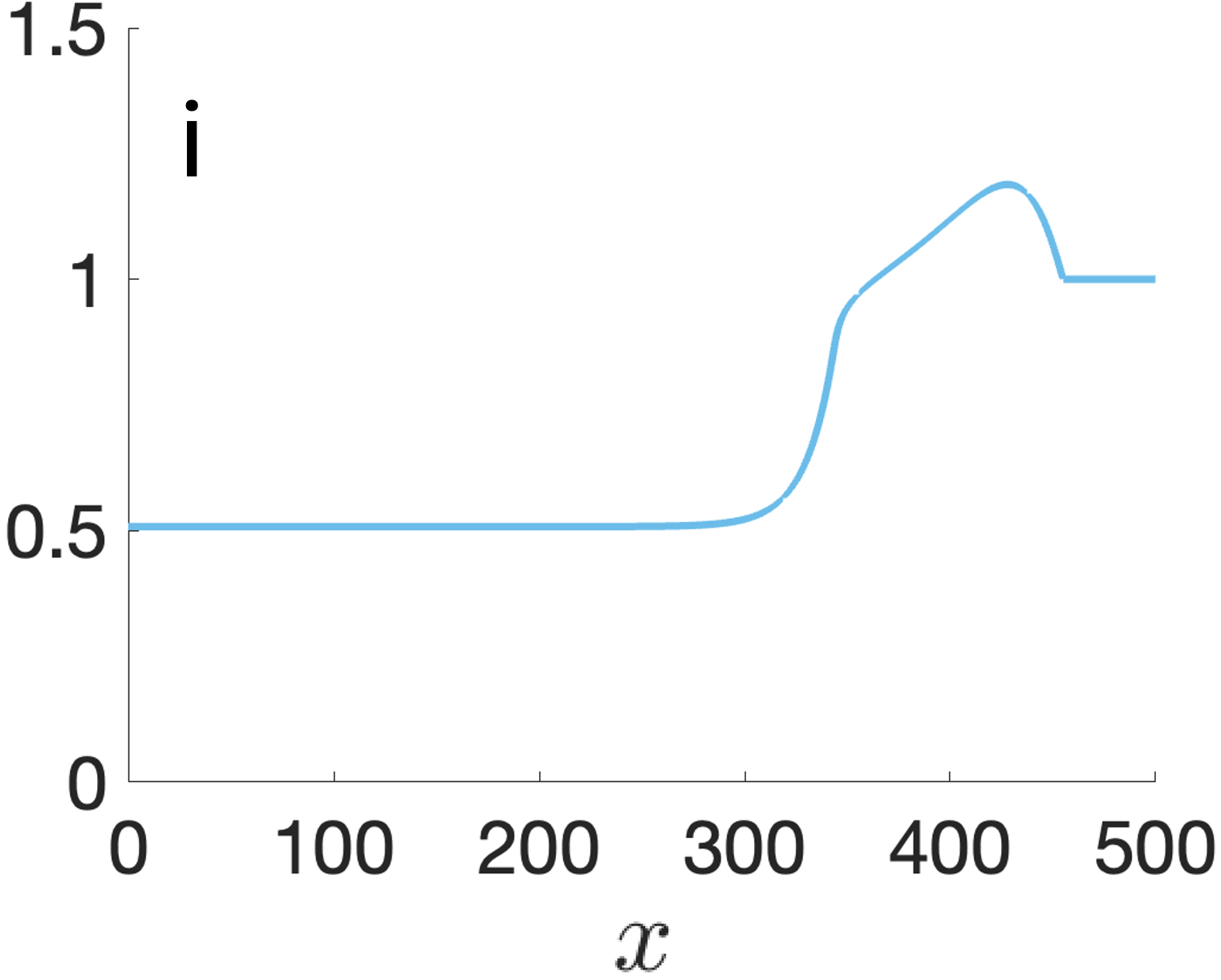}
        \caption{$k_1=1,k_2=-1$}
        \label{Wang_Figure52}
    \end{subfigure}
    \begin{subfigure}[t]{0.24\textwidth}
        \includegraphics[width=\textwidth]{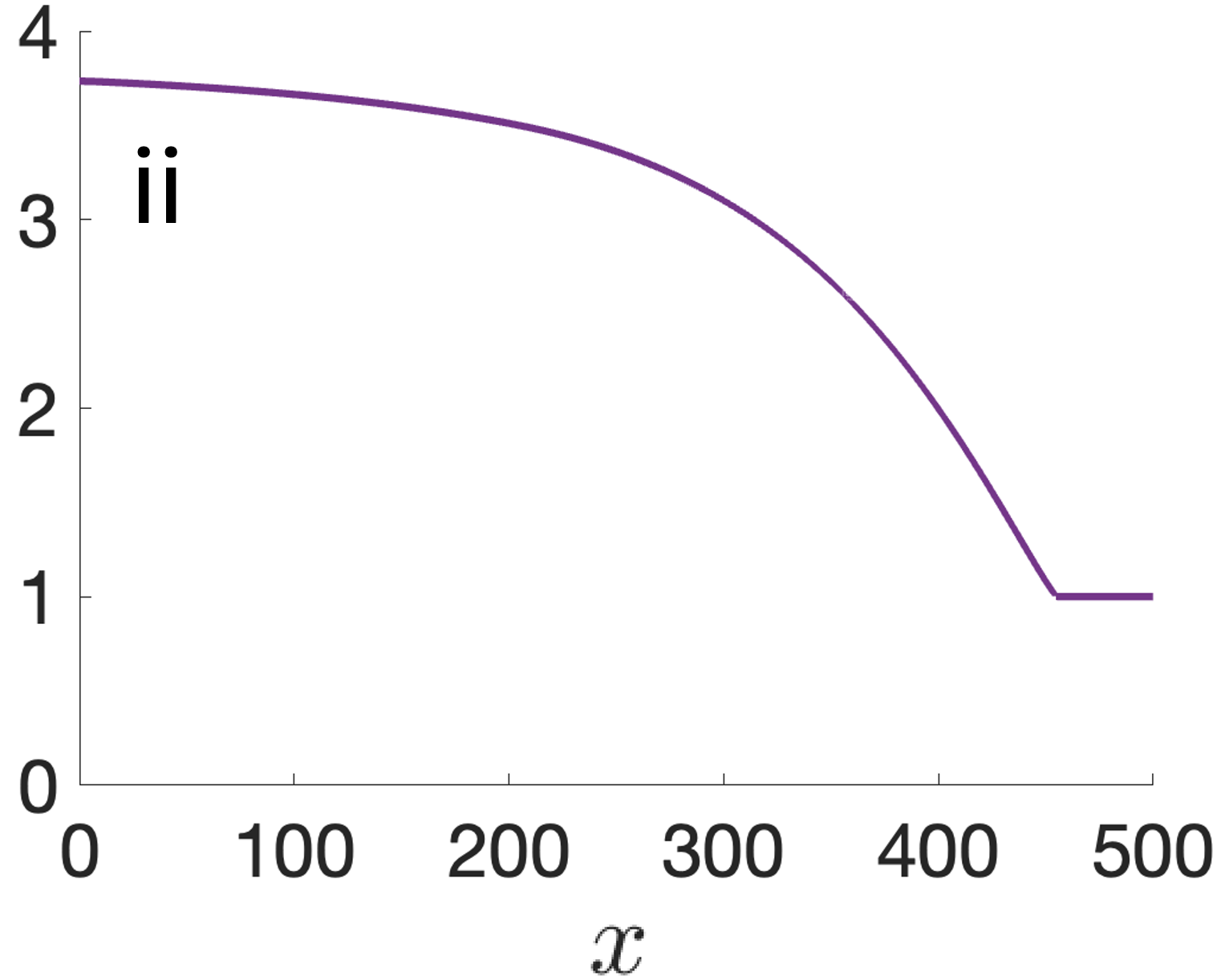}
        \caption{$k_1=1,k_2=1$}
        \label{Wang_Figure53}
    \end{subfigure}
    \begin{subfigure}[t]{0.24\textwidth}
        \includegraphics[width=\textwidth]{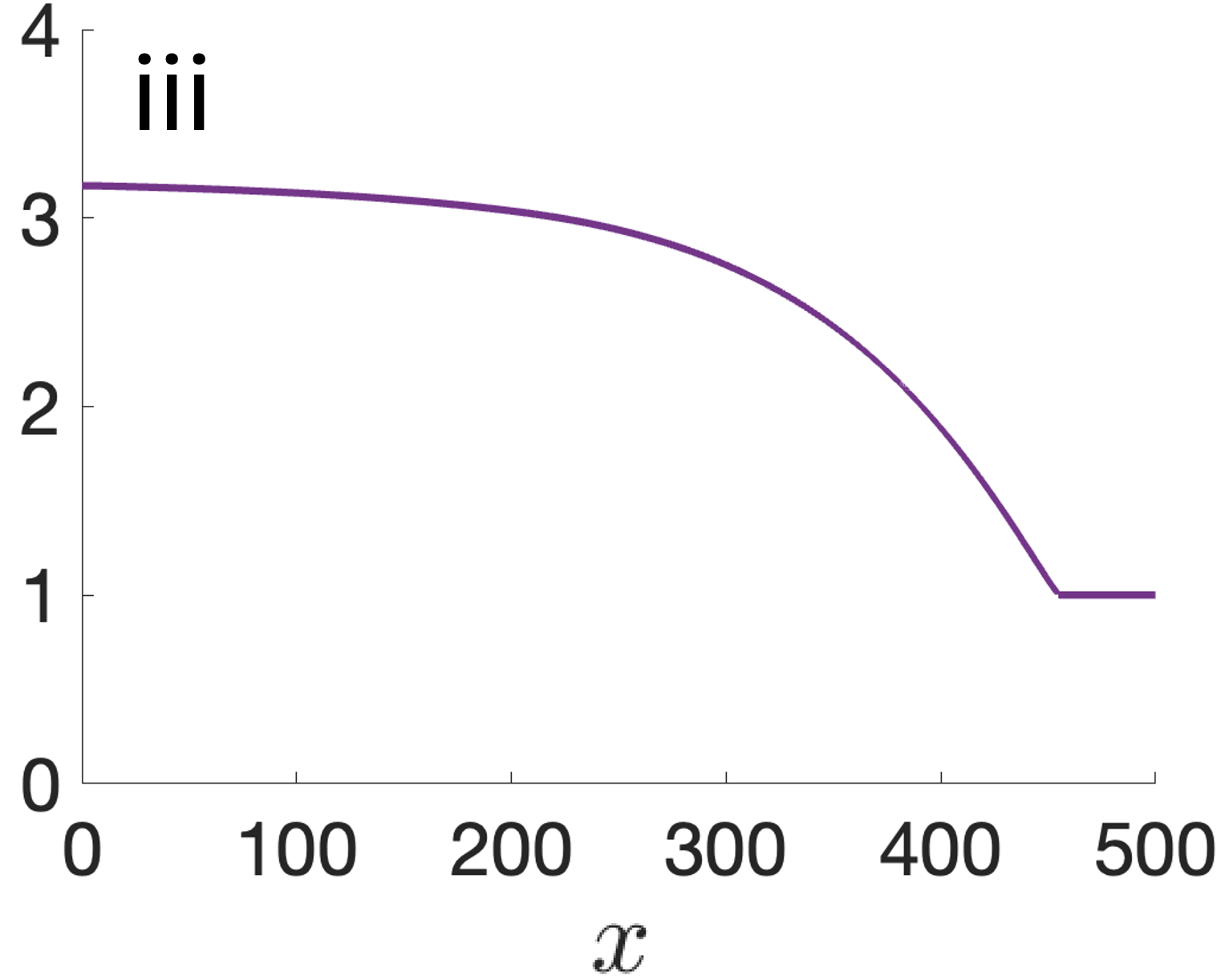}
        \caption{$k_1=-1,k_2=1$}
        \label{Wang_Figure54}
    \end{subfigure}
    \begin{subfigure}[t]{0.24\textwidth}
        \includegraphics[width=\textwidth]{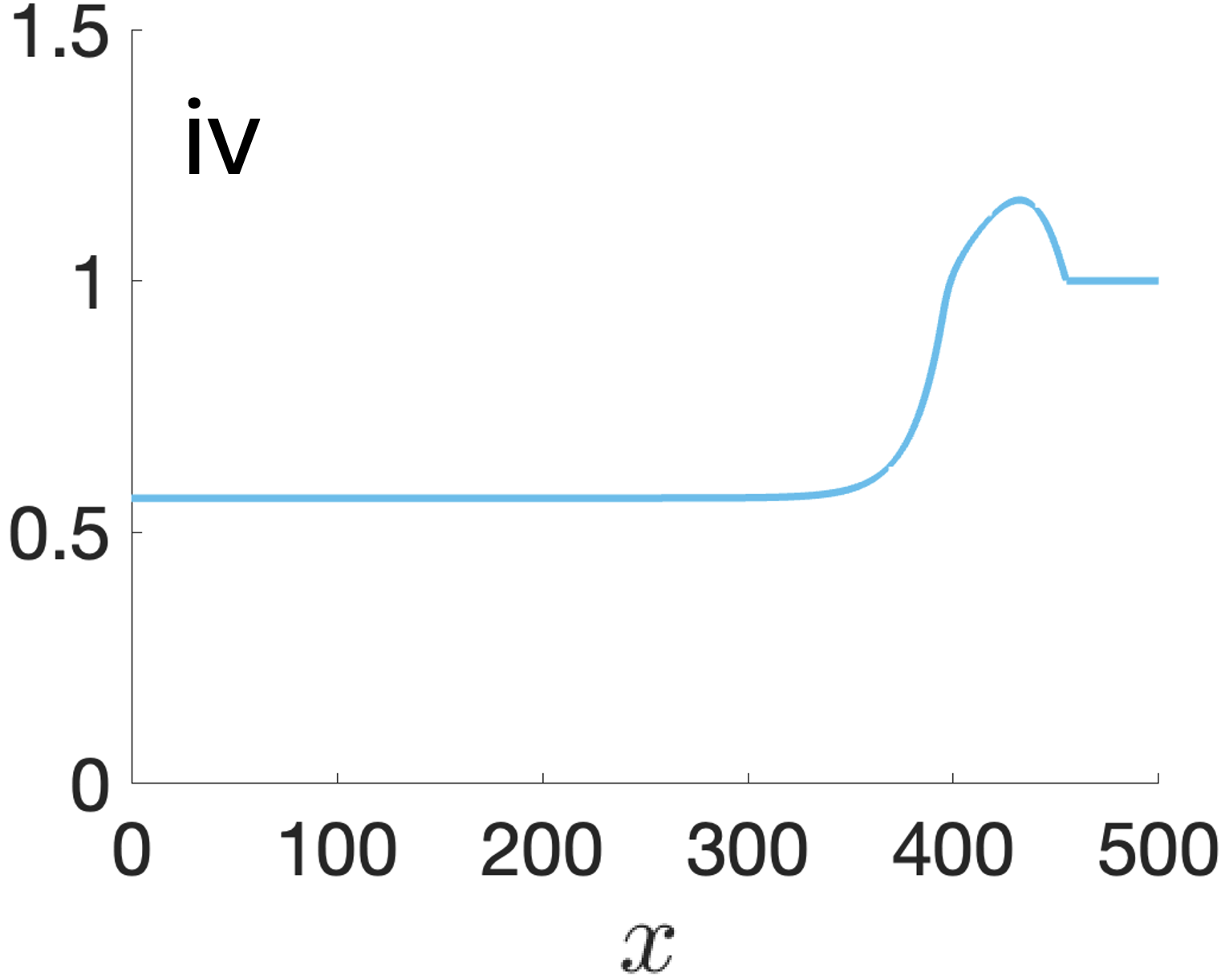}
        \caption{$k_1=-1,k_2=-1$}
        \label{Wang_Figure55}
    \end{subfigure}
    \caption{(a) Categorization of regions in parameter space $(k_1,k_2) \in [-5,5] \times [-5,5]$ when $k=1$ into two categories: (1) where tension is relieved ($v_{diff} \le 0$--light blue) and (2) where it is exacerbated ($v_{diff} > 0$--purple). (b)-(e) illustrate sample tension profiles from (a) with the dots and labels indicating the $(k_1,k_2)$ coordinates (b)-(e) are from. Note: there is no green region in this figure.}
    \label{Wang_Figure56}
\end{figure}

In the scenario where $k_0>0$ (exemplified in Figure \ref{Wang_Figure56}), there is an immediate uptick in tension following the invading wave of protest activity. After this escalation, two outcomes emerge: (1) tension continues to rise monotonically until reaching a steady state with a higher tension in the system than the initial state (depicted in purple in Figure \ref{Wang_Figure51} with example profiles in Figures \ref{Wang_Figure53} and \ref{Wang_Figure54}), and (2) tension attains a local maximum, after which it decreases monotonically until reaching a steady state with lower tension than the initial state (illustrated in light blue in Figure \ref{Wang_Figure51} with example profiles in Figures \ref{Wang_Figure52} and \ref{Wang_Figure55}). If $k_2>0$, no value of $k_1$ can prevent a tension increase at the steady state. If $k_2$ is slightly negative, even with $k_1>0$, a residual tension increase occurs, and some $k_1<0$ values lead to residual tension increases as well. Figure \ref{Wang_Figure39} demonstrates that residual tension monotonically increases with both $k_1$ and $k_2$.

\begin{figure} [H]
    \centering
    \begin{subfigure}[t]{0.6\textwidth}
        \includegraphics[width=\textwidth]{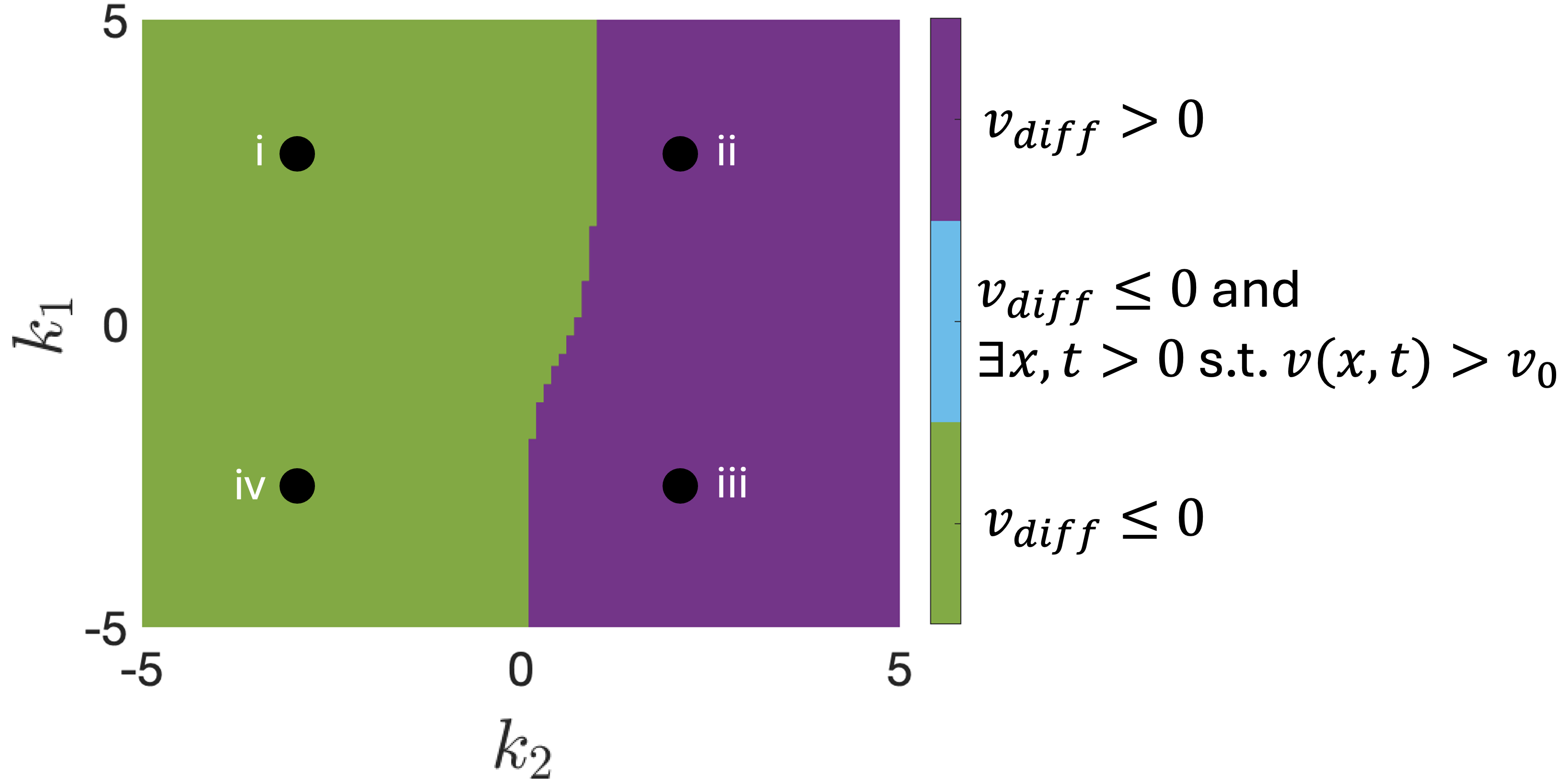}
        \caption{$k_0=-1$}
        \label{Wang_Figure57}
    \end{subfigure}
    
    \begin{subfigure}[t]{0.24\textwidth}
        \includegraphics[width=\textwidth]{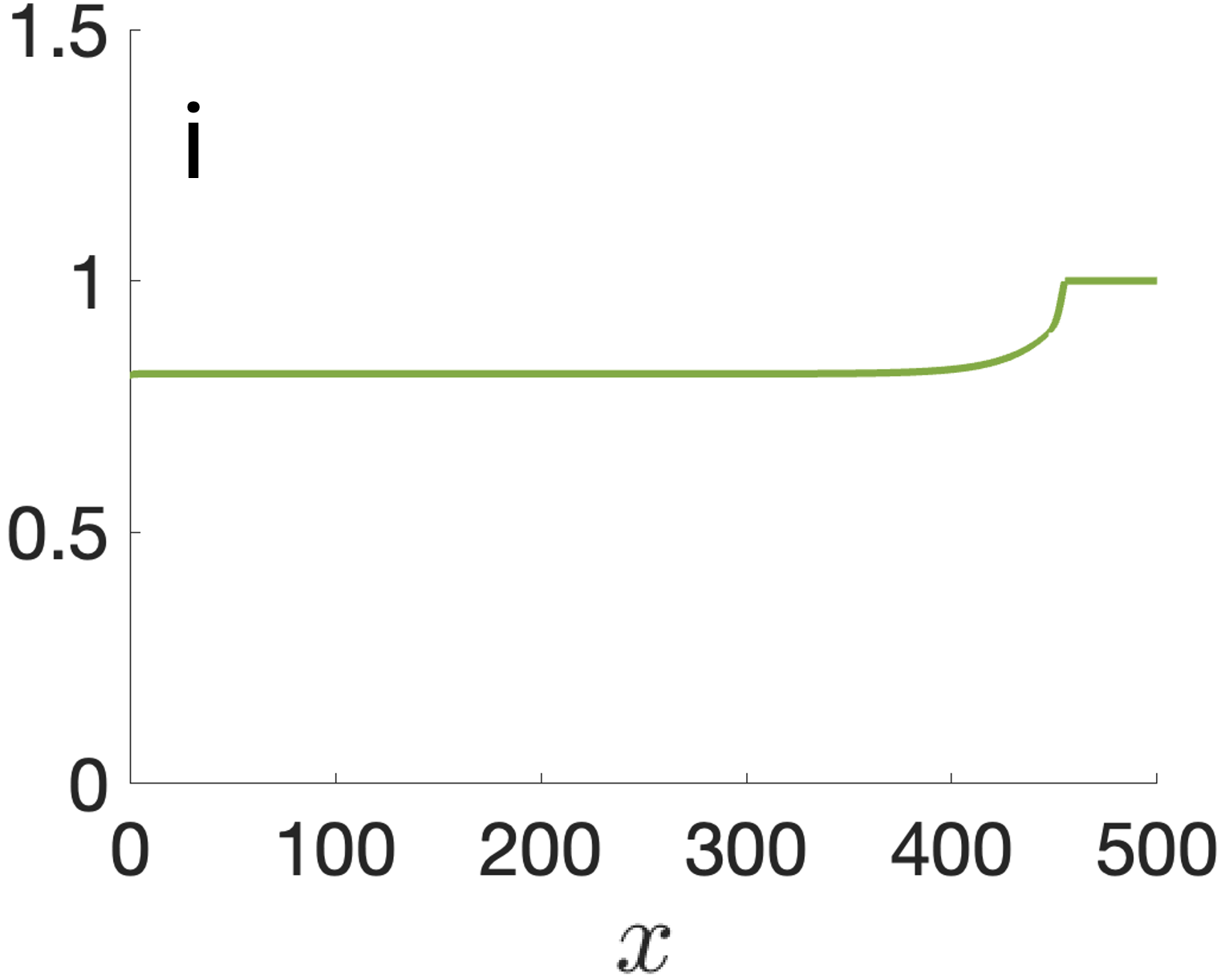}
        \caption{$v(k_1=3,k_2=-3)$}
        \label{Wang_Figure58}
    \end{subfigure}
    \begin{subfigure}[t]{0.24\textwidth}
        \includegraphics[width=\textwidth]{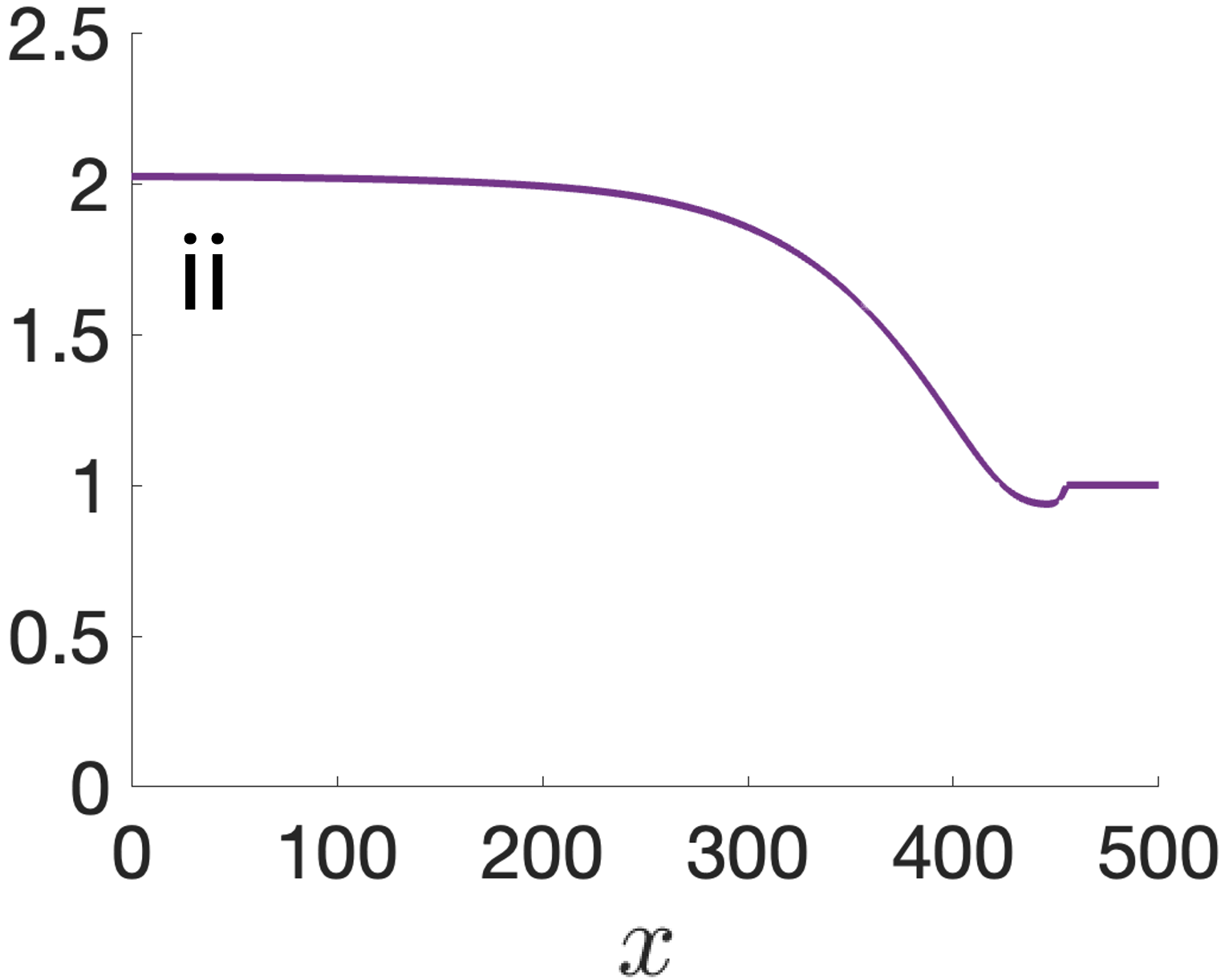}
        \caption{$k_1=3,k_2=2$}
        \label{Wang_Figure59}
    \end{subfigure}
    \begin{subfigure}[t]{0.24\textwidth}
        \includegraphics[width=\textwidth]{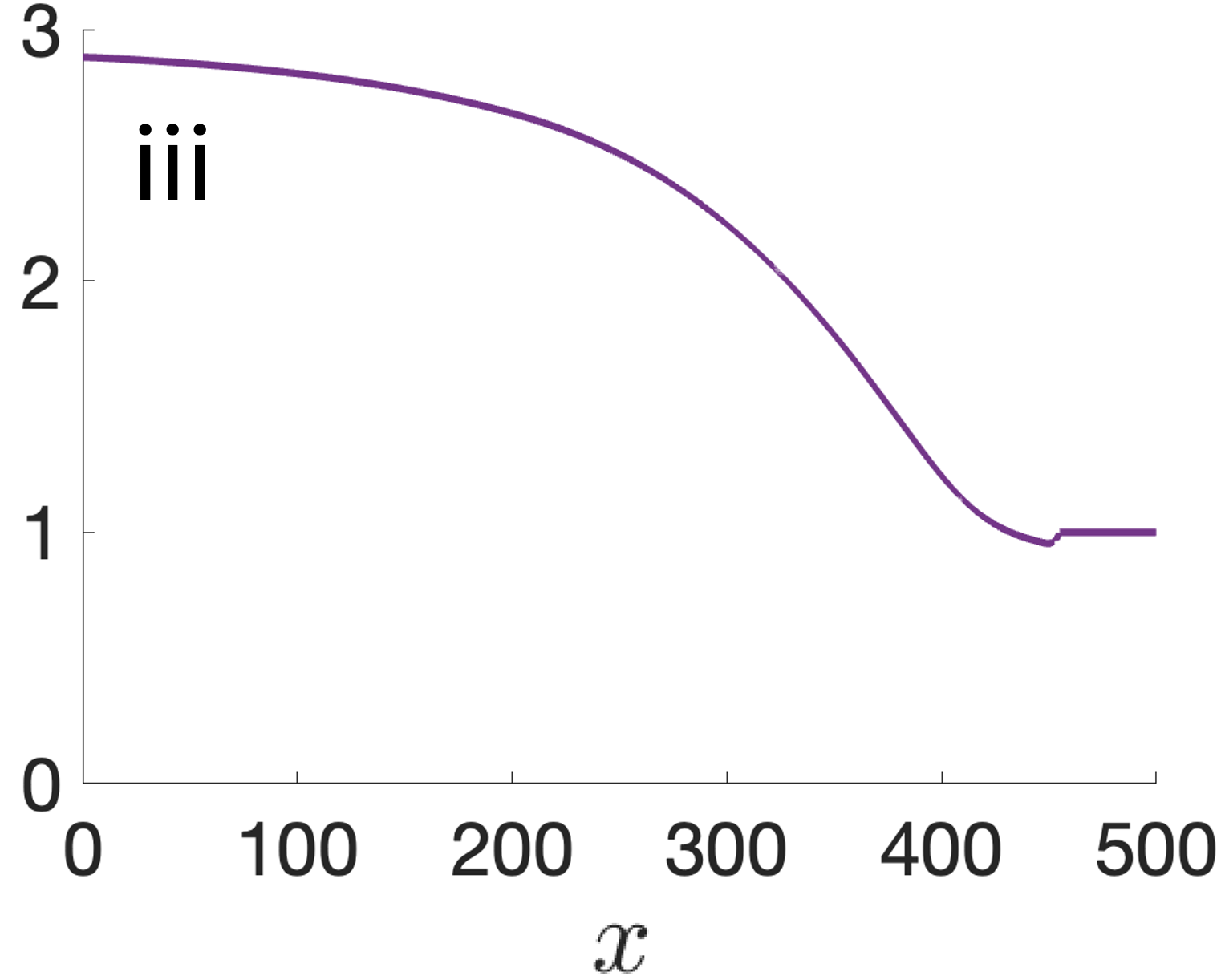}
        \caption{$k_1=-3,k_2=2$}
        \label{Wang_Figure60}
    \end{subfigure}
    \begin{subfigure}[t]{0.24\textwidth}
        \includegraphics[width=\textwidth]{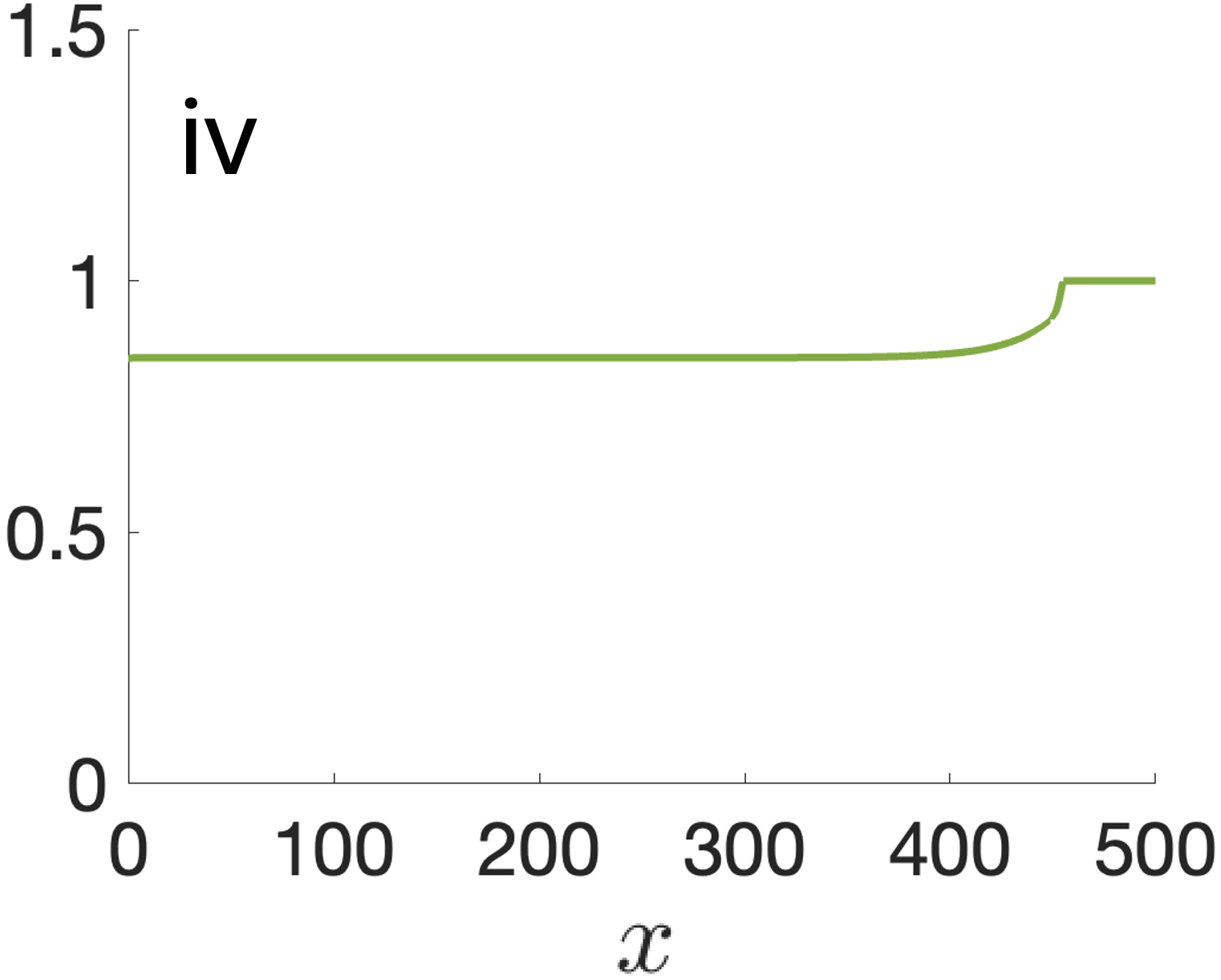}
        \caption{$k_1=-3,k_2=-3$}
        \label{Wang_Figure61}
    \end{subfigure}
    \caption{(a) Categorization of parameter regions $(k_1,k_2) \in [-5,5] \times [-5,5]$ $k_0=-1$ into two categories: (1) where tension is relieved ($v_{diff} \le 0$--green) and where it is exacerbated ($v_{diff} > 0$--purple). (b)-(e) illustrate sample tension profiles from (a) with the dots and labels indicating the $(k_1,k_2)$ coordinates (b)-(e) are from. Note: there is no light blue region in this figure.}
    \label{Wang_Figure62}
\end{figure}

The system when $k_0<0$ (as seen in Figure \ref{Wang_Figure62}) appears to exhibit two possible outcomes: (1) immediate and ultimate relief of tension (indicated in green in Figure \ref{Wang_Figure57}) or (2) initial abatement followed by an ultimate increase in tension compared to the initial levels (depicted in purple in Figure \ref{Wang_Figure57}). When $k_2<0$, the value of $k_1$ becomes irrelevant; tension will invariably be relieved. Additionally, if $k_2>0$, tension will consistently increase unless the magnitude of $k_2$ is small and $k_1$ is positive or nearly positive. Moreover, the profiles of $v$ (shown in Figures \ref{Wang_Figure58} to \ref{Wang_Figure61}) exhibit qualitative similarities in terms of their shapes, regardless of whether they belong to the green or purple region. The primary distinction lies in the residual values, as illustrated in Figure \ref{Wang_Figure37}.

\subsubsection{Location of maximum levels of protest activity}

In Figures \ref{Wang_Figure33} and \ref{Wang_Figure43}, we are aware of the protest levels at the invading interface and the residual levels remaining. From a local standpoint, it becomes pertinent to discern whether the level observed at the invading interface will be the peak level experienced, with protesting reaching higher levels before diminishing, or if the protest activity will persistently escalate to a maximum as $t \rightarrow \infty$.

\begin{figure} [H]
    \centering
    \begin{subfigure}[t]{0.5\textwidth}
        \includegraphics[width=\textwidth]{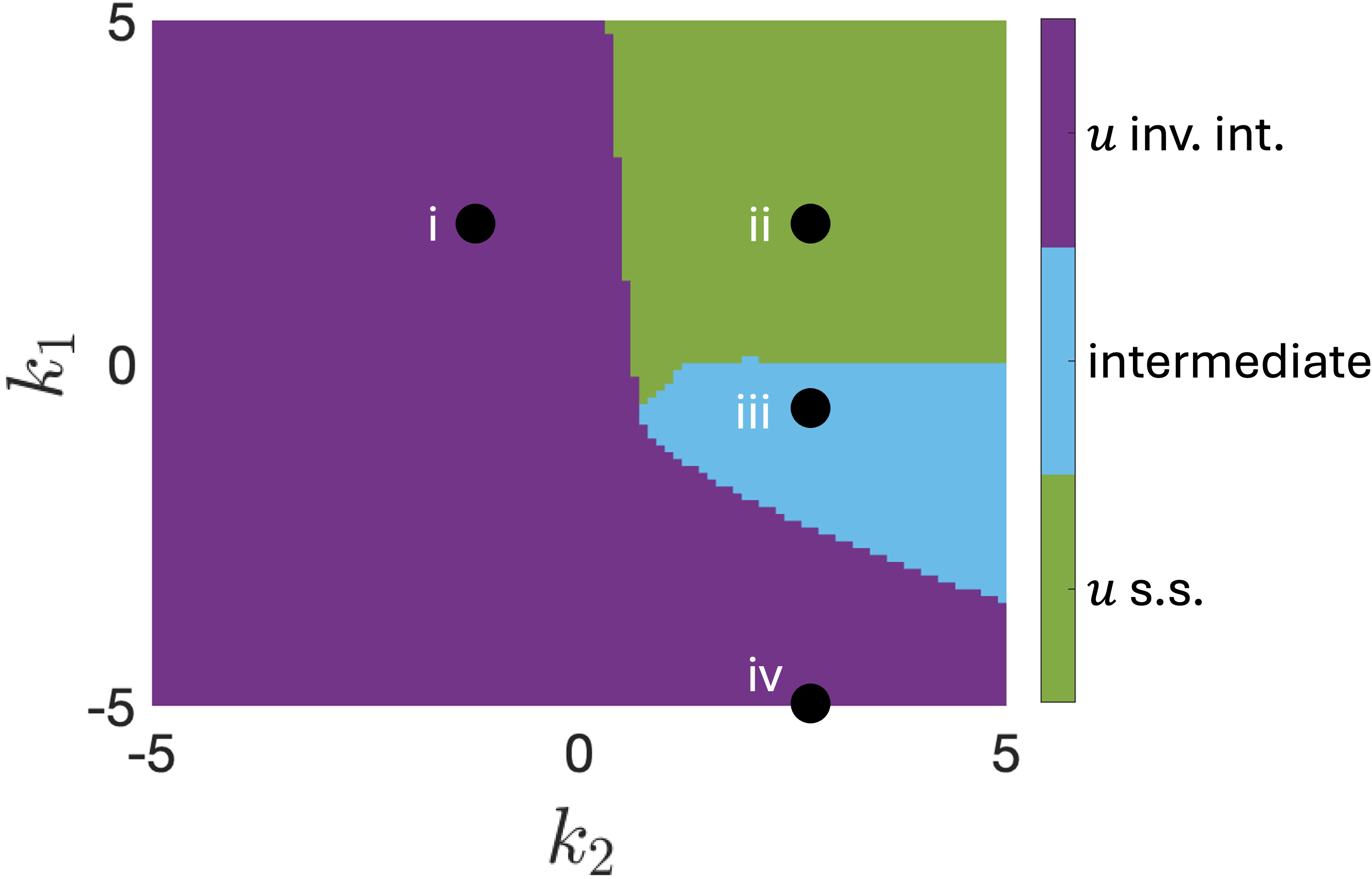}
        \caption{$k_0=-1$}
        \label{Wang_Figure63}
    \end{subfigure}
    
    \begin{subfigure}[t]{0.24\textwidth}
        \includegraphics[width=\textwidth]{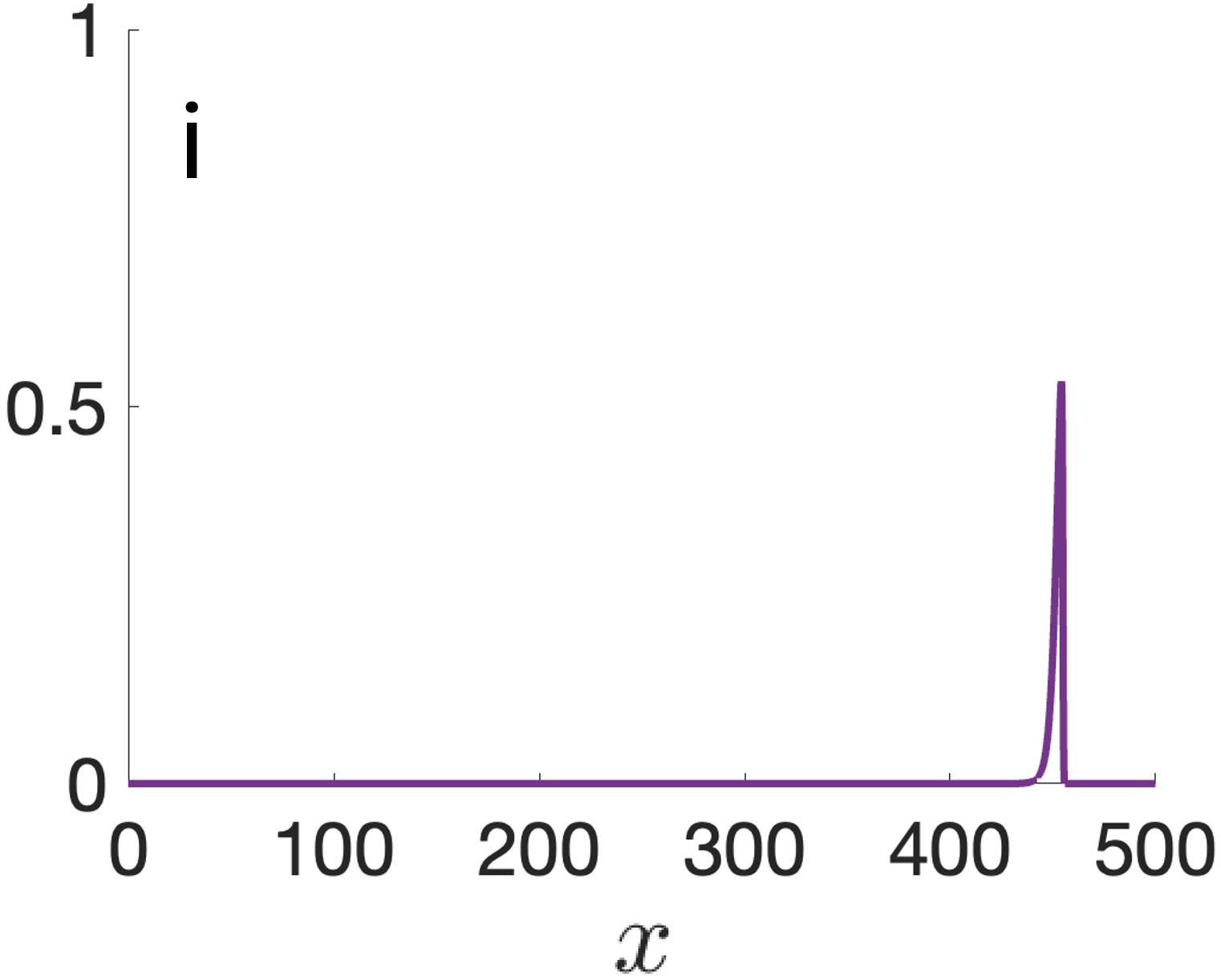}
        \caption{$k_1=2,k_2=-1$}
        \label{Wang_Figure64}
    \end{subfigure}
    \begin{subfigure}[t]{0.24\textwidth}
        \includegraphics[width=\textwidth]{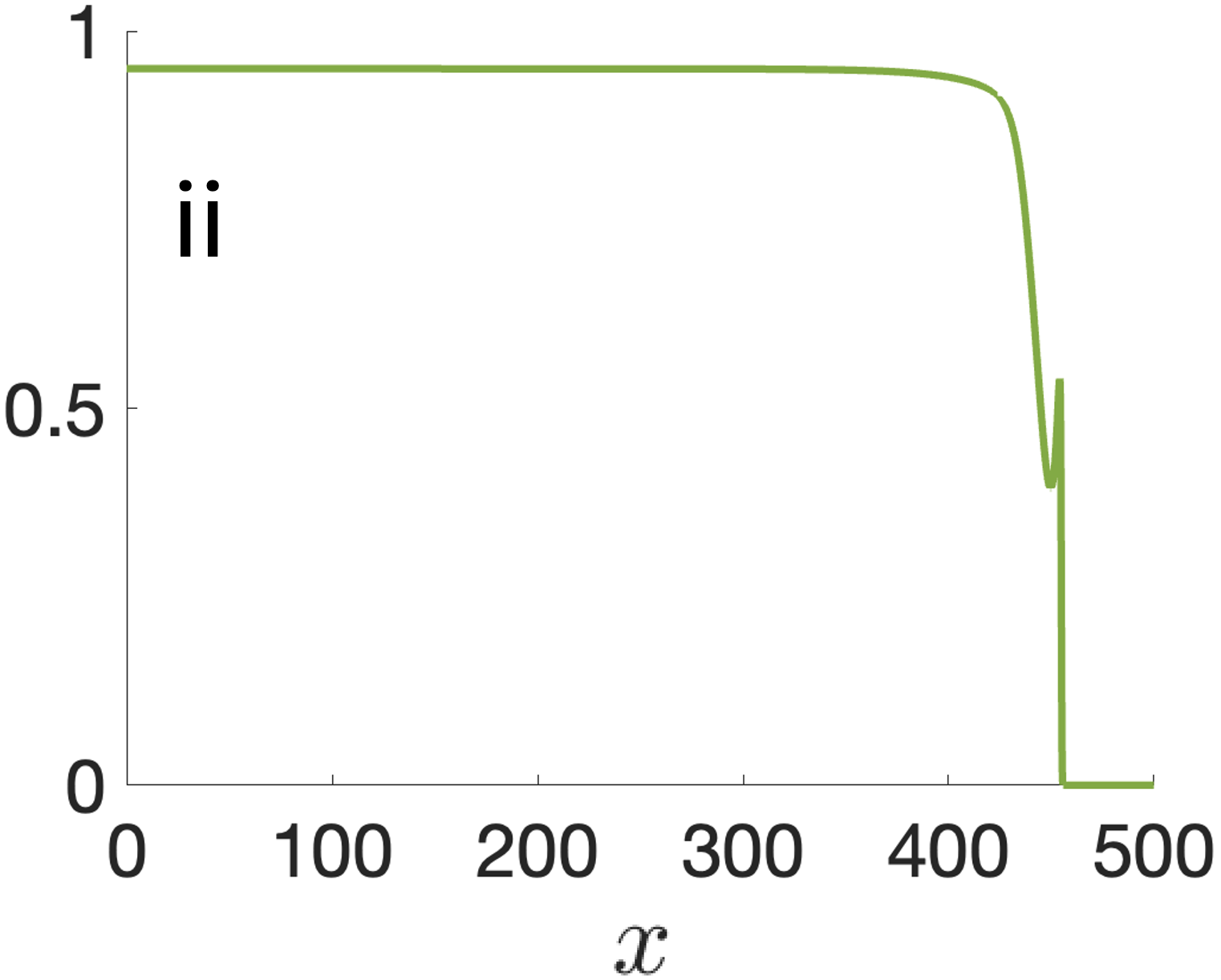}
        \caption{$k_1=2,k_2=3$}
        \label{Wang_Figure65}
    \end{subfigure}
    \begin{subfigure}[t]{0.24\textwidth}
        \includegraphics[width=\textwidth]{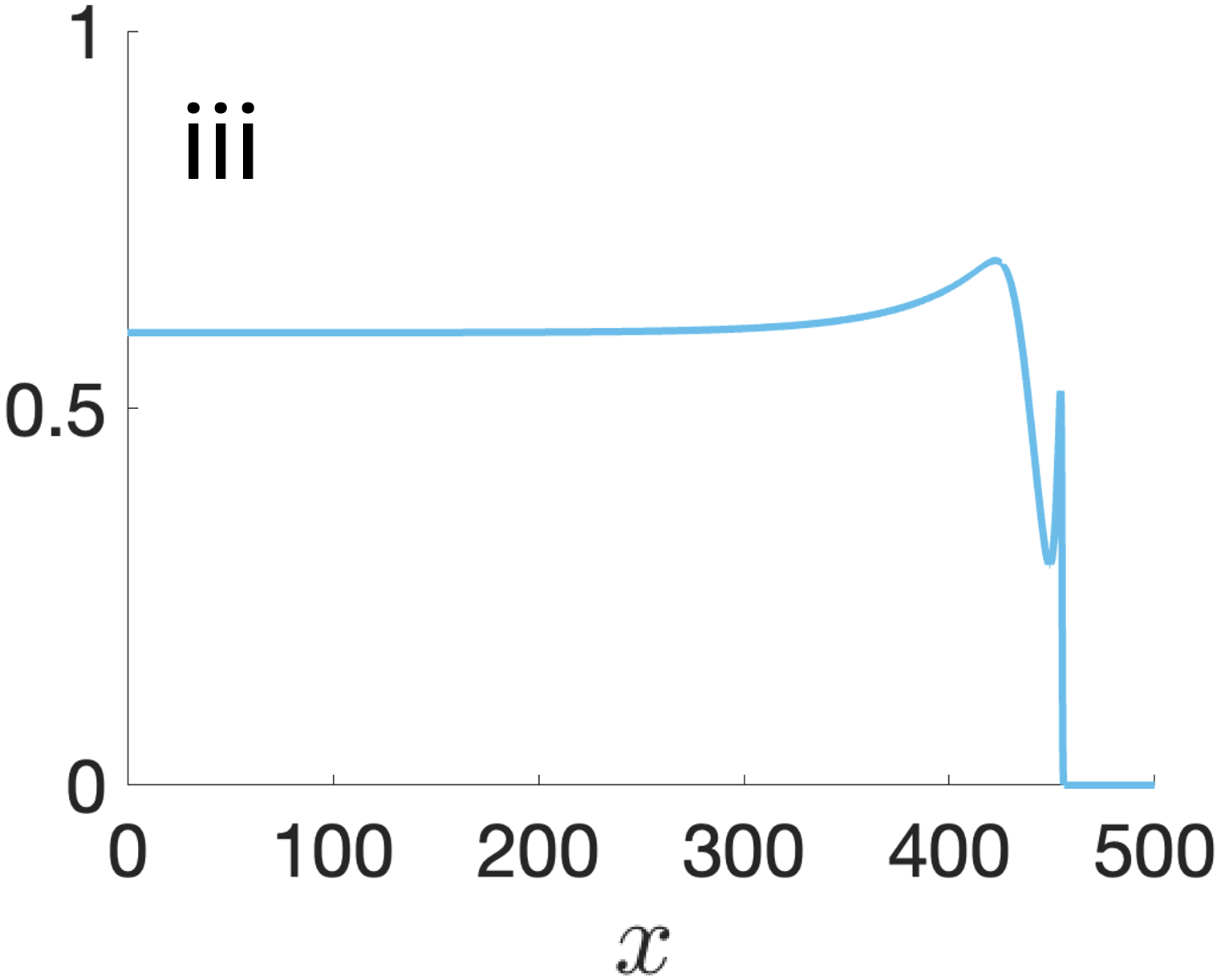}
        \caption{$k_1=-1,k_2=3$}
        \label{Wang_Figure66}
    \end{subfigure}
    \begin{subfigure}[t]{0.24\textwidth}
        \includegraphics[width=\textwidth]{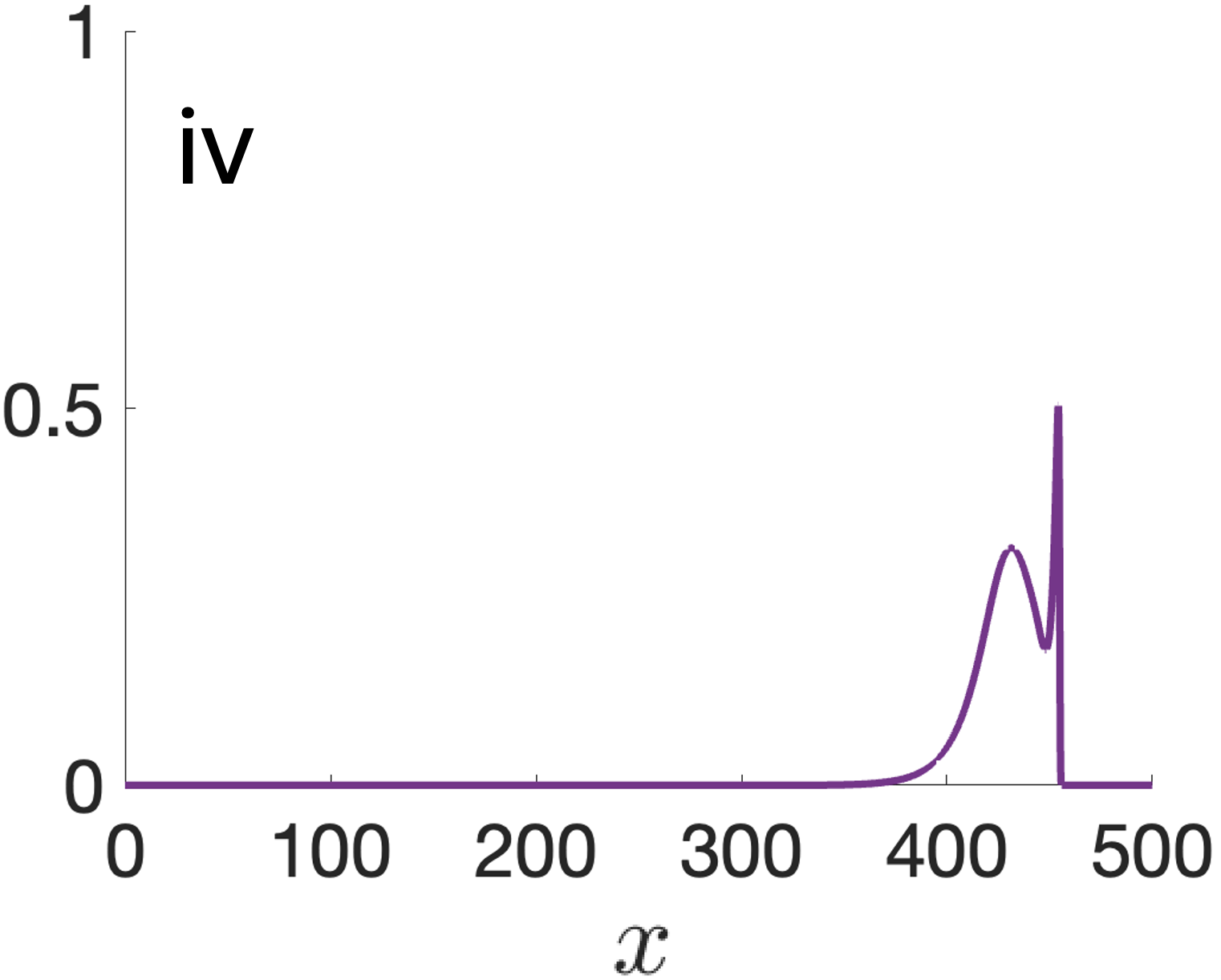}
        \caption{$k_1=-5,k_2=3$}
        \label{Wang_Figure67}
    \end{subfigure}
    \caption{(a) The locations of global maximum levels of protesting for $k_0=-1$ and $(k_1,k_2) \in [-5,5] \times [-5,5]$, while (b)-(e) are example protest activity profiles from (a) with the dots and labels indicating the $(k_1,k_2)$ coordinates (b)-(e) are from. Purple indicates regions where the global maximum occurs at the invading interface; green indicates regions where the global maximum is the steady state (occurs far behind the invading interface); and light blue indicates regions where the global maximum occurs at some intermediate location.}
    \label{Wang_Figure68}
\end{figure}

Based on the preceding discussion, it is established that all profiles in Figure \ref{Wang_Figure63} fall within the receding regime. For very large coordinates of $k_1$ and/or $k_2$, the decrease in activity immediately behind the invading interface is subtle, while Figures \ref{Wang_Figure65} and \ref{Wang_Figure66} exhibit a more visible decrease.

In Figure \ref{Wang_Figure63}, when holding $k_1>0$ fixed and transitioning from some $k_2<0$ to some $k_2>0$, there is a shift from the global maximum being situated at the invading interface (depicted in purple, e.g., Figure \ref{Wang_Figure64}) to being positioned far behind the wave (indicated in green, e.g., Figure \ref{Wang_Figure65}). Examining profiles along this trajectory from $k_2<0$ to $k_2>0$, we observe pulse-like profiles that eventually transform into non-pulse profiles, yet with $u(x_1,15) \approx u(x=0,15) \forall x_1<300$. As we progress to higher $k_2$ values, the residual level steadily increases until, at some point (dependent on the fixed $k_1>0$), the residual value surpasses the value at the invading interface. At this juncture, the system enters the green region.

Similarly, by fixing a $k_2>0$ in Figure \ref{Wang_Figure63} and descending from some $k_1>0$ to some $k_1<0$, there is a transition from the global maximum being the residual value (indicated in green, e.g., Figure \ref{Wang_Figure65}) to being positioned at some intermediary point between the invading interface and $u(x=0)$ (depicted in light blue, e.g., Figure \ref{Wang_Figure66}), and finally to the maximum being situated at the invading interface (depicted in purple, e.g., Figure \ref{Wang_Figure67}). Examining profiles along this trajectory from $k_1>0$ to $k_1 \ll 0$, we observe the residual value decreasing progressively until another inflection point emerges, and a local maximum forms at some $x \in (400,450)$. When this local maximum materializes, the system enters the light blue region, and this local maximum becomes the global maximum. With the continued reduction in the $k_1$ value, the residual value continues to decrease, gradually lowering the local maximum. Eventually, the local maximum falls below the value at the invading interface, marking the transition into the purple region.

\begin{figure} [H]
    \centering
    \begin{subfigure}[t]{0.5\textwidth}
        \includegraphics[width=\textwidth]{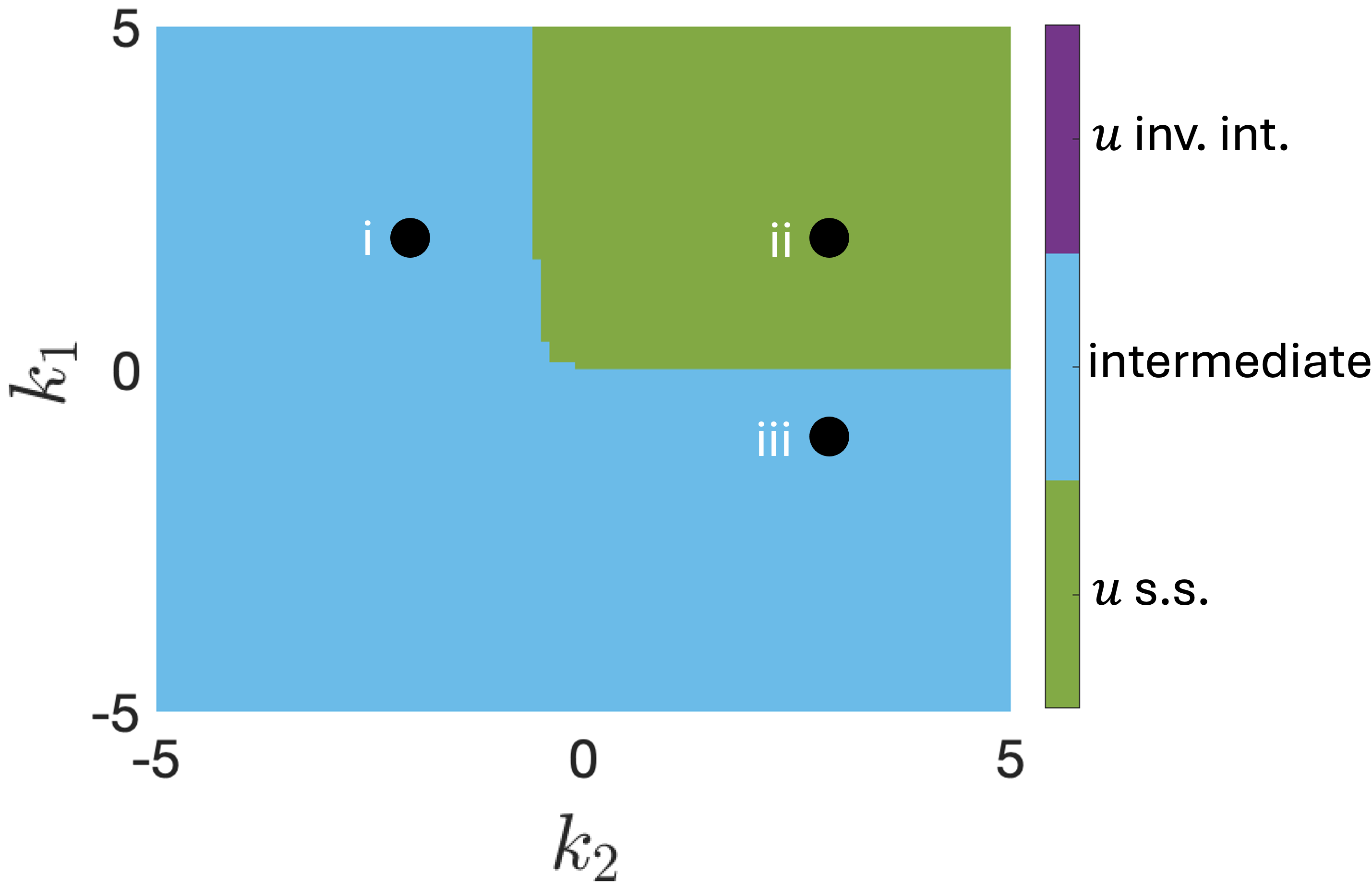}
        \caption{$k_0=1$}
        \label{Wang_Figure69}
    \end{subfigure}
    
    \begin{subfigure}[t]{0.24\textwidth}
        \includegraphics[width=\textwidth]{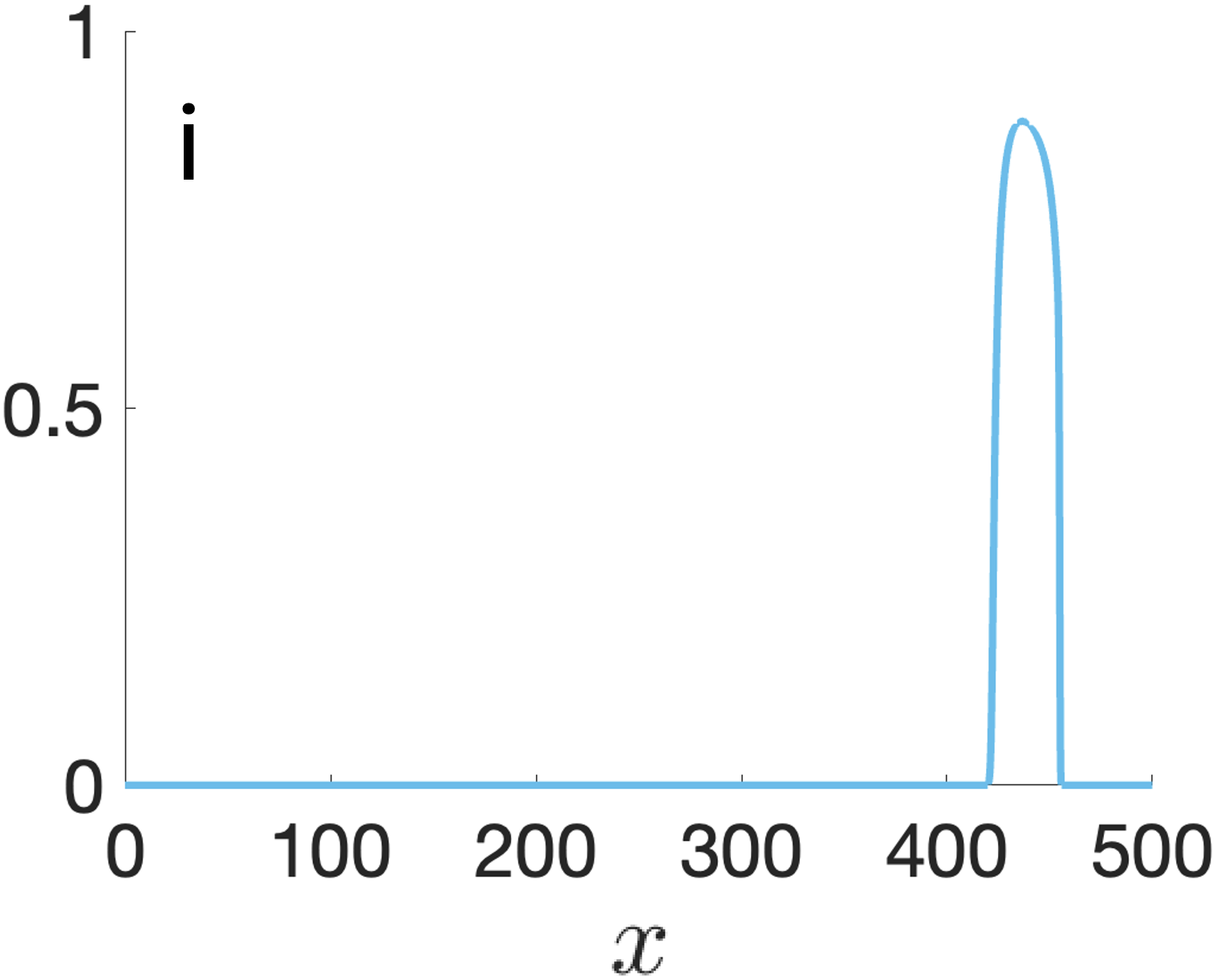}
        \caption{$k_1=2,k_2=-2$}
        \label{Wang_Figure70}
    \end{subfigure}
    \begin{subfigure}[t]{0.24\textwidth}
        \includegraphics[width=\textwidth]{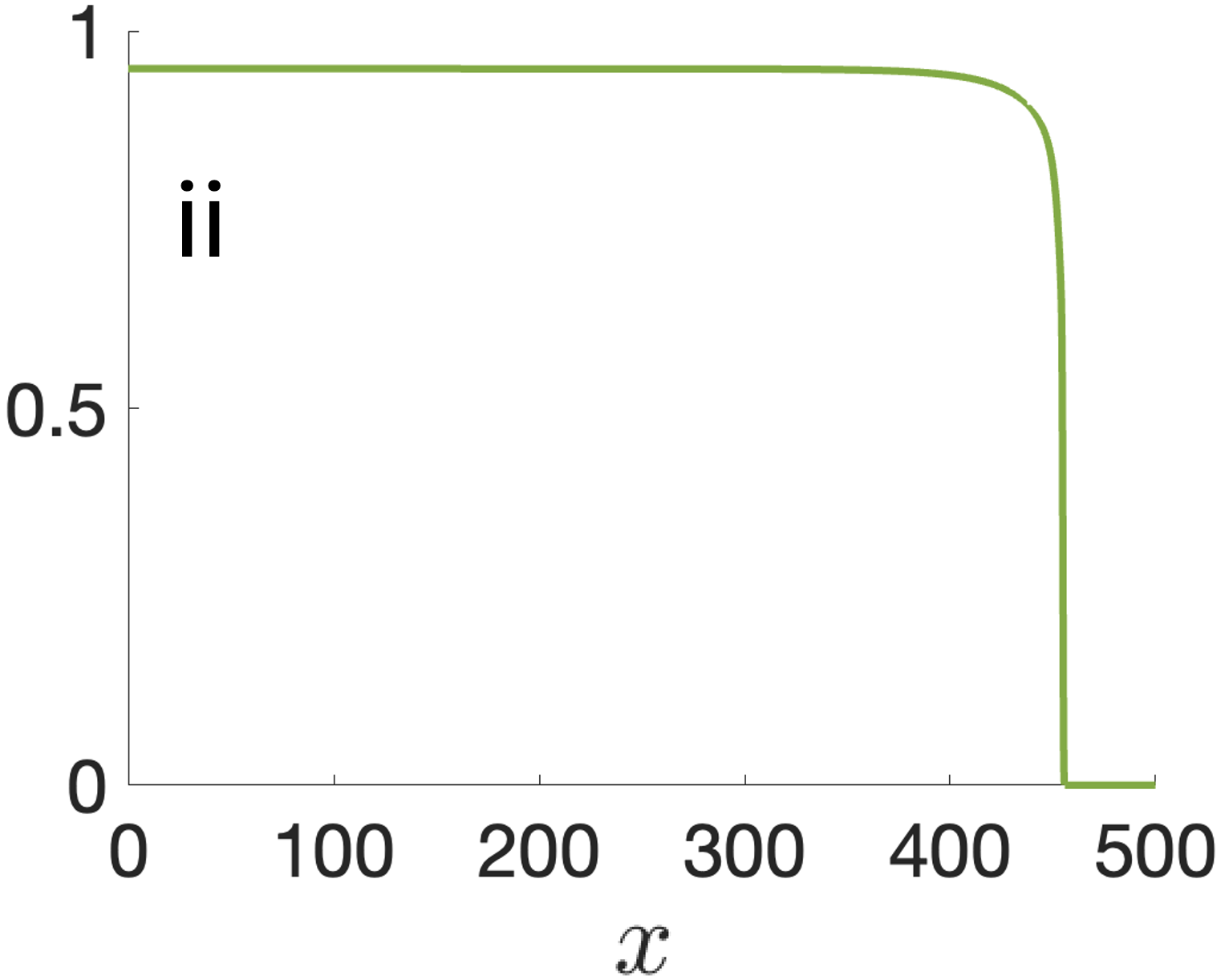}
        \caption{$k_1=2,k_2=3$}
        \label{Wang_Figure71}
    \end{subfigure}
    \begin{subfigure}[t]{0.24\textwidth}
        \includegraphics[width=\textwidth]{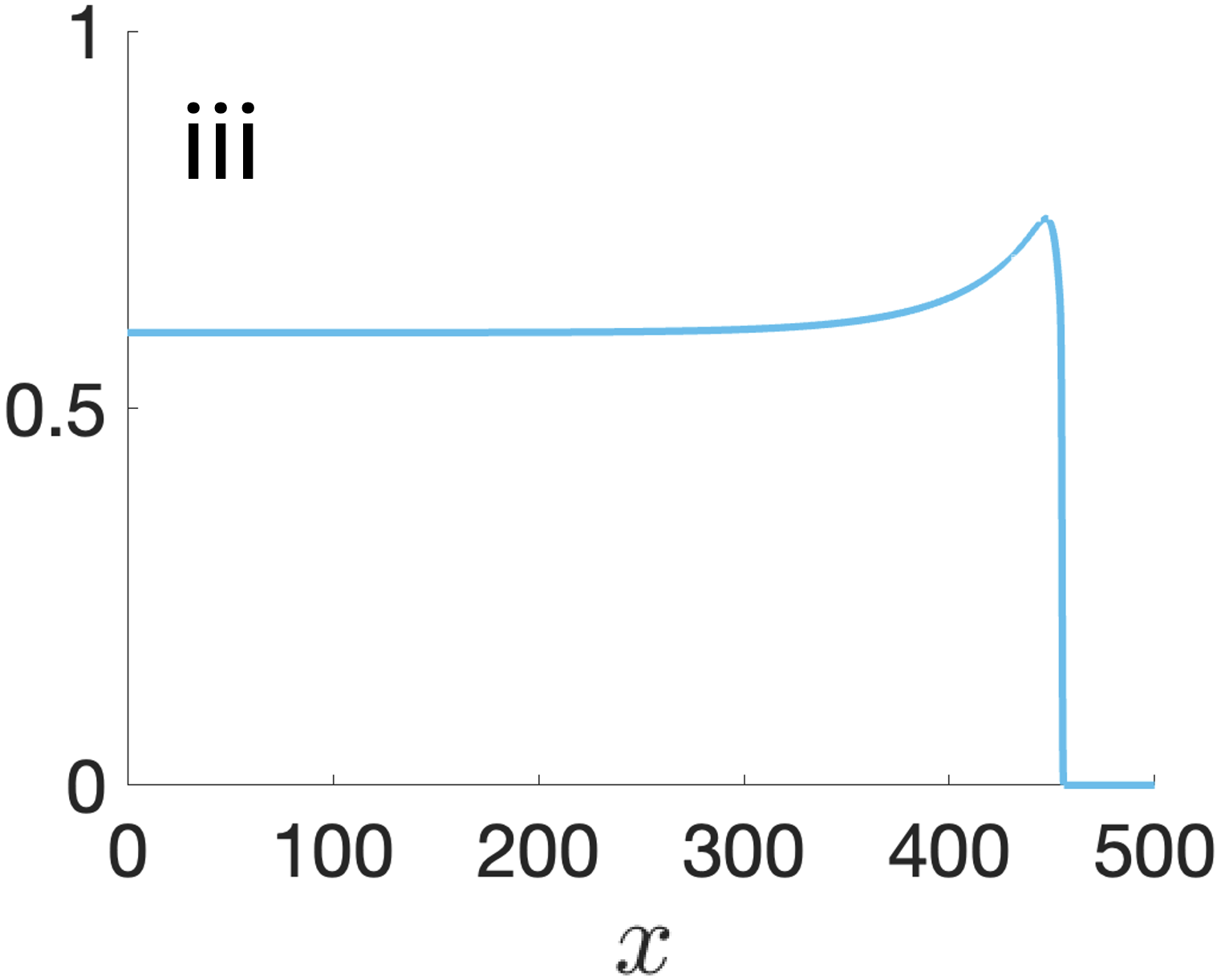}
        \caption{$k_1=-1,k_2=3$}
        \label{Wang_Figure72}
    \end{subfigure}
    \caption{(a) shows locations of global maximum levels of protesting for $k_0=-1$ and $(k_1,k_2) \in [-5,5] \times [-5,5]$, while (b)-(d) are example protest activity profiles from (a) with the dots and labels indicating the $(k_1,k_2)$ coordinates (b)-(d) are from. Purple indicates regions where the global maximum occurs at the invading interface; green indicates regions where the global maximum is the steady state (occurs far behind the invading interface); and light blue indicates regions where the global maximum occurs at some intermediate location.}
    \label{Wang_Figure73}
\end{figure}

In Figure \ref{Wang_Figure73}, for a fixed $k_1>0$ and moving from some $k_2<0$ to some $k_2>0$, we see a pulse-like profile (e.g., Figure \ref{Wang_Figure70}), which as $k_2$ increases will get wider and wider (e.g., the green curve $k_1=1$,$k_2=-1$ in Fig \ref{Wang_Figure31}) until the plateau is as wide as the grid and hence does not drop to near-zero. At this point, the maximum level is the residual level (e.g., Fig \ref{Wang_Figure71}) rather than occurring at some intermediate point. If starting at a fixed $k_2>0$ and moving from some $k_1>0$ down to some $k_1<0$, we observe a decrease in residual value, then eventually an emergence of an inflection point and formation of a local maximum at some $x \in (400,450)$, at which point we have transitioned from the green region (e.g., Figure \ref{Wang_Figure76}) to the light blue region (e.g., Fig \ref{Wang_Figure72}). As we continue moving along the fixed $k_1>0$ to lower and lower $k_2$ values, the residual value continues decreasing until it is classified as a near-pulse profile. Eventually, the pulse-like profile will become narrow enough that the algorithm for detecting the location of the global maximum will determine whether it occurs at the invading interface.

It should be noted that for simulations $k_0=1$, $k_1<-5$, and $k_2<-5$, our algorithm found global maxima occurring at the invading interface in clear contradiction to our observation that when $k_0>0$ there is an increase in the level of protesting immediately behind the invading interface. We believe that this contradiction is purely due to limitations of numerical simulations: we believe that in these parameter regimes, there is a local maximum behind the invading interface, but that these maxima fall between points on our discretized grid so will not be found in our algorithm.

\begin{figure} [H]
    \centering
    \begin{subfigure}[t]{0.5\textwidth}
        \includegraphics[width=\textwidth]{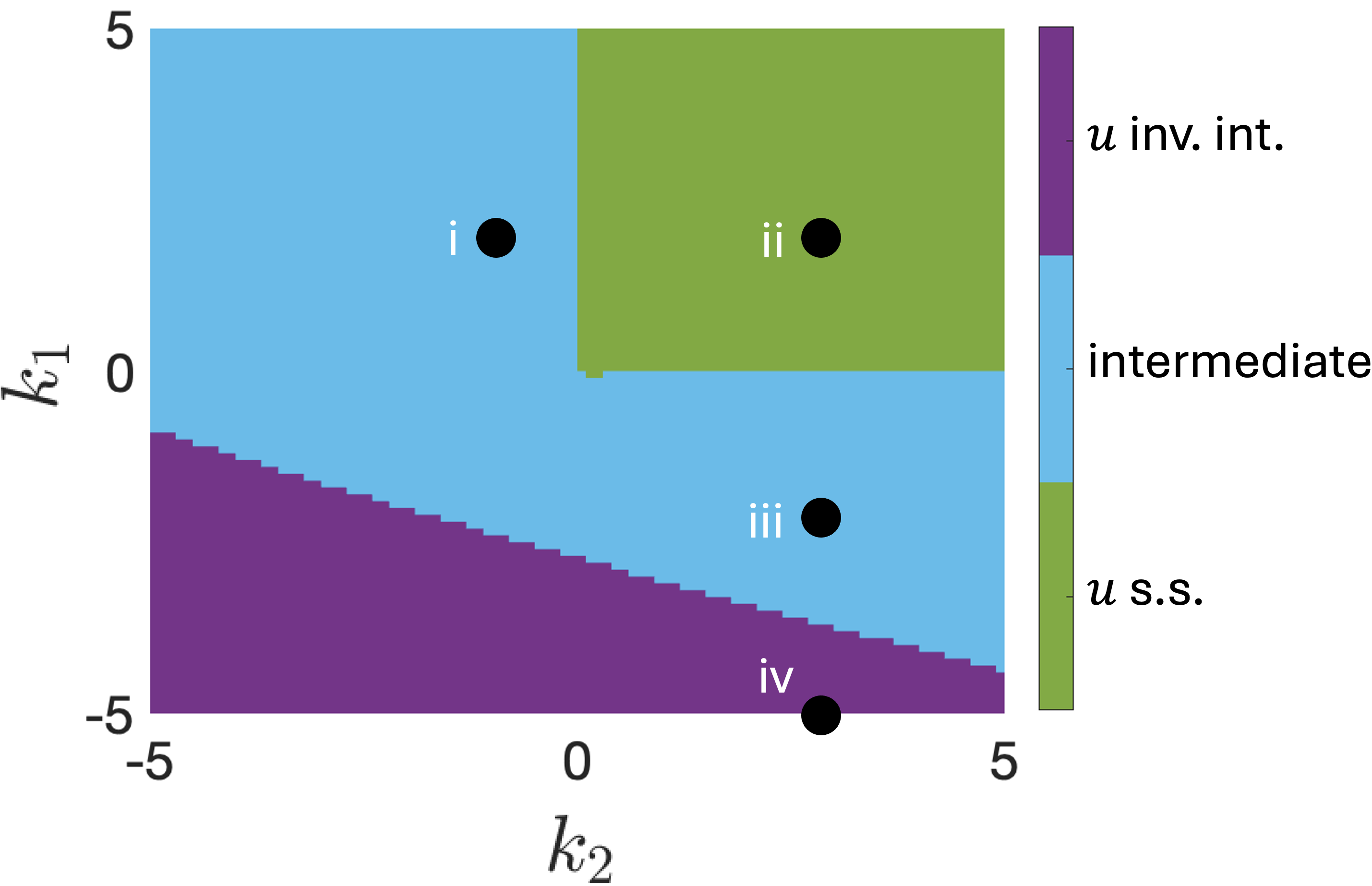}
        \caption{$k_0=0$}
        \label{Wang_Figure74}
    \end{subfigure}
    
    \begin{subfigure}[t]{0.24\textwidth}
        \includegraphics[width=\textwidth]{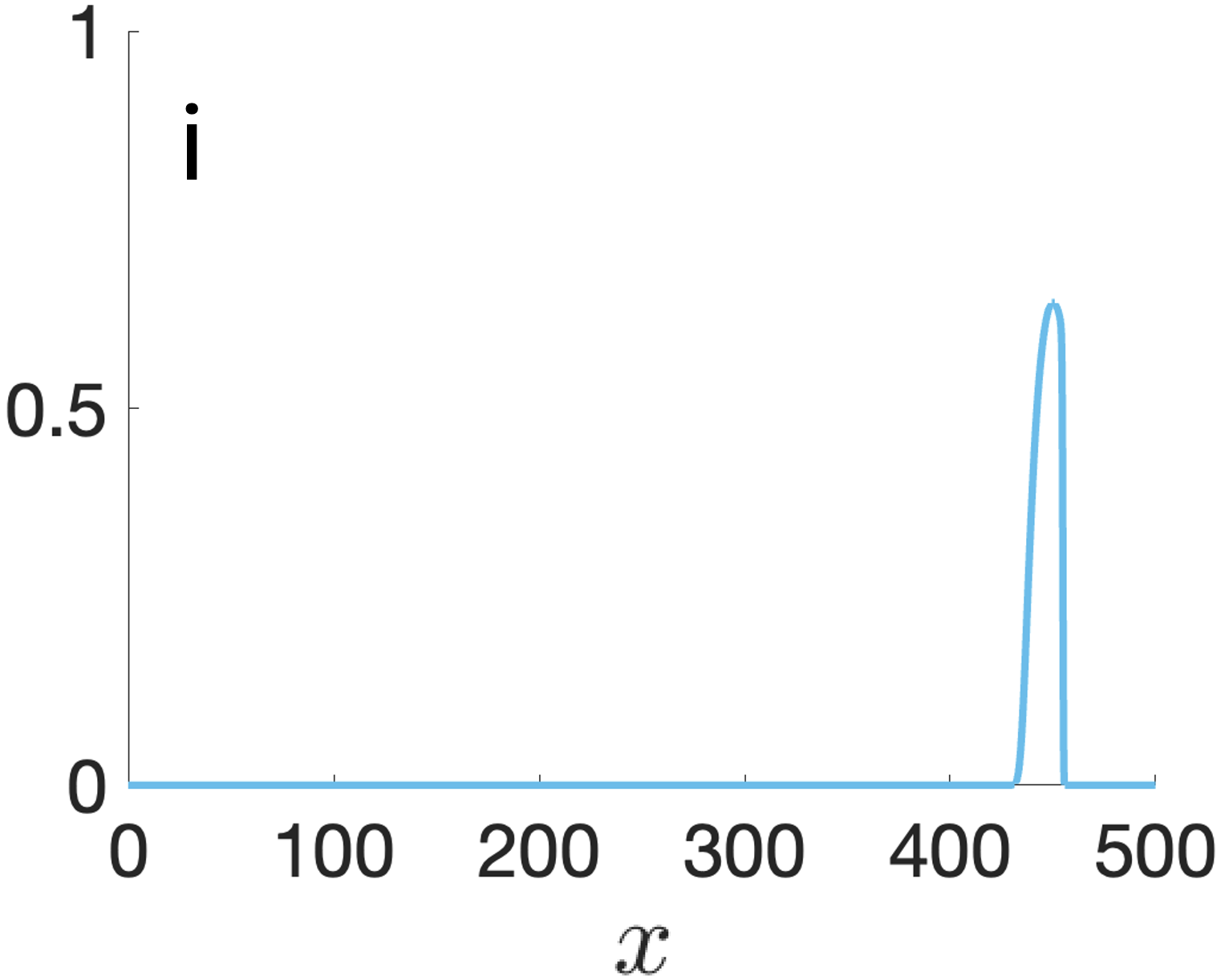}
        \caption{$k_1=2,k_2=-1$}
        \label{Wang_Figure75}
    \end{subfigure}
    \begin{subfigure}[t]{0.24\textwidth}
        \includegraphics[width=\textwidth]{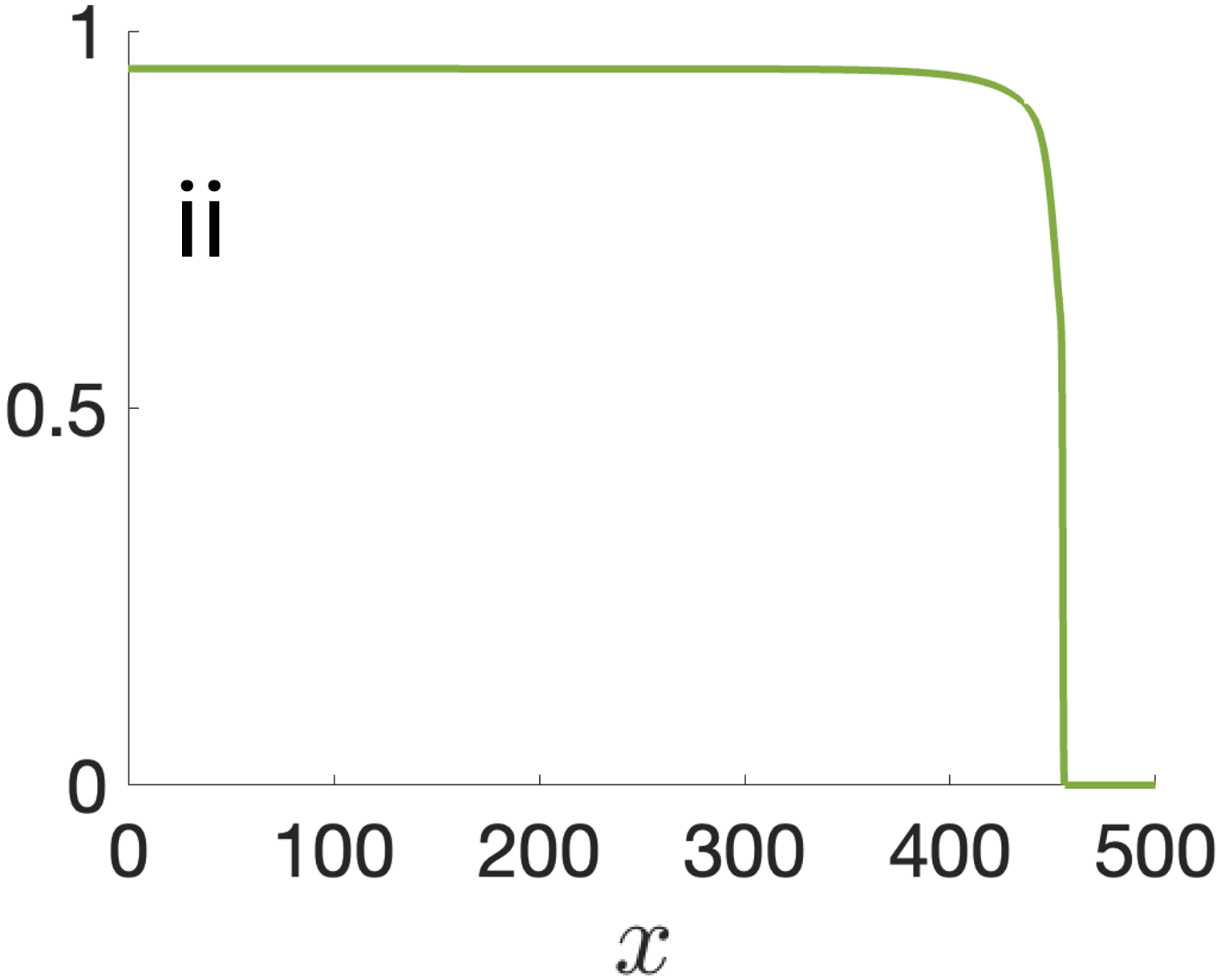}
        \caption{$k_1=2,k_2=3$}
        \label{Wang_Figure76}
    \end{subfigure}
    \begin{subfigure}[t]{0.24\textwidth}
        \includegraphics[width=\textwidth]{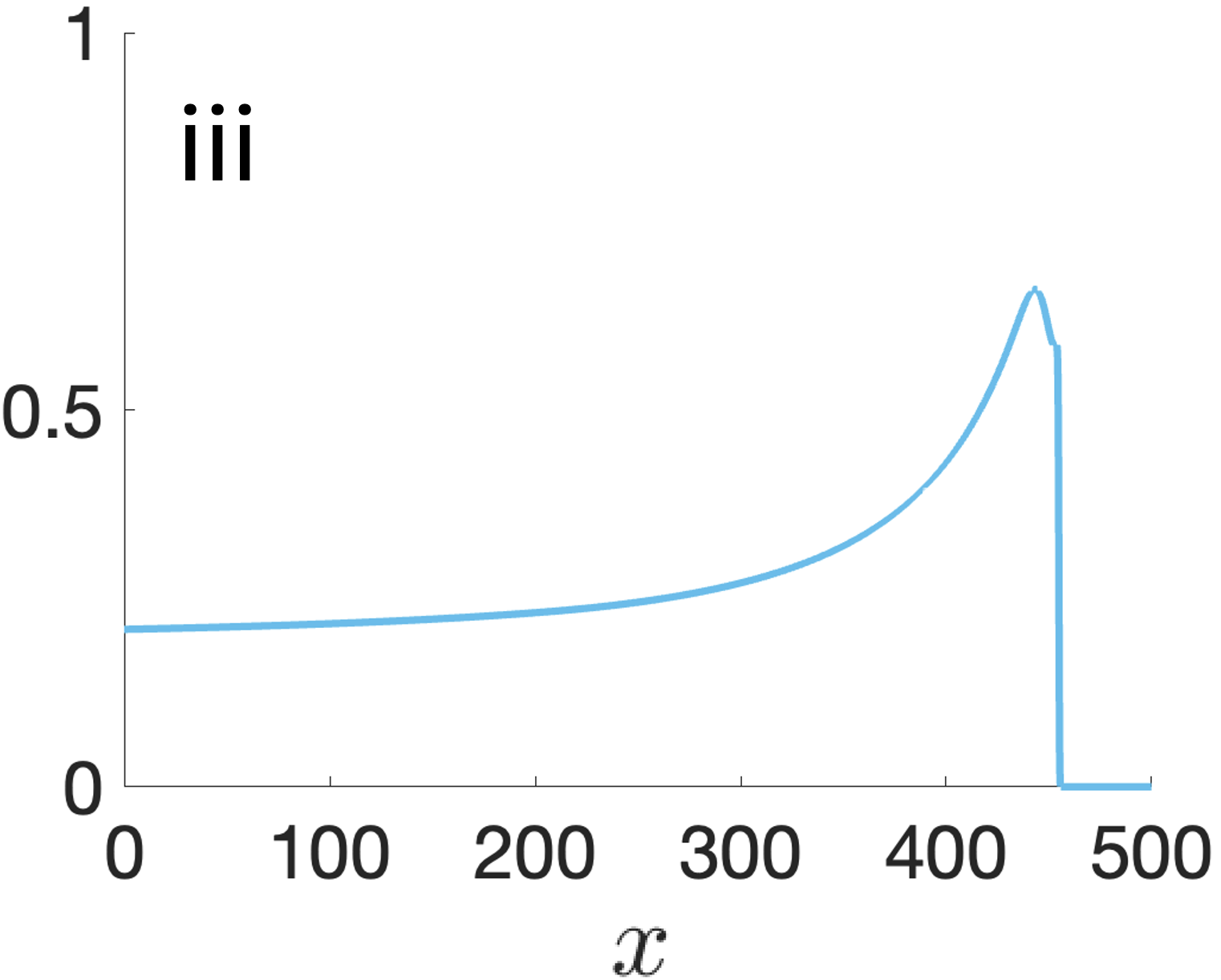}
        \caption{$k_1=-2,k_2=3$}
        \label{Wang_Figure77}
    \end{subfigure}
    \begin{subfigure}[t]{0.24\textwidth}
        \includegraphics[width=\textwidth]{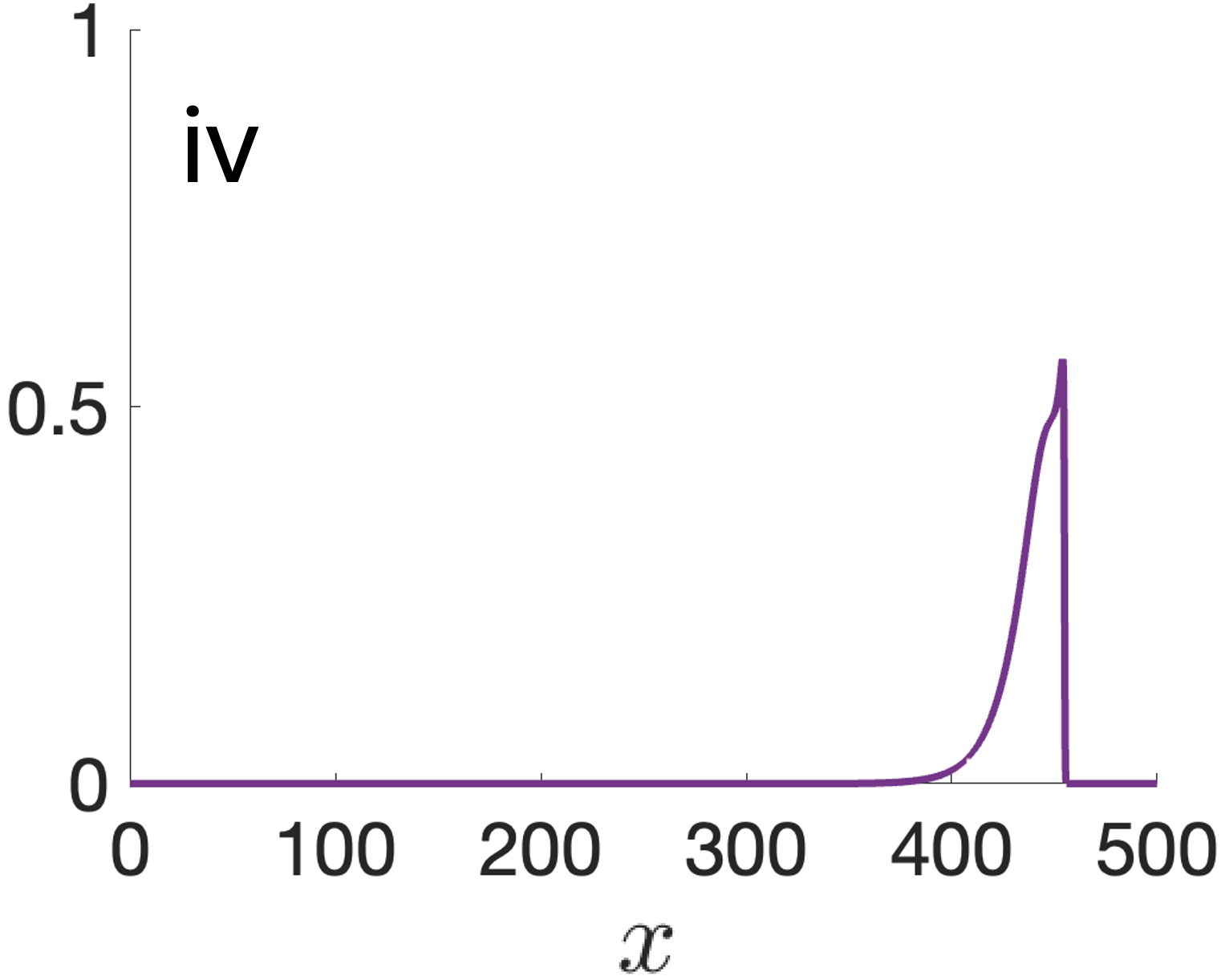}
        \caption{$k_1=-5,k_2=3$}
        \label{Wang_Figure78}
    \end{subfigure}
    \caption{(a) shows locations of global maximum levels of protesting for $k_0=0$ and $(k_1,k_2) \in [-5,5] \times [-5,5]$, while (b)-(e) are example protest activity profiles from (a) with the dots and labels indicating the $(k_1,k_2)$ coordinates (b)-(e) are from. Purple indicates regions where the global maximum occurs at the invading interface; green indicates regions where the global maximum is the steady state (occurs far behind the invading interface); and light blue indicates regions where the global maximum occurs at some intermediate location.}
    \label{Wang_Figure79}
\end{figure}

The heatmap Figure \ref{Wang_Figure74} and the profiles behind it are qualitatively similar to the heatmap Fig \ref{Wang_Figure69} and its profiles. One key difference is that pulse-like profiles for $k_1<0$ are much narrower in this $k_0=0$ case compared to the $k_0=1$ case above. Another key difference is that the observed expansion immediately behind the wavefront for $k_1<0$ and $k_2>0$ is also much narrower in the $k_0=0$ case compared to the $k_0=1$ case.

There is another obvious difference, which is the existence of a purple region. From earlier discussions, however, we know that for $k_0>0$, the protest activity is expanding immediately following the wavefront, so the global maximum cannot occur at the invading interface. When $k_0=0$, it is less clear whether there truly is a decrease immediately behind the invading interface or if there is an increase that is not observed in our numerical simulations. That is, it is unclear whether or not there should be a purple region, or if the purple region we see is entirely due to the limitations of numerical simulations. From visual inspection, Fig \ref{Wang_Figure78} appears to show a decrease in $u$ immediately behind the wavefront, but the actual global maximum value may occur between points on the discretized grid. Due to known limitations of the numerical methods, we cannot declare definitely that the purple regions have global maxima at their respective invading interfaces; rather, we can only state they are global maxima in our numerical simulations.

%%%%%%%%%%%%%%%%%%%%%%%%%%%%%%%%%%%%%%%%%%%%%%%%%%%%%%%%%%%%%%%%%%%%%%%%%%%%%%%%%%%%%%%%%%%%%%%%%%%%
%%%%%%%%%%%%%%%%%%%%%%%%%%%%%%%%%%%%%%%%%%%%%%%%%%%%%%%%%%%%%%%%%%%%%%%%%%%%%%%%%%%%%%%%%%%%%%%%%%%%
%%%%%%%%%%%%%%%%%%%%%%%%%%%%%%%%%%%%%%%%%%%%%%%%%%%%%%%%%%%%%%%%%%%%%%%%%%%%%%%%%%%%%%%%%%%%%%%%%%%%
%%%%%%%%%%%%%%%%%%%%%%%%%%%%%%%%%%%%%%%%%%%%%%%%%%%%%%%%%%%%%%%%%%%%%%%%%%%%%%%%%%%%%%%%%%%%%%%%%%%%
%%%%%%%%%%%%%%%%%%%%%%%%%%%%%%%%%%%%%%%%%%%%%%%%%%%%%%%%%%%%%%%%%%%%%%%%%%%%%%%%%%%%%%%%%%%%%%%%%%%%
%%%%%%%%%%%%%%%%%%%%%%%%%%%%%%%%%%%%%%%%%%%%%%%%%%%%%%%%%%%%%%%%%%%%%%%%%%%%%%%%%%%%%%%%%%%%%%%%%%%%
%%%%%%%%%%%%%%%%%%%%%%%%%%%%%%%%%%%%%%%%%%%%%%%%%%%%%%%%%%%%%%%%%%%%%%%%%%%%%%%%%%%%%%%%%%%%%%%%%%%%
%%%%%%%%%%%%%%%%%%%%%%%%%%%%%%%%%%%%%%%%%%%%%%%%%%%%%%%%%%%%%%%%%%%%%%%%%%%%%%%%%%%%%%%%%%%%%%%%%%%%
%%%%%%%%%%%%%%%%%%%%%%%%%%%%%%%%%%%%%%%%%%%%%%%%%%%%%%%%%%%%%%%%%%%%%%%%%%%%%%%%%%%%%%%%%%%%%%%%%%%%
%%%%%%%%%%%%%%%%%%%%%%%%%%%%%%%%%%%%%%%%%%%%%%%%%%%%%%%%%%%%%%%%%%%%%%%%%%%%%%%%%%%%%%%%%%%%%%%%%%%%

\section{Discussion}

We have presented a model to examine the dynamics of levels of social activity and social tension under the influence of policing. The parameters of the model were classified based on whether they led to a monotone or non-monotone system, and the stability of stationary solutions. In the monotone system, we established the existence of traveling waves, where the uniqueness of the wave speed depends on whether the source type is bistable or monostable. We have not established the existence or stability of traveling waves in the non-monotone system, a topic we aim to address in future work.

It is also intriguing and crucial to investigate solutions to system \eqref{Wang_Equation3} in non-monotone regimes where at least one of the three $k$ parameters has an inhibitory effect, as observed in real-world incident data (see \cite{Wang_Bonnasse-Gahot_2018}). Our numerical simulations confirm the existence of non-monotone wavefronts and near-pulse solutions within these regimes. Moreover, given that the classical theorems employed for asymptotic stability are inapplicable for non-monotone systems, exploring spectral stability for these regimes could be an initial step before delving into nonlinear stability. Additionally, it is natural to examine the solutions using bifurcation theory and identify and classify any potential bifurcations that may occur.

Another direction of interest is the study of system \eqref{Wang_Equation3} as a control problem and optimize the level of social activity or tension, or a combination of both. This would be a more direct approach to obtaining optimal policing strategies for relieving tension or reducing activities. Moreover, in addition to the self-reinforcement accounted for in the $u$ equation, exogenous factors can also be considered and incorporated into our model. For example, in the original model proposed in \cite{Wang_Berestycki_2015}, continuing external events that happen at particular times and locations have been included in the source term for social tension.

A limitation of our study is that we cannot provide a comprehensive analytic derivation to classify the stability of the nontrivial stationary solutions due to the implicit expression for the nontrivial steady states. Also, our model has not been fitted to any real-world dataset. The application of our system to datasets containing measurements of the variables $u, v, P$ is an intriguing possible course for future investigation.

From a numerical standpoint, we extensively investigate various parameter regimes and extract metrics with practical applications in mind. We envision that this research can contribute to categorizing different protesting activities as more data becomes available.\\

One crucial observation is that the effectiveness of a police management strategy can vary depending on the circumstances. For instance, employing an escalated force approach might initially reduce social tension, followed by an increase if $k_0<0$. Conversely, when $k_0>0$, tension and activity typically rise continuously. Regardless of the situation, the escalated force approach tends to amplify residual social tension, even if activity remains unchanged. However, there are instances where police presence is indispensable. For instance, in cases where $k_0>0$ and unrest escalates in the absence of police, adopting a negotiated management model can ameliorate protest outcomes. The condition $k_0>0$ arises when there are counterprotesters or agitators within a protest, possibly leading to violence regardless of police intervention. In such scenarios, we observe decreased social tension and activity, sometimes accompanied by traveling pulses.

The analysis suggests that the negotiated management model is generally preferable due to its ability to reduce social tension. However, it's important to note that the costs associated with action haven't been considered. Accounting for these costs might warrant reducing action in certain situations, even if tension increases. However, relying on an escalated management model seldom yields positive outcomes, particularly since the system's state is often unknown. Developing methods to measure social tension and understand the demographics of potential participants can assist in selecting an appropriate strategy.
\\

{\bf Acknowledgement:}  Rodriguez partially funded by NSF-DMS-2042413 and AFOSR MURI FA9550-22-1-0380.  Brantingham, Wang, Wessler were funded by AFOSR MURI FA9550-22-1-0380.

\section{Appendix A. Asymptotic decay rates for $k_0 \ge 0$}
\label{sec:appendix-asymptotic-rates}
Assume that the system has traveling wave solutions $\hat{\bm{u}}(z) = (\hat{u}(z),\hat{v}(z),\hat{P}(z))$, where we change the coordinate system to $z = x-ct$ and $c$ is the wave speed. Then, system ~\eqref{Wang_Equation3} becomes
\begin{align*}
    \begin{cases}
        -cu_z = u_{zz} + f(u,v,P), \\
        -cv_z = v_{zz} + g(u,v,P), \\
        -cP_z = s(u,v,P),
    \end{cases}
\end{align*}
which can be written as a five-dimensional, first-order system
\begin{align*}
    \begin{cases}
        u' = w, \\
        v' = y, \\
        P' = -\frac{1}{c} s(u,v,P), \\
        w' = -cw - f(u,v,P), \\
        y' = -cy - g(u,v,P).
    \end{cases}
\end{align*}
The linearization of this system near the fixed points is
\begin{align}
\label{Wang_Equation16}
U' =
\begin{bmatrix}
    u \\ v \\ P \\ w \\ y
\end{bmatrix}'
=
\begin{bmatrix}
    0 & 0 & 0 & 1 & 0 \\
    0 & 0 & 0 & 0 & 1 \\
    -\frac{1}{c}s_u(u,v,P) & 0 & -\frac{1}{c}s_P(u,v,P) & 0 & 0 \\
    -f_u(u,v,P) & -f_v(u,v,P) & -f_P(u,v,P) & -c & 0 \\
    -g_u(u,v,P) & -g_v(u,v,P) & -g_P(u,v,P) & 0 & -c
\end{bmatrix}
\begin{bmatrix}
    u \\ v \\ P \\ w \\ y
\end{bmatrix}
=
AU.
\end{align}
Recall from the previous sections, that the system has the trivial steady states $(0, (1+P)^{-k_2}, P, 0, 0)$ and, in certain parameter regimes, also the non-trivial steady states $(u^*, v^*, 1, 0, 0).$
The linearization matrix $A$ at the trivial steady states is
\begin{align*}
    A \rvert_{(0, (1+P)^{-k_2}, P, 0, 0)} = 
    \begin{bmatrix}
    0 & 0 & 0 & 1 & 0 \\
    0 & 0 & 0 & 0 & 1 \\
    -\frac{1}{c}s_u(\hat{\bm{u}}_+) & 0 & 0 & 0 & 0 \\
    -f_u(\hat{\bm{u}}_+) & 0 & 0 & -c & 0 \\
    -g_u(\hat{\bm{u}}_+) & -g_v(\hat{\bm{u}}_+) & -g_P(\hat{\bm{u}}_+) & 0 & -c
\end{bmatrix}
\end{align*}
and has eigenvalues
\begin{align*}
\begin{cases}
    \lambda_1 = 0, \\
    \lambda_{2,3} = \frac{-c \pm \sqrt{c^2 - 4f_u(\hat{\bm{u}}_+)}}{2}, \\
    \lambda_{4,5} = \frac{-c \pm \sqrt{c^2 - 4g_v(\hat{\bm{u}}_+)}}{2},
\end{cases}
\end{align*}
where the plus sign is for $\lambda_{2,4}$, and minus sign is for $\lambda_{3,5}$. This convention holds for the entire paper.

The matrix at the non-trivial steady states is
\begin{align*}
    A \rvert_{(u^*, v^*, 1, 0, 0)} = 
    \begin{bmatrix}
   0 & 0 & 0 & 1 & 0 \\
    0 & 0 & 0 & 0 & 1 \\
    0 & 0 & -\frac{1}{c}s_P(\hat{\bm{u}}_-) & 0 & 0 \\
    -f_u(\hat{\bm{u}}_-) & -f_v(\hat{\bm{u}}_-) & -f_P(\hat{\bm{u}}_-) & -c & 0 \\
    -g_u(\hat{\bm{u}}_-) & -g_v(\hat{\bm{u}}_-) & -g_P(\hat{\bm{u}}_-) & 0 & -c
\end{bmatrix}
\end{align*}
with eigenvalues
\begin{align*}
\begin{cases}
    \lambda_1 &= -\frac{1}{c}s_P(\bm{\hat{\bm{u}}_-}), \\
    \lambda_{2,3} &= \frac{-c \pm \sqrt{c^2 + 2\mu_2}}{2}, \\
    \lambda_{4,5} &= \frac{-c \pm \sqrt{c^2 + 2\mu_1}}{2},
\end{cases}
\end{align*}
where
\begin{align*}
    \mu_{1,2} = -(f_u(\hat{\bm{u}}_-) + g_v(\hat{\bm{u}}_-)) \pm \sqrt{(f_u(\hat{\bm{u}}_-) + g_v(\hat{\bm{u}}_-))^2 + 4f_v(\hat{\bm{u}}_-) g_u(\hat{\bm{u}}_-) - 4f_u(\hat{\bm{u}}_-) g_v(\hat{\bm{u}}_-)}.
\end{align*}

\subsection{As $z \to +\infty$}
As $z \to +\infty$, $\hat{\bm{u}} \to \hat{\bm{u}}_+ = (0,(1+P)^{-k_2},P)$. To avoid having complex eigenvalues, we restrict that $c^2 - 4f_u(\hat{\bm{u}}_+) \ge 0$ and $c^2 - 4g_v(\hat{\bm{u}}_+) \ge 0$. The second inequality holds automatically since $g_v(\hat{\bm{u}}_+) < 0$ always. The first inequality leads to the condition
\begin{align*}
    c \ge 2\sqrt{l},
\end{align*}
where we define $l:=\text{max}(f_u(\hat{\bm{u}}_+),0)$.
To find the principal eigenvalue, which is the negative eigenvalue with the smallest magnitude, we discuss by cases depending on the value of $f_u(\hat{\bm{u}}_+)$.

\textbf{Case 1} If $f_u(\hat{\bm{u}}_+ < g_v(\hat{\bm{u}}_+))$, then $\lambda_3 < \lambda_5 < 0 < \lambda_4 < \lambda_2$, and $\lambda_5$ is the principal eigenvalue. The corresponding eigenvector is
\begin{align*}
    \bm{y}_5 = 
    \begin{bmatrix}
        0 \\
        - \cfrac{\lambda_5 + c}{g_v(\hat{\bm{u}}_+)} \\
        0 \\
        0 \\
        1
    \end{bmatrix},
\end{align*}
and the linear approximation about the steady state $\hat{\bm{u}}_+$ is
\begin{align*}
    \begin{bmatrix}
        u(z) \\
        v(z) - (1+P)^{-k_2} \\
        P(z) - P \\
        w(z) \\
        y(z)
    \end{bmatrix}
    \approx A_1\bm{y}_5 e^{\lambda_5 z},
\end{align*}
as $z \to +\infty$, for some constant $A_1$. Since the second entry in $\bm{y}_5$ is negative, we know that $A_1 < 0$ so that the wave is decaying.

\textbf{Case 2} If $f_u(\hat{\bm{u}}_+) = g_v(\hat{\bm{u}}_+)$, then $\lambda_3 = \lambda_5 < 0 < \lambda_2 = \lambda_4$, and $\lambda_3 = \lambda_5$ is the principal eigenvalue. The linear approximation about $\hat{\bm{u}}_+$ is
\begin{align*}
    \begin{bmatrix}
        u(z) \\
        v(z) - (1+P)^{-k_2} \\
        P(z) - P \\
        w(z) \\
        y(z)
    \end{bmatrix}
    \approx A_2 z \bm{y}_5 e^{\lambda_5 z} + O(e^{\lambda_5 z}),
\end{align*}
as $z \to +\infty$, for some constant $A_2$. Again, $A_2 < 0$ since we need the wave to be decaying.

\textbf{Case 3} If $g_v(\hat{\bm{u}}_+) < f_u(\hat{\bm{u}}_+) < 0$, then $\lambda_5 < \lambda_3 < 0 < \lambda_2 < \lambda_4$, and $\lambda_3$ is the principal eigenvalue. The corresponding eigenvector is
\begin{align*}
    \bm{y}_3 = 
    \left.\begin{bmatrix}
        \cfrac{g_v - f_u}{-\lambda_3 g_u + \frac{1}{c} g_P s_u} \\
        - \cfrac{\lambda_3 + c}{f_u} \\
        \cfrac{s_u(g_v-f_u)}{\lambda_3 c(\lambda_3 g_u - \frac{1}{c}g_P s_u)} \\
        \cfrac{\lambda_3(g_v-f_u)}{-\lambda_3 g_u + \frac{1}{c} g_P s_u} \\
        1
    \end{bmatrix}\right\rvert_{\hat{\bm{u}}_+},
\end{align*}
and the linear approximation about $\hat{\bm{u}}_+$ is
\begin{align*}
    \begin{bmatrix}
        u(z) \\
        v(z) - (1+P)^{-k_2} \\
        P(z) - P \\
        w(z) \\
        y(z)
    \end{bmatrix}
    \approx A_3\bm{y}_3 e^{\lambda_3 z},
\end{align*}
as $z \to +\infty$, for some constant $A_3$. One can check that the first, second, and third entries in $\bm{y}_3$ are negative, and the fourth and fifth entries are positive. This implies that $A_3 < 0$ to guarantee the decay of the wave.

\textbf{Case 4} If $f_u(\hat{\bm{u}}_+) = 0$, then $\lambda_5 < \lambda_3 < 0 = \lambda_2 < \lambda_4$. The linear approximation about $\hat{\bm{u}}_+$ is
\begin{align*}
    \begin{bmatrix}
        u(z) \\
        v(z) - (1+P)^{-k_2} \\
        P(z) - P \\
        w(z) \\
        y(z)
    \end{bmatrix}
    \approx A_4 \bm{y}_2 + O(e^{\lambda_3 z}),
\end{align*}
as $z \to +\infty$, for some constant $A_4$. Therefore, we cannot talk about asymptotic decay rates in this case.

\textbf{Case 5} If $f_u(\hat{\bm{u}}_+) > 0$, then there are two subcases:

\textbf{Case 5.1} If $c > 2\sqrt{f_u(\hat{\bm{u}}_+)}$, then $\lambda_5 < \lambda_3 < \lambda_2 < 0 < \lambda_4$, and $\lambda_2$ is the principal eigenvalue. The corresponding eigenvector is
\begin{align*}
    \bm{y}_2 = 
    \left.\begin{bmatrix}
        \cfrac{g_v - f_u}{-\lambda_2 g_u + \frac{1}{c} g_P s_u} \\
        - \cfrac{\lambda_2 + c}{f_u} \\
        \cfrac{s_u(g_v-f_u)}{\lambda_2 c (\lambda_2 g_u - \frac{1}{c}g_P s_u)} \\
        \cfrac{\lambda_2(g_v-f_u)}{-\lambda_2 g_u + \frac{1}{c} g_P s_u} \\
        1
    \end{bmatrix}\right\rvert_{\hat{\bm{u}}_+},
\end{align*}
and the linear approximation about $\hat{\bm{u}}_+$ is
\begin{align*}
    \begin{bmatrix}
        u(z) \\
        v(z) - (1+P)^{-k_2} \\
        P(z) - P \\
        w(z) \\
        y(z)
    \end{bmatrix}
    \approx A_5\bm{y}_2 e^{\lambda_2 z},
\end{align*}
as $z \to +\infty$, for some constant $A_5$. Similarly to the previous cases, the first three entries in $\bm{y}_2$ are negative and the last two are positive, implying that $A_5 < 0$.

\textbf{Case 5.2} If $c = 2\sqrt{f_u(\hat{\bm{u}}_+)}$, then $\lambda_5 < \lambda_3 = \lambda_2 < 0 < \lambda_4$, and $\lambda_2 = \lambda_3$ is the principal eigenvalue. The linear approximation about $\hat{\bm{u}}_+$ is
\begin{align*}
    \begin{bmatrix}
        u(z) \\
        v(z) - (1+P)^{-k_2} \\
        P(z) - P \\
        w(z) \\
        y(z)
    \end{bmatrix}
    \approx A_6 z \bm{y}_2 e^{\lambda_2 z} + O(e^{\lambda_2 z}),
\end{align*}
as $z \to +\infty$, for some negative constant $A_6$.

\subsection{As $z \to -\infty$}
As $z \to -\infty$, $\hat{\bm{u}} \to \hat{\bm{u}}_- = (u*,v*,1,0,0)$. The principal eigenvalue is the positive eigenvalue with the smallest magnitude. Note that $\lambda_3$ and $\lambda_5$ are always negative, and that $\lambda_2$ and $\lambda_4$ are nonnegative if and only if $\mu_2$ and $\mu_1$ are nonnegative, respectively. Therefore, we want to discuss by cases depending on the signs of $\mu_1,\mu_2$. We find numerically that $\mu_1,\mu_2$ are always real for $k_0 \ge 0$, and are not always real for $k_0 < 0$. As for their signs, recall from section \ref{sec:Fixed point analysis on the spatially homogeneous system} that $f_u + g_v < 0$, implying that $\mu_1 > 0$ always holds, and that $\mu_2 > 0$ if and only if $f_v(\hat{\bm{u}}_-)g_u(\hat{\bm{u}}_-) - f_u(\hat{\bm{u}}_-)g_v(\hat{\bm{u}}_-) < 0$. This inequality condition can be numerically shown true for $k_0 \ge 0$ and as a result, $\mu_2 > 0$ always holds.

Since now we know that $\lambda_{1,2,4}$ are always positive and that $\lambda_2 \le \lambda_4$, let us discuss by a few cases depending on what the principle eigenvalue is as $z \to -\infty$. 

\textbf{Case 1} If $\lambda_1 < \lambda_2 \le \lambda_4$, then $\lambda_1$ is the principle eigenvalue. The corresponding eigenvector is
\begin{align*}
    \bm{l}_1 = 
    \left.\begin{bmatrix}
        \cfrac{c^3(f_v g_P + f_P(s_P-g_v)) - c f_P s_P^2}{\eta s_P} \\
        \cfrac{1}{\lambda_1} \\
        \cfrac{c^4(-f_v g_u + (f_u - s_P)(g_v - s_P)) + c^2 (f_u + g_v - 2s_P)s_P^2 + s_P^4}{c \eta s_P} \\
        \cfrac{- c^2(f_v g_P + f_P(s_P - g_v)) + f_P s_P^2}{\eta} \\
        1
    \end{bmatrix}\right\rvert_{\hat{\bm{u}}_-},
\end{align*}
where $\eta = g_P s_P^2 - c^2 (-f_u g_P + f_Pg_u + g_Ps_P)$.
Therefore, the linear approximation about $\hat{\bm{u}}_-$ is
\begin{align*}
    \begin{bmatrix}
        u^* - u(z) \\
        v^* - v(z) \\
        1 - P(z) \\
        -w(z) \\
        -y(z)
    \end{bmatrix}
    \approx B_1 \bm{l}_1 e^{\lambda_1 z},
\end{align*}
as $z \to -\infty$, for some positive constant $B_1$.

\textbf{Case 2} If $\lambda_1 = \lambda_2 < \lambda_4$, then the linear approximation about $\hat{\bm{u}}_-$ is
\begin{align*}
    \begin{bmatrix}
        u^* - u(z) \\
        v^* - v(z) \\
        1 - P(z) \\
        -w(z) \\
        -y(z)
    \end{bmatrix}
    \approx B_2 z \bm{l}_1 e^{\lambda_1 z} + O(e^{\lambda_1 z}),
\end{align*}
as $z \to -\infty$, for some negative constant $B_2$.

\textbf{Case 3} If $\lambda_1 = \lambda_2 = \lambda_4$, then we have three repeated eigenvectors, and the linear approximation about $\hat{\bm{u}}_-$ is
\begin{align*}
    \begin{bmatrix}
        u^* - u(z) \\
        v^* - v(z) \\
        1 - P(z) \\
        -w(z) \\
        -y(z)
    \end{bmatrix}
    \approx B_3 z^2 \bm{l}_1 e^{\lambda_1 z} + O(ze^{\lambda_1 z}),
\end{align*}
as $z \to -\infty$, for some positive constant $B_3$.

\textbf{Case 4} If $\lambda_2 = \lambda_4 < \lambda_1$, then the corresponding eigenvector is
\begin{align*}
    \bm{l}_2 = 
    \left.\begin{bmatrix}
        \cfrac{\left (c + \sqrt{c^2 + 2\mu_2} \right ) \left (\mu_1 + 2f_u \right )}{2g_u \mu_2} \\
        -\cfrac{2}{c - \sqrt{c^2 + 2\mu_2}} \\
        0 \\
        \cfrac{\mu_1 + 2f_u}{2g_u} \\
        1
    \end{bmatrix}\right\rvert_{\hat{\bm{u}}_-},
\end{align*}
and the linear approximation about $\hat{\bm{u}}_-$ is
\begin{align*}
    \begin{bmatrix}
        u^* - u(z) \\
        v^* - v(z) \\
        1 - P(z) \\
        -w(z) \\
        -y(z)
    \end{bmatrix}
    \approx B_4 z \bm{l}_2 e^{\lambda_2 z} + O(e^{\lambda_2 z}),
\end{align*}
as $z \to -\infty$, for some negative constant $B_4$.

\textbf{Case 5} If $\lambda_2 < \min (\lambda_1, \lambda_4)$, then $\lambda_2$ is the principle eigenvalue. In this case, we have
\begin{align*}
    \begin{bmatrix}
        u^* - u(z) \\
        v^* - v(z) \\
        1 - P(z) \\
        -w(z) \\
        -y(z)
    \end{bmatrix}
    \approx B_5 \bm{l}_2 e^{\lambda_2 z},
\end{align*}
as $z \to -\infty$, for some positive constant $B_5$.

\section{Appendix B. Values used in numerical simulations}
\label{app:tables}

\begin{table}[H]
	\begin{center}
		\begin{tabular}{ |c|c|c| } 
 			\hline
 			Term & Value (short runs, long runs) & Description \\ 
 			\hline\hline
 			$tEnd$ & 20,100 & $t$ at end of simulation \\
 			$dt$ & 0.0001 & Temporal step size \\
 			$xMin$ & 0 & $x$ value at left of grid \\
 			$xMax$ & 500,5000 & $x$ value at right of grid \\
 			$dx$ & 0.1,1 & Spatial step size \\
 			$u(x,0)$ & $e^{-5x}$ & Level of $u$ at start of simulation \\
 			$v(x,0)$ & $1$ & Level of $v$ at start of simulation \\
 			$P(x,0)$ & $0$ & Level of $P$ at start of simulation \\
 			\hline
		\end{tabular}
	\end{center}
 	\caption{Numerical parameter values that were used in simulations. \label{table:NumericalVals_NumericalSim}}
\end{table}

\begin{table}[H]
	\begin{center}
		\begin{tabular}{ |c|c| } 
 			\hline
 			Parameter & Value(s) in simulations \\ 
 			\hline\hline
 			$d_1$ & 1 \\
 			$d_2$ & 1 \\
 			$\alpha$ & 1 \\
 			$\omega$ & 100 \\
 			$\theta$ & 1 \\
 			$\Gamma$ & 500 \\
 			$\beta$ & 20 \\
 			$k_0$ & [-5,5] \\
 			$k_1$ & [-10,10] \\
 			$k_2$ & [-10,10] \\
 			\hline
		\end{tabular}
	\end{center}
 	\caption{Parameter values from system in system \ref{Wang_Equation3} that were used in simulations. \label{table:ModelVals_NumericalSim}}
\end{table}

\bibliographystyle{abbrv}
\setlength{\bibhang}{0pt}
\bibliography{Wang_Bibliography}

\begin{comment}

\bibliographystyle{plainurl}
\bibliography{references}

\end{comment}

\end{document}